\documentclass[11pt, psamsfonts, reqno]{amsart}

\usepackage[T1]{fontenc}
\usepackage[latin1]{inputenc}
\usepackage[frenchb, english]{babel}
\usepackage{amsthm}
\usepackage{amssymb,amsmath}
\usepackage{enumitem}
\usepackage[all]{xy}
\usepackage{mathrsfs}
\usepackage[left=0.0cm,right=0.0cm,top=1.7cm,bottom=1.7cm]{geometry}
\usepackage{hyperref}

\theoremstyle{plain}
\newtheorem{theorem}{Théorème}[section]

\newtheorem{lemme}[theorem]{Lemme}
\newtheorem{proposition}[theorem]{Proposition}
\newtheorem{prop}[theorem]{Proposition}

\newtheorem{cor}[theorem]{Corollaire}

\theoremstyle{definition}
\newtheorem{definition}{Definition}

\newtheorem{conjecture}[theorem]{Conjecture}

\theoremstyle{remark}
\newtheorem{remarque}[theorem]{Remarque}

\def\Q{{\bf Q}}
\def\Z{{\bf Z}}
\def\C{{\bf C}}
\def\N{{\bf N}}

\def\F{{\bf F}}
\def\PP{{\bf P}}

\def\H{{H}}

\def\zp{{\Z_p}}
\def\zpe{{\mathbf{Z}_p^\times}}
\def\qpe{{\mathbf{Q}_p^\times}}

\def\cp{{\C_p}}
\def\qp{{\Q_p}}

\def\D{{\bf D}}

\def\Nrig{{\mathbf{N}_{\mathrm{rig}}}}
\def\DdR{{\D_{\mathrm{dR}}}}
\def\Dcris{{\D_{\mathrm{cris}}}}
\def\Ddif{{\D_{\mathrm{dif}}}}
\def\Ddifp{{\D_{\mathrm{dif}}^+}}

\def\Robba{\mathscr{R}}
\def\E{\mathscr{E}}

\def\BdR{{\mathbf{B}_{\mathrm{dR}}}}

\def\j{j}
\def\epsilon{\varepsilon}
\def\det{\mathrm{det}}

\def\matrice#1#2#3#4{{\big(\begin{smallmatrix}#1&#2\\ #3&#4\end{smallmatrix}\big)}}

\title{La conjecture $\epsilon$ locale de Kato en dimension $2$}
\author{Joaqu\'in Rodrigues Jacinto}

\makeatletter
\def\@tocline#1#2#3#4#5#6#7{\relax
  \ifnum #1>\c@tocdepth % then omit
  \else
    \par \addpenalty\@secpenalty\addvspace{#2}%
    \begingroup \hyphenpenalty\@M
    \@ifempty{#4}{%
      \@tempdima\csname r@tocindent\number#1\endcsname\relax
    }{%
      \@tempdima#4\relax
    }%
    \parindent\z@ \leftskip#3\relax \advance\leftskip\@tempdima\relax
    \rightskip\@pnumwidth plus4em \parfillskip-\@pnumwidth
    #5\leavevmode\hskip-\@tempdima
      \ifcase #1
       \or\or \hskip 1em \or \hskip 2em \else \hskip 3em \fi%
      #6\nobreak\relax
    \dotfill\hbox to\@pnumwidth{\@tocpagenum{#7}}\par
    \nobreak
    \endgroup
  \fi}
\makeatother

\usepackage{hyperref}

\usepackage{mathtools}

\makeatletter
\DeclareRobustCommand\widecheck[1]{{\mathpalette\@widecheck{#1}}}
\def\@widecheck#1#2{%
    \setbox\z@\hbox{\m@th$#1#2$}%
    \setbox\tw@\hbox{\m@th$#1%
       \widehat{%
          \vrule\@width\z@\@height\ht\z@
          \vrule\@height\z@\@width\wd\z@}$}%
    \dp\tw@-\ht\z@
    \@tempdima\ht\z@ \advance\@tempdima2\ht\tw@ \divide\@tempdima\thr@@
    \setbox\tw@\hbox{%
       \raise\@tempdima\hbox{\scalebox{1}[-1]{\lower\@tempdima\box
\tw@}}}%
    {\ooalign{\box\tw@ \cr \box\z@}}}
\makeatother

\def\check{\widecheck}

\begin{document}

\maketitle

%\begin{resume} résumé en français \end{resume}
%   \english
%\selectlanguage{francais}

\begin{abstract}
On démontre une équation fonctionnelle dans la théorie d'Iwasawa d'une représentation $p$-adique du groupe de Galois absolu de $\qp$ de dimension $2$. Ceci nous permet de compléter la preuve de Nakamura de la conjecture $\epsilon$ locale de Kato en dimension $2$.
\end{abstract}

\selectlanguage{english}
\begin{abstract} 
We show an Iwasawa functional equation for a two dimensional $p$-adic representation of the absolute Galois group of $\qp$. This allows us to complete Nakamura's proof of Kato's local $\epsilon$-conjecture in dimension $2$.
\end{abstract}
  % \french

\selectlanguage{french}
\setcounter{tocdepth}{2}
\tableofcontents

%\newpage
\section*{Introduction}

Soit $p$ un nombre premier. La conjecture $\epsilon$ locale de Kato \cite{Kato2b}, \cite{FukayaKato} \cite{Nakamura2} prédit l'existence d'une trivialisation canonique du déterminant de la cohomologie des représentations galoisiennes à coefficients dans un anneau $p$-adiquement complet $A$, interpolant des trivialisations standard (qui font intervenir les facteurs epsilon de la représentation, d'où le nom de la conjecture) quand $A$ est une extension finie de $\qp$ et la représentation est de Rham. Elle peut donc être vue comme une interpolation $p$-adique des facteurs locaux des représentations de de Rham.

Ce texte a pour objectif de montrer une équation fonctionnelle dans la théorie d'Iwasawa d'un $(\varphi, \Gamma)$-module sur l'anneau de Robba de rang $2$ et de Rham, ainsi que voir comment elle permet de compléter les résultats de Nakamura \cite{Nakamura2} pour démontrer la conjecture $\epsilon$-locale de Kato pour le cas de dimension $2$ pour une certaine classe d'anneaux $A$, en montrant que l'isomorphisme $\epsilon$ construit dans \cite{Nakamura2} interpole l'isomorphisme $\epsilon$ de de Rham pour un tel $(\varphi, \Gamma)$-module (dans \cite{Nakamura2}, ceci a été montré quand les poids de Hodge-Tate du $(\varphi, \Gamma)$-module sont $k_1 \leq 0$ et $k_2 \geq 1$).

\subsection{Une équation fonctionnelle locale}

Voici une description un peu plus détaillée des résultats.

\subsubsection{Rappels et notations}

Soient $\qp$ le corps des nombres $p$-adiques et $\zp$ l'anneau des entiers $p$-adiques. Fixons une clôture algébrique $\overline{\Q}_p$ de $\qp$ et notons $\mathscr{G}_\qp = \mathrm{Gal}(\overline{\Q}_p / \Q)$ le groupe de Galois absolu de $\qp$. Notons $\chi \colon \mathscr{G}_\qp \to \zpe$ le caractère cyclotomique, défini par l'identité $\sigma(\zeta) = \zeta^{\chi(\sigma)}$, où $\sigma \in \mathscr{G}_{\qp}$ et $\zeta \in \mu_{p^\infty}$ est une racine de l'unité d'ordre une puissance de $p$. On note $\mathscr{H} = \ker(\chi) = \mathrm{Gal}(\overline{\Q}_p / \qp(\mu_{p^\infty}))$ et $\Gamma = \mathscr{G}_\qp / \mathscr{H} = \mathrm{Gal}(\qp(\mu_{p^\infty}) / \qp) \cong \zpe$, le dernier isomorphisme étant donné par $\chi$. On fixe $(\zeta_{p^n})_{n \in \N}$ un système de racines $p^n$-ièmes de l'unité tel que $\zeta_{p} \neq 1$, $\zeta_{p^{n+1}}^p = \zeta_{p^n}$, $n \geq 1$

Soit $L$ une extension finie de $\qp$ \footnote{Le corps $L$ jouera le rôle du corps de coefficients. Il sera fixe mais l'on se permettra éventuellement, si besoin, de remplacer $L$ par une extension finie de lui même. } et soit $V \in \mathrm{Rep}_L \mathscr{G}_\qp$ une $L$-représentation continue du groupe $\mathscr{G}_\qp$, de dimension $2$, de Rham, non trianguline et à poids de Hodge-Tate $0$ et $k \geq 0$. Soit $\check{V} = V^*(1)$ le dual de Tate de $V$, qui est isomorphe (car $V$ est de dimension $2$) à $V \otimes \omega_V^\vee$, où $\omega_V = \det_V \otimes \chi^{-1}$. Notons $\DdR(V) = (\BdR \otimes V)^{\mathscr{G}_\qp}$ le module de de Rham de la représentation $V$ et considérons \[ \exp \colon \DdR(V) \to H^1(\qp, V), \] \[ \exp^* \colon H^1(\qp, V) \to \DdR(V) \] les applications exponentielle et exponentielle duale de Bloch-Kato.

On aura besoin de comparer des éléments habitant dans le module de de Rham des différents tordus de $V$ et de son dual de Tate et les identifications suivantes permettront de voir tous ces éléments dans $\DdR(V)$. Si $\delta \colon \Q^\times_p \to L^\times$ est un caractère continu unitaire, on note $L(\delta)$ la représentation de dimension $1$ où $\mathscr{G}_\qp$ agit à travers $\delta$ vu comme caractère de $\mathscr{G}_\qp$ via la théorie du corps de classes locale \footnote{Explicitement, si $g \in \mathscr{G}_\qp$ se réduit modulo $p$ à la puissance $\mathrm{deg}(g) \in \Z$ du Frobenius absolu, alors $\delta(g) = \delta(p)^{- \mathrm{deg}(g)} \delta(\chi(g))$.}, dont on fixe $e_\delta$ une base. Si $\eta \colon \zpe \to L^\times$ est un caractère constant modulo $p^n$, vu comme un caractère de $\qpe$ en posant $\eta(p) = 1$, l'élément $\mathbf{e}^{\rm dR}_\eta = G(\eta) e_\eta$ est un générateur du module $\DdR(L(\eta))$, où $G(\eta) := \sum_{a \in (\Z / p^n \Z)^\times} \eta(a) \zeta_{p^n}^a$ dénote la somme de Gauss de $\eta$. Notons $e_\eta^\vee = e_{\eta^{-1}}$ la base de $L(\eta)^*$ duale de $e_\eta$, d'où $\DdR(L(\eta)^*) =  L \cdot G(\eta)^{-1} e_\eta^\vee = L \cdot \mathbf{e}^{\rm dR, \vee}_{\eta}$. Si $j \in \Z$, on note aussi $\mathbf{e}^{\rm dR}_j = t^{-j} e_j = t^{-j} e_{\chi^j} \in \BdR \otimes L(\chi^j)$, qui est une base de $\DdR(L(\chi^j))$.
%Notons \[ \mathbf{e}^{\rm dR}_{\eta, j} =  \mathbf{e}^{\rm dR}_\eta \otimes \mathbf{e}^{\rm dR}_j = G(\eta) e_\eta \otimes t^{-j} e_j \] qui est une base du module $\DdR(L(\eta \chi^j))$ et, pour $e_\eta^\vee = e_{\eta^{-1}}$ et $e_{-j}$ les éléments duaux de $e_\eta$ et $e_j$, on note \[ \mathbf{e}^{\rm dR, \vee}_{\eta, j} = G(\eta)^{-1} e_\eta^\vee \otimes t^j e_{-j}, \] qui est une base du module $\DdR(L(\eta \chi^j)^*)$.

Fixons une base $f_1, f_2$ de $\DdR(V)$ et notons $$ \langle \;,\; \rangle_{\rm dR} \colon \DdR(V) \times \DdR(V) \to L, $$ le produit scalaire défini par la formule $\langle a_1 f_1 + a_2 f_2, b_1 f_1 + b_2 f_2 \rangle_{\rm dR} = a_1 b_1 + a_2 b_2$. L'isomorphisme $\wedge^2 V = L(\chi) \otimes \omega_V$ induit un isomorphisme $\wedge^2 \DdR(V) = \DdR(L(\chi) \otimes \omega_V) = (t^{-k} L_\infty \cdot e_V)^\Gamma$, où $L_\infty = L(\mu_{p^\infty})$ et $e_V$ dénote une base de $L(\det_V)$. On définit $\Omega \in L_\infty$ par la formule $f_1 \wedge f_2 = (t^k \Omega)^{-1} e_V$, ce qui nous permet de fixer les bases $(t^k \Omega)^{-1} e_V$ et $t^k \Omega e_V^\vee$ du module $\DdR(\wedge^2 V)$ et de son dual. On fixe aussi les bases $\mathbf{e}^{\rm dR}_{\omega_V} = (t^{k-1} \Omega)^{-1} e_{\omega_V}$ et $\mathbf{e}^{\rm dR, \vee}_{\omega_V} = (t^{k-1} \Omega) e_{\omega_V}^\vee$ du module $\DdR(L(\omega_V))$ et de son dual.
%Avec les notations introduites ci-dessus, on a des isomorphismes $$ \DdR(D) \xrightarrow{\sim} \DdR(\check{D}(\eta \chi^{-{\j}})); \;\;\; x \mapsto x \otimes G(\eta) e_{\eta} \otimes t^{{\j}} e_{-{\j}} \otimes  \Omega t^{k-1}  e_{\omega_D}^\vee $$ $$ \DdR(D) \xrightarrow{\sim} \DdR(D(\eta^{-1} \chi^{\j})) ; \;\;\; x \mapsto x \otimes G(\eta^{-1}) e_{\eta^{-1}} \otimes t^{-{\j}} e_{\j}. $$ 

Enfin, notons, pour $\eta \colon \zpe \to L^\times$ et $j \in \Z$ comme ci-dessus,$$ \mathbf{e}^{\rm dR, \vee}_{\eta, j, \omega_V^\vee} = \mathbf{e}^{\rm dR, \vee}_\eta \otimes \mathbf{e}^{\rm dR}_{-j}  \otimes \mathbf{e}^{\rm dR}_{\omega_V}, $$ $$ \mathbf{e}^{\rm dR, \vee}_{\eta, j} = \mathbf{e}^{\rm dR, \vee}_\eta \otimes \mathbf{e}^{\rm dR}_{-j}, $$ des bases des duaux des modules $\DdR(L(\eta \chi^j \omega_D^{-1}))$ et $\DdR(L(\eta \chi^j))$. L'application $x \mapsto x \otimes \mathbf{e}^{\rm dR, \vee}_{\eta, -j, \omega_V^\vee}$ induit donc un isomorphisme $\DdR(\check{V}(\eta \chi^{-{\j}})) \xrightarrow{\sim} \DdR(\check{V}) $ et, de même, $x \mapsto x \otimes \mathbf{e}^{\rm dR, \vee}_{\eta^{-1}, j}$ induit isomorphisme $\DdR(V(\eta^{-1} \chi^j)) \xrightarrow{\sim} \DdR(V). $

Soit \[ H^1_{\rm Iw}(\qp, V) = \varprojlim_n H^1(\qp(\mu_{p^n}), T) \otimes_\zp \qp \] la cohomologie d'Iwasawa de la représentation $V$, où $T$ dénote n'importe quel $\zp$-réseau de $V$ stable par $\mathscr{G}_\qp$ et où la limite est prise par rapport aux applications de corestriction. On dispose des applications de spécialisation
\[ H^1_{\rm Iw}(\qp, V) \to H^1(\qp, V(\eta \chi^j)) : \mu \mapsto \int_\Gamma \eta \chi^j \cdot \mu. \]

La correspondance de Langlands $p$-adique \cite{ColmezPhiGamma} nous permet de construire une application \[ w_V \colon H^1_{\rm Iw}(\qp, V) \to H^1_{\rm Iw}(\qp, \check{V}). \] Cette application est définie à partir de l'action de la matrice $\matrice 0 1 1 0$ sur la représentation de Banach de $\mathrm{GL}_2(\qp)$ associée à $V$ par la correspondance de Langlands $p$-adique pour $\mathrm{GL}_2(\qp)$ et, au même temps, aux facteurs $\epsilon$ d'une représentation lisse admissible de $\mathrm{GL}_2(\qp)$ via la théorie du modèle de Kirillov, d'où son importance.

Notons finalement $\pi$ la représentation lisse de $\mathrm{GL}_2(\qp)$ associée à $V$ par la correspondance de Langlands classique, et $\epsilon(-)$ le facteur epsilon d'une représentation lisse de $\mathrm{GL}_2(\qp)$.

\subsubsection{Le théorème principal}

Le théorème suivant, dont la preuve est inspirée largement des techniques introduites par Nakamura \cite{Nakamura2}, Dospinescu et Colmez \cite{ColmezPoids}, décrit le comportement de l'involution de Colmez en termes de la théorie d'Iwasawa de la représentation $V$.

\begin{theorem} [Th. \ref{eqfonct1b}] \label{eaeaeqfonctintro}
Soit $\mu \in H^1_{\rm Iw}(\qp, V)$ et notons $\check{\mu} = w_V(\mu) \in H^1_{\rm Iw}(\qp, \check{V})$. Alors, pour tout $j \geq 1$, on a \[ \exp^*(\int_\Gamma \eta \chi^{-j} \cdot \check{\mu}) \otimes \mathbf{e}^{\rm dR, \vee}_{\eta, -j, \omega_V^\vee} = C(V, \eta, j) \cdot \exp^{-1}(\int_\Gamma \eta^{-1} \chi^{j} \cdot \mu) \otimes \mathbf{e}^{\rm dR, \vee}_{\eta^{-1}, j}, \] où \[ C(V, \eta, j) = \Omega^{-1} \frac{(-1)^{j}}{(j + k - 1)! (j - 1)!} \, \epsilon(\eta^{-1})^{-2}  \epsilon(\pi \otimes \eta^{-1} \otimes |\cdot|^j). \]
\end{theorem}

\begin{remarque}
D'après la philosophie de Perrin-Riou et Bloch-Kato, si $\mu$ provient par restriction d'un système d'Euler d'une représentation globale, les valeurs que l'on compare dans le théorème \ref{eaeaeqfonctintro} ci-dessus sont étroitement liées aux valeurs spéciales de la fonction $L$ $p$-adique associée à cette représentation, et le théorème \ref{eaeaeqfonctintro} ci-dessus implique une équation fonctionnelle pour cette fonction $L$ $p$-adique \cite{FonctLpadiques}.
\end{remarque}

La démonstration du théorème \ref{eaeaeqfonctintro} est un long processus de traduction et occupe le deuxième chapitre de ce texte. L'idée principale, comme dans \cite{Nakamura2}, est d'exprimer les deux termes du théorème \ref{eaeaeqfonctintro} en termes de la correspondance de Langlands $p$-adique et d'utiliser la compatibilité locale-globale entre la correspondance de Langlands $p$-adique et classique, ce qui fait naturellement apparaître le facteur epsilon de la représentation lisse de $\mathrm{GL}_2(\qp)$ associée à $V$ quand on applique l'involution. De fait, nous utilisons les techniques de changement de poids de \cite{ColmezPoids} pour ramener aux calculs faits dans \cite{Nakamura2}.

Nous allons decrire les etapes principales du calcul. Soit $D = \D(V) \in \Phi \Gamma^{\text{ét}}(\Robba)$ le $(\varphi, \Gamma)$-module sur l'anneau de Robba \footnote{On renvoie le lecteur non familiarisé aux sections correspondantes du texte pour les définitions des objets introduits sans définition dans cette introduction.} associé à la représentation $V$ par les équivalences de catégories de Fontaine, Cherbonnier-Colmez et Kedlaya (prop. \ref{equivPhiGamma}). On dispose aujourd'hui de toute une cosmogonie d'objets et applications permettant d'exprimer un grand nombre d'invariants arithmétiques de la représentation $V$ intrinsèquement en termes de $D$.

\subsubsection{L'équation différentielle $p$-adique et le cas de poids nuls}

Soit $\Delta = \Nrig(D) \in \Phi \Gamma(\Robba)$ le module construit (cf. \S \ref{eqdiff}) par Berger. Le module $\Delta$ n'est pas étale et il est de rang $2$, de Rham à poids de Hodge-Tate nuls. On a sur $\Delta$ un opérateur de dérivation $\partial$ au dessus de l'opérateur $\partial = (1 + T) \frac{d}{dT}$ sur $\Robba$. Si $D$ est à poids de Hodge-Tate positifs, alors $D \subseteq \Delta$ et l'on peut donc voir l'élément $z \in D^{\psi = 1}$ dans $\Delta^{\psi = 1}$.

Rappelons (cf. \S \ref{generalites}) qu'il existe un entier $m(\Delta) \geq 0$ et, pour tout $n \geq m(\Delta)$, des sous-$\Robba^{]0, r_n]}$-modules $\Delta^{]0, r_n]}$ libres de rang $2$ tels que $\Delta = \Robba \otimes_{\Robba^{]0, r_n]}} \Delta^{]0, r_n]}$, où $r_n = \frac{1}{p^{n - 1} (p - 1)}$ et qu'on a des applications de localisation \[ \varphi^{-n} \colon \Delta^{]0, r_n]} \to L_n((t)) \otimes \DdR(\Delta), \] où $t$ dénote le ``$2i\pi$'' de Fontaine et $L_n = L \cdot \qp(\mu_{p^n})$. L'opérateur $\partial$ sur $\Delta$ satisfait la relation de commutation $\varphi^{-n} \circ \partial = \frac{d}{dt} \circ \varphi^{-n}$.

\subsubsection{La correspondance de Langlands $p$-adique pour $\mathrm{GL}_2(\qp)$}

Dans \cite{ColmezPoids}, Colmez construit (cf. \S \ref{LLpadique}) une représentation localement analytique admissible $\Pi(\Delta)$ à caractère centrale $\omega_\Delta := \det_\Delta \chi^{-1}$ de $G = \mathrm{GL}_2(\qp)$ et il explique comment retrouver la représentation $\Pi(D)$ associée à $D$ par la correspondance de Langlands $p$-adique par des méthodes de `changement de poids'. La représentation $\Pi(\Delta)$ est irréductible mais elle est munie d'un opérateur $\partial$ et, en tordant l'action de $G$ d'une manière appropriée en utilisant $\partial$, l'on peut récupérer les vecteurs localement algébriques de $\Pi(D)$ et donc en particulier la représentation lisse $\pi$ de $G$ associée à $V$ par la correspondance de Langlands classique.

Un $(\varphi, \Gamma)$-module peut être vu comme un faisceau ${\matrice {\zp - \{0\}} {\zp} 0 1}$-équivariant \footnote{L'action de $\varphi$, $\sigma_a$ et la multiplication par $(1 + T)^b$, $a \in \zpe, b \in \zp$, correspondant à ${\matrice p 0 0 1}$, ${\matrice a 0 0 1}$ et ${\matrice 1 b 0 1}$ respectivement.} $U \mapsto \Delta \boxtimes U$ sur $\zp$ dont les sections globales sont données par $\Delta \boxtimes \zp = \Delta$ et la construction de $\Pi(\Delta)$ est fondée sur l'extension de ce faisceaux en un faisceau $G$-équivariant $U \mapsto \Delta \boxtimes U$ sur $\PP^1 = \PP^1(\qp)$. On a un accouplement parfait et $G$-équivariant $[\; , \; ]_{\PP^1}$ sur $\Delta \boxtimes \PP^1$ et la suite exacte fondamentale de $G$-modules suivante:
\[ 0 \to \Pi(\Delta)^* \otimes \omega_\Delta \to \Delta \boxtimes \PP^1 \to \Pi(\Delta) \to 0. \]

Dans \S \ref{psiinvsec}, on étend des résultats classiques et on démontre que, pour $\alpha = 1$ ou $\alpha = \omega_\Delta(p)$, l'on a un isomorphisme
\[ (\Pi(\Delta)^* \otimes \omega_\Delta)^{{\matrice p 0 0 1} = \alpha} \xrightarrow{\sim} \Delta^{\psi = \alpha^{-1}}. \] Si $z \in \Delta^{\psi = 1}$, on note $\tilde{z}$ l'image inverse de $z$ par cet isomorphisme. En notant $w = {\matrice 0 1 1 0}$, l'élément $w \cdot \tilde{z}$ appartient donc à $(\Pi(\Delta)^* \otimes \omega_\Delta)^{{\matrice p 0 0 1} = \omega_\Delta(p)}$ et l'on note $\check{z} \in \Delta^{\psi = \omega_\Delta(p)^{-1}} = \check{\Delta}^{\psi = 1}$, où $\check{\Delta} = \Delta^*(1)$ dénote le dual de Tate de $\Delta$, qui s'identifie (car $\Delta$ est de dimension $2$) à $\Delta \otimes \omega_\Delta^{-1}$.

On oublie pour l'instant notre $(\varphi, \Gamma)$-module $D$ (et en particulier que $k$ dénote un de ses poids de Hodge-Tate) et on va montrer une équation fonctionnelle sur $\Delta$.

\subsubsection{Cohomologie d'Iwasawa}

Notons (cf. \S \ref{cohomIw}) $H^i_{\rm Iw}(\Delta)$ le premier groupe de cohomologie d'Iwasawa défini dans \cite{Pottharst}. On a un isomorphisme
\[ \mathrm{Exp}^* \colon H^1_{\rm Iw}(\Delta) \xrightarrow{\sim} \Delta^{\psi = 1}. \]
Notons $\mu_z = (\mathrm{Exp}^*)^{-1} (z) \in \Delta^{\psi = 1}$ et $\mu_{\check{z}} = (\mathrm{Exp}^*)^{-1} (\check{z}) \in \check{\Delta}^{\psi = 1} = \Delta^{\psi = \omega_\Delta(p)^{-1}}$.
%et $\check{z} = \mathrm{Exp}^* (\mu) \in \check{D}^{\psi = 1}$, où $\check{D} = D^*(1) = \D(\check{V})$ dénote le dual de Tate de $D$, qui s'identifie (car $\Delta$ est de dimension $2$) à $D \otimes \omega_D^{-1}$.

\subsubsection{Lois de réciprocité}

Les lois de réciprocité (cf. \S \ref{loisrec}) nous permettent d'écrire les valeurs $\exp^*(\int_\Gamma \eta \chi^{-j} \cdot \check{\mu}) \otimes \mathbf{e}^{\rm dR, \vee}_{\eta, -j, \omega_\Delta^\vee}$ et $\exp^{-1}(\int_\Gamma \eta^{-1} \chi^{j} \cdot \mu) \otimes \mathbf{e}^{\rm dR, \vee}_{\eta^{-1}, j}$ en termes de $z$ et des applications de localisation. On commence par démontrer (lemme \ref{expd1}) que, pour $m \gg 0$ et $j, k \in \Z$ tels que $j + k \geq 1$, on a \footnote{Notons que les deux variables $j$ et $k$ semblent être redondantes, mais elles seront utiles plus tard.}
\begin{equation} \label{eq1intro} \exp^*(\int_\Gamma \eta \chi^{-j - k} \cdot \mu_{\check{z}}) \otimes \mathbf{e}^{\rm dR, \vee}_{\eta, -j - k, \omega_\Delta^\vee} = * \cdot \mathrm{Tr}_{L_n / L} ( G(\eta)^{-1} \Omega^{-1} [ \varphi^{-n} \partial^{j + k - 1} \check{z} ]_0).
\end{equation}
%\[ \exp^* (\int_\Gamma \eta \chi^{-j} \cdot \mu) = p^{-m} \, \mathrm{Tr}_{L_m / L}([\varphi^{-m} (z \otimes e_\eta \otimes e_{-j})]_0) \in \DdR(V(\eta \chi^{-j})) = \DdR(V) \otimes \mathbf{e}^{\rm dR}_{\eta, j}, \]
où $*$ est une constante explicite, $\mathrm{Tr}_{L_m / L} \colon L_n \to L$ dénote la trace et $[ \cdot ]_0 \colon L_n((t)) \to L_n$ dénote le coefficient du terme de degré $0$ en $t$. On dispose d'une formule analogue pour les valeurs de l'exponentielle (lemma \ref{exp1}): pour $j \geq 1$, on a
\begin{equation} \label{eq1bintro} \exp^{-1}(\int_\Gamma \eta \chi^{\j} \cdot \mu_z) \otimes \mathbf{e}^{\rm dR, \vee}_{\eta, j} = * \cdot  \mathrm{Tr}_{L_n / L} ( G(\eta)^{-1} [\varphi^{-n} \partial^{-{\j}} z]_0).
\end{equation}

\subsubsection{Le modèle de Kirillov}

L'étape suivante consiste à exprimer les termes de droite des équations \eqref{eq1intro} et \eqref{eq1bintro} en termes de la correspondance de Langlands $p$-adique. La théorie du modèle de Kirillov (\S \ref{Kirillov}) et une formule de Colmez (cf. \S \ref{dualKir}, eq. \eqref{formulemagique}) nous permettent (cf. \S \ref{dualKir}) de construire une injection $B$-équivariante
\[ \mathrm{LA}_{\rm c}(\qpe, \Delta_{\rm dif}^-)^\Gamma \to \Pi(\Delta), \] où $B = {\matrice * * 0 *} \subseteq G$ dénote le Borel de $G$, $\Delta_{\rm dif}^- = (L((t)) / L[[t]]) \otimes \DdR(\Delta)$ et $\mathrm{LA}_{\rm c}(\qpe, \Delta_{\rm dif}^-)^\Gamma$ dénote l'espace des fonctions localement analytiques $\phi \colon \qpe \to \Delta_{\rm dif}^-$ à support compact vérifiant $\sigma_a(\phi(x)) = \phi(x)$, et il est muni d'une action de $B$ définie par la formule \footnote{Si $r \in \qp$ et $n \in \N$ est tel que $rp^n \in \zp$, on pose $[\epsilon^{r}] = \varphi^{-n}((1 + T)^{p^n r}) = \zeta_{p^n}^{p^n r} e^{tr} \in L_n[[t]]$.} $$ ({\matrice a b 0 d} \cdot \phi)(x) = \omega_\Delta(d) [\epsilon^{bx / d}] \phi(ax/d). $$

En suivant des idées de Nakamura, on construit (cf. \S \ref{accexpd} et \S \ref{accexp}), pour $\eta \colon \zpe \to L^\times$ un caractère d'ordre fini, $k \geq 1$, $i \in \{1, 2\}$ et $m \in \Z$, des fonctions $f^i_{\eta, k, m} \in \mathrm{LA}_{\mathrm{c}}(\qpe, \Delta_{\rm dif}^-)^\Gamma$ et $g^i_{\eta, m} \in \mathrm{LA}_{\mathrm{c}}(\qp, \Delta_{\rm dif}^-)^\Gamma$ et on démontre (lemmes \ref{expd2} et \ref{exp2}) que
\begin{equation} \label{eq2intro}
\mathrm{Tr}_{L_n / L} \big(  G(\eta)^{-1} \Omega^{-1} \langle [  \varphi^{-n} \partial^{{\j}+k-1} \check{z}]_0, f_i \rangle_{\rm dR} \big) = * \cdot [ \partial^{\j} w \cdot \tilde{z}, f^{3 - i}_{\eta, k, 0}]_{\PP^1},
\end{equation}
\begin{equation} \label{eq2bintro} \mathrm{Tr}_{L_n / L}( G(\eta)^{-1} \langle [\varphi^{n} \partial^{-j} z]_0, f_i \rangle_{\rm dR} )  = * \cdot [ \partial^{-{\j}} \tilde{z}, g^{3 - i}_{\eta, m}]_{\PP^1}
\end{equation}
où $(f_i)_{i = 1, 2}$ dénote une base fixe de $\DdR(\Delta)$ comme expliqué avant.

\subsubsection{Apparition des vecteurs localement algébriques, facteurs epsilon et fin de la preuve}

Plus précisément, dans \cite{ColmezPoids} (cf. \S \ref{Changement}) on construit, pour tout $k \geq 1$, une représentation $\Pi(\Delta, k)$ en tordant l'action de $G$ sur $\Pi(\Delta)$ à l'aide de l'opérateur $\partial$. En identifiant $\Pi(\Delta, k)$ et $\Pi(\Delta)$ et en notant $\ast_k$ l'action tordue de $G$ sur $\Pi(\Delta)$ on a $w \ast_k v = \partial^k \circ w \cdot v$, $v \in \Pi(\Delta)$, et on voit que $\Pi(\Delta, k)$ possède des vecteurs localement algébriques et que la représentation lisse $\pi$ associée ne dépend pas de $k$. Le miracle arrive quand on observe que les fonctions $f^i_{\eta, k, m}$ et $g^i_{\eta, m}$ sont des vecteurs localement algébriques de $\Pi(\Delta, k)$ qui peuvent être exprimés via des fonctions introduites dans \cite{BH} pour retrouver le facteur epsilon d'une représentation lisse de $G$ via l'action de l'élément $w$, ce qui nous permet (lemme \ref{formkir}) de démontrer que
\begin{equation} \label{eq3intro} w \ast_k f^i_{\eta, k, m} = * \cdot \epsilon(\pi \otimes \eta^{-1}) \, g^i_{\eta^{-1}, -c(\pi \otimes \eta^{-1}) - m},
\end{equation}
où $c(\pi \otimes \eta^{-1})$ dénote le conducteur de la représentation $\pi \otimes \eta^{-1}$.

Enfin, en utilisant la $G$-équivariance de l'accouplement $[\; , \;]_{\PP^1}$ (pour passer l'action de $w$ du côté de droite de l'accouplement dans la formule \eqref{eq2intro} et appliquer l'identité \eqref{eq3intro}) et les équations \eqref{eq1intro}, \eqref{eq1bintro}, \eqref{eq2intro} (avec $m = 0$) et \eqref{eq2bintro} (avec $m = -c(\pi \otimes \eta^{-1})$) ci-dessus, on obtient (théorème \ref{eqfonct1}) une équation fonctionnelle en poids nuls et, en tordant cette équation, on montre facilement (théorème \ref{eqfonct1b}) notre résultat principal.

\subsection{Équation fonctionnelle et conjecture $\epsilon$ locale de Kato}

Donnons une application du théorème précedent à la conjecture $\epsilon$. Des cas particuliers de cette conjecture ont été montrés: elle est connue en dimension $1$ \cite{Kato2b}, \cite{Nakamura3}, et pour certains types de représentations grâce aux travaux de Benois-Berger \cite{BB}, Loeffler-Venjakob-Zerbes \cite{LVZ} et Nakamura \cite{Nakamura3}. Nakamura dans \cite{Nakamura2} a construit, en utilisant la théorie des $(\varphi, \Gamma)$-modules et l'accouplement d'Iwasawa défini par Colmez, un candidat pour l'isomorphisme $\epsilon$ pour les $(\varphi, \Gamma)$-modules de rang $1$ ou $2$, et il a montré que, dans certaines instances, cet isomorphisme interpole bien les trivialisations standards.

% (où le cas particulier de $0 \leq j \leq k -1$ est montré) \footnote{Précisement, si $D$ est le $(\varphi, \Gamma)$-module à poids de Hodge-Tate $k_1 \leq 0$ et $k_2 \geq 1$ dans l'énonce de \cite[Prop. 3.14]{Nakamura2}, alors la proposition suit en appliquant le théorème \ref{eaeaeqfonctintro} à $V(- k_1)$, et en posant $\eta = 1$ et $j = - k_1$, où $V$ est la représentation galoisienne associée à $D$.},
Le théorème \ref{eaeaeqfonctintro} ci-dessus est un complement de \cite[Prop. 3.14]{Nakamura2} \footnote{Dans les notations de ce texte, \cite[Prop. 3.14]{Nakamura2} est équivalent à montrer une équation fonctionnelle analogue à celle du théorème \ref{eaeaeqfonctintro} réliant les valeurs $\exp^*(\int_\Gamma \eta \chi^{-j} \cdot \check{\mu}) \otimes \mathbf{e}^{\rm dR, \vee}_{\eta, -j, \omega_V^\vee}$ et $\exp^*(\int_\Gamma \eta^{-1} \chi^{j} \cdot \mu) \otimes \mathbf{e}^{\rm dR, \vee}_{\eta^{-1}, j}$ pour $1 -k \leq j \leq 0$. On peut facilement adapter les techniques de ce texte pour retrouver l'énoncé de \cite[Prop. 3.14]{Nakamura2}, mais les nouvelles tehniques introduites ici ne ne sont pas nécessaires dans ce cas.}, qui est le point clé pour montrer que les isomorphismes construits dans \cite{Nakamura2} interpolent les isomorphismes $\epsilon$ classiques quand $V$ est de Rham et à poids de Hodge-Tate $k_1 \leq 0$ et $k_2 > 0$. Comme un corollaire de nos résultats et de ceux de Nakamura, on obtient

\begin{theorem} [Th. \ref{conjepsilon}]
La conjecture $\epsilon$ locale de Kato est vraie pour les représentations galoisiennes de dimension $2$.
\end{theorem}

\subsection{Remerciement}

Ce travail fait partie de ma thèse de doctorat, réalisée à l'Institut de Mathématiques de Jussieu sous la direction de Pierre Colmez, à qui j'exprime ma plus sincère gratitude. Je remercie aussi Kentaro Nakamura et Gabriel Dospinescu, dont les travaux ont été une source constante d'inspiration. Je remercie finalement le rapporteur pour ses nombreuses remarques et corrections qui ont aidé à améliorer considérablement la pr\'esentation de cet article.

\section{Notations et préliminaires}

Cette section contient les généralités sur les $(\varphi, \Gamma)$-modules dont on aura besoin dans l'article.

\subsection{Anneaux de Fontaine} \label{anneaux}

Fixons $L$ une extension finie de $\qp$ contenue dans $\overline{\Q}_p$ et notons:
\begin{itemize} 
\item $\mathscr{O}_\E$ l'anneau des séries de Laurent à coefficients dans $\mathscr{O}_L$ bornés dont les coefficients négatifs tendent vers $0$, i.e,  $\mathscr{O}_\E = \mathscr{O}_L[[T]][1/T]^\wedge$, où la complétion est la $p$-adique;
\item $\E = \mathscr{O}_\E[1/p]$ le corps des fractions de $\mathscr{O}_\E$;
\item $\E^\dagger$ l'anneau constitué des éléments de $\E$ convergents sur une couronne d'intérieur non vide;
\item $\Robba^{[r, s]}$, $0 < r < s \leq +\infty$, l'anneau des séries de Laurent à coefficients dans $L$ et convergeant sur la couronne $r \leq v_p(T) \leq s$;
\item $\Robba^{]0, s]} := \varprojlim_{0 < r < s} \Robba^{[r , s]}$;
\item $\Robba = \varinjlim_{s > 0} \Robba^{]0, s]}$ l'anneau de Robba, formé des séries de Laurent (non nécessairement bornées) convergentes sur une couronne non vide contenue dans la boule unité ouverte.
\end{itemize}
On notera souvent $\Robba^+ = \Robba^{]0, +\infty]}$ l'anneau des fonctions analytiques sur la boule unité ouverte. Rappelons que l'on obtient $\E$ (resp. $\Robba$) en complétant $\E^\dagger$ par rapport à sa topologie $p$-adique (resp. la topologie définie par la convergence uniforme sur des couronnes fermées).

\subsection{$(\varphi, \Gamma)$-modules}

Rappelons que, si  $\mathscr{A} \in \{ \mathscr{O}_\E, \E, \E^\dagger, \Robba \}$, un $(\varphi, \Gamma)$-module sur $\mathscr{A}$ est un $\mathscr{A}$-module de type fini muni d'actions semi-linéaires continues de $\Gamma$ et d'un opérateur $\varphi$, commutant entre elles. Si $D$ est un $(\varphi, \Gamma)$-module sur $\mathscr{O}_\E$, on dit que $D$ est étale si $\varphi(D)$ engendre $D$ comme $\mathscr{O}_\E$-module. Un $(\varphi, \Gamma)$-module sur $\E$ est étale s'il possède un $\mathscr{O}_\E$-réseau stable par $\Gamma$ et $\varphi$ qui est étale. Un $(\varphi, \Gamma)$-module $D$ sur $\E^\dagger$ est étale si $D \otimes_{\E^\dagger} \E$ est étale. Finalement, un $(\varphi, \Gamma)$-module $D$ sur $\Robba$ est dit étale s'il contient un $(\varphi, \Gamma)$-module $D^\dagger$ étale sur $\E^\dagger$ tel que $D = D^\dagger \otimes_{\E^\dagger} \Robba$.

Notons, pour $\mathscr{A} \in \{ \mathscr{O}_\E, \E, \E^\dagger, \Robba \}$, $\Phi \Gamma(\mathscr{A})$ (resp. $\Phi \Gamma^{\text{ét}}(\mathscr{A})$) la catégorie des $(\varphi, \Gamma)$-modules (resp. étales) sur $\mathscr{A}$. L'intérêt de ces objets réside dans le résultat classique fondamental suivant:

\begin{prop} [\cite{Fontaine90}, \cite{CC}, \cite{Kedlaya}] \label{equivPhiGamma}
Il existe des équivalences de catégories $\mathrm{Rep}_L \mathscr{G}_\qp \cong \Phi \Gamma^{\text{ét}}(\E) \cong \Phi \Gamma^{\text{ét}}(\E^\dagger) \cong \Phi \Gamma^{\text{ét}}(\Robba)$.
\end{prop}

Si $D \in \Phi \Gamma(\mathscr{A})$ et $\delta \colon \qpe \to L^\times$ est un caractère, on note $e_\delta$ une base du module $L(\delta)$ muni d'actions de $\varphi$ et $\Gamma$ via les formules $\varphi(e_\delta) = \delta(p) e_\delta$ et $\sigma_a(e_\delta) = \delta(a) e_\delta$, $a \in \zpe$. On note $D(\delta) = D \otimes \delta$ le module $D \otimes_L L(\delta)$. Le choix de $e_\delta$ fournit un isomorphisme de $L$-espaces vectoriels $D \xrightarrow{\sim} D(\delta): x \mapsto x \otimes e_\delta$.

\subsection{Sous-modules naturels de $D$} \label{generalites}

Soit $D \in \Phi\Gamma(\Robba)$ de rang $d$. L'algèbre de Lie de $\Gamma$ agit sur $D$ (cf. \cite[\S 5.1]{Berger02}) via l'opérateur $L$-linéaire $$ \nabla = \lim_{a \to 1} \frac{\sigma_a - 1}{a - 1} = \frac{\log (\gamma)}{\log (\chi(\gamma))} = \frac{1}{\log (\chi(\gamma))} \sum_{i = 1}^{+\infty} (-1)^{i + 1} \frac{(\gamma - 1)^i}{i}, $$ où $\gamma \in \Gamma$ dénote un élément assez proche de $1$, ce qui définit un opérateur différentiel au-dessus de l'opérateur $\nabla = t (1 + T) \frac{d}{dT}$, où $t = \log(1 + T)$ dénote le $2 i \pi$ de Fontaine, agissant sur $\Robba$. 
%En utilisant les morphismes de localisation, on a l'égalité $\varphi^{-n} \circ \nabla = t \frac{d}{dt} \circ \varphi^{-n}$ dans $F_n[[t]]$.

Rappelons que, d'après \cite[Th. I.3.3]{Berger08}, il existe $m(D) \geq 0$ et, pour tout $s \leq r_{m(D)} := \frac{1}{(p - 1) p^{m(D) - 1}}$, des uniques sous-$\E^{]0, s]}$-modules $D^{]0, s]}$ de $D$ tels que $D = \Robba \otimes_{\Robba^{]0, s]}} D^{]0, s]}$ et $\varphi(D^{]0, s]}) \subseteq D^{]0, s/p]}$ induit un isomorphisme $\Robba^{]0, s/p]} \otimes_{\varphi, \Robba^{]0, s]}} D^{]0, s]} \to D^{]0, s/p]}: f \otimes x \mapsto f \varphi(x)$. Pour $0 < r < s \leq r_{m(D)}$, on pose \[ D^{[r, s]} = \E^{[r, s]} \otimes_{\E^{]0, s]}} D^{]0, s]}. \] On a alors \[ D = \varinjlim_{s > 0} \varprojlim_{0 < r < s} D^{[r, s]}, \] ce qui montre que $D$ est un espace de type $\mathrm{LF}$ (i.e limite inductive d'espaces de Fréchet). 

Rappelons que l'on a des morphismes de localisation \[ \varphi^{-n} \colon \Robba^{]0, r_n]} \hookrightarrow L_n[[t]] \] envoyant $T$ sur $\zeta_{p^n} e^{t / p^n} - 1$. Pour $n \geq m(D)$ on définit \[ \D^+_{\mathrm{dif}, n}(D) = L_n[[t]] \otimes_{\varphi^{-n}, \Robba^{]0, r_n]}} D^{]0, r_n]}, \] \[ \D_{\mathrm{dif}, n}(D) = L_n((t)) \otimes_{L_n[[t]]} \D^+_{\mathrm{dif}, n}(D), \] qui sont des $L_n[[t]]$ et $L_n((t))$-modules, respectivement, libres de rang $d$ et munis d'une action semi-linéaire de $\Gamma$. Finalement, on définit $$ \D_{\mathrm{dif}}(D) = \varinjlim_n \D_{\mathrm{dif}, n}(D) \; \;\; \D^+_{\mathrm{dif}}(D) = \varinjlim_n \D^+_{\mathrm{dif}, n}(D), $$ qui sont, respectivement, des $L_\infty((t)) = \cup_n L_n((t))$ et $L_\infty[[t]] = \cup_n L_n[[t]]$-modules libres de rang  $d$.

\subsection{Théorie de Hodge $p$-adique}

Soit $D \in \Phi\Gamma(\Robba)$ de rang $d$. On définit (cf. \cite[Def. 2.5]{Nakamura}) $$ \Dcris(D) = (D[1 / t])^\Gamma = (D \otimes_\Robba \Robba[1/t])^\Gamma \;\;\;\; \DdR(D) = (\D_{\mathrm{dif}}(D))^\Gamma, $$ qui sont des $L$-espaces vectoriels de dimension finie. On munit $\DdR(D)$ de sa filtration de Hodge, donnée par $ \mathrm{Fil}^i \, \DdR(D) = \DdR(D) \cap t^i \D_{\mathrm{dif}}^+ = (t^i \D_{\mathrm{dif}}^+)^\Gamma$. On observe que $\Dcris(D)$ est muni d'une action bijective du Frobenius $\varphi$ ainsi que d'une filtration induite par l'inclusion $\Dcris(D) \subseteq \DdR(D)$ définie par $x \in \Dcris(D) \mapsto \iota_n (\varphi^n(x)) \in \DdR(D)$, où on a noté $\iota_n = \varphi^{-n} \colon D^{]0, r_n]}[1/t] \to \Ddif(D)$ l'application de localisation \footnote{Il existe ici un petit abus évident en notant par $\varphi^{-n}$ deux applications différentes, l'une étant l'application de localisation notée usuellement $\iota_n$ et l'autre l'inverse de l'opérateur $\varphi$ agissant sur $\Dcris(D)$, mais cela ne devrait pas causer problèmes de lecture.}. On a $$ \dim_L \Dcris(D) \leq \dim_L \DdR(D) \leq \mathrm{rang}_\Robba \; D,  $$ où la première inégalité est évidente par ce qui précède et la dernière suit en remarquant que $\Ddif(D)$ est un $L_\infty((t))$-espace vectoriel de rang $d = \mathrm{rang}_\Robba \; D$ et $\DdR(D) = (\Ddif(D))^\Gamma$ est donc un $(L_\infty((t)))^\Gamma = L$-espace vectoriel de rang $\leq d$.  

\begin{definition}
Soit $D$ un $(\varphi, \Gamma)$-module sur $\Robba$. On dit que $D$ est \textit{cristallin} (resp. \textit{de Rham}) si l'inégalité $\dim_\qp \Dcris(D) \leq \mathrm{rang}_\Robba \; D$ (resp. $\dim_\qp \DdR(D) \leq \mathrm{rang}_\Robba \; D$) est une égalité.
\end{definition}

Si $D$ est de Rham, on définit ses \textit{poids de Hodge-Tate} comme les opposés des entiers où la filtration change, c'est-à-dire l'ensemble $\{ -h \in \Z: \mathrm{Fil}^h \, \DdR(D) / \mathrm{Fil}^{h + 1} \, \DdR(D) \neq 0 \}$.

\subsection{L'équation différentielle $p$-adique $\Nrig(D)$} \label{eqdiff}

Rappelons la construction de l'équation différentielle $p$-adique associée à un $(\varphi, \Gamma)$-module de de Rham.

\begin{proposition} [{\cite[Th. III.2.3]{Berger08}}]
Soit $D \in \Phi\Gamma(\Robba)$ de rang $d$, de Rham, et, pour chaque $n \geq m(D)$, posons
\[ \mathbf{N}_{\mathrm{rig}}^{]0, r_n]}(D) := \{ x \in D^{]0, r_n]}[1/t] : \text{ $\varphi^{-m}(x) \in L_m[[t]] \otimes_L \DdR(D)$ pour tout $m \geq n$} \}, \]
et $\mathbf{N}_{\mathrm{rig}}(D) := \varinjlim_n \Nrig^{]0, r_n]}(D)$. Alors, $\Nrig(D)$ est un $(\varphi, \Gamma)$-module sur $\Robba$, de rang $d$, qui satisfait
\begin{itemize}
\item $\Nrig(D)[1/t] = D[1/t]. $
\item $\mathbf{D}_{\mathrm{dif}, n}^+(\Nrig(D)) = L_n[[t]] \otimes_L \mathbf{D}_{\mathrm{dR}}(D)$ pour tout $n \geq m(D)$.
\item $\nabla(\Nrig(D)) \subseteq t \Nrig(D).$
\end{itemize}
\end{proposition}

Le $(\varphi, \Gamma)$-module $\Nrig(D)$ ainsi obtenu est de Rham à poids de Hodge-Tate tous nuls. Remarquons que l'on peut reconstruire $D$, à partir de la donnée de $\Nrig(D)$ et de la filtration de Hodge sur $\DdR(D)$, en utilisant (cf. \cite[\S II.2]{Berger08} ou bien \cite[\S 3.1]{ColmezPoids}) la formule $$ D = \{ x \in \Nrig(D)[1/t] \; : \; \varphi^{-n}(x) \in \mathrm{Fil}^0(L_n((t)) \otimes_\qp \DdR(D)) \;\;\; \forall n \gg 0 \}. $$

La troisième propriété caractérisant $\Nrig(D)$ permet de définir un opérateur différentiel $$ \partial := \frac{1}{t} \nabla \colon \Nrig(D) \to \Nrig(D), $$ satisfaisant les identités $\partial \circ \varphi = p \; \varphi \circ \partial$ et $\partial \circ \sigma_a = a \; \sigma_a \circ \partial $.
%Si $D$ est un $(\varphi, \Gamma)$-module de de Rham sur $\Robba$, on notera généralement $\Delta = \Nrig(D)$.

\subsection{$(\varphi, \Gamma)$-modules relatifs} \label{phiGammarel} Soit $A$ une algèbre affino\"ide sur $\qp$. L'anneau de Robba $\Robba_A$ relatif à $A$ est défini en posant, pour $0 < r < s \leq +\infty$, $$ \Robba_A^{[r, s]} = \Robba^{[r, s]} \widehat{\otimes} A; \;\;\; \Robba_A^{]0, s]} = \varprojlim_{0 < r < s} \Robba_A^{[r, s]}; \; \; \; \Robba_A = \varinjlim_{s > 0} \Robba_A^{]0, s]}, $$ où le premier produit tensoriel est le produit tensoriel complété entre deux espaces de Banach.

Si $0 < r < s \leq \frac{1}{p - 1}$, on a un endomorphisme $A$-linéaire d'anneaux $\varphi \colon \Robba_A^{[r, s]} \to \Robba_A^{[r/p, s / p]}$ (resp. $\varphi \colon \Robba_A^{]0, s]} \to \Robba_A^{]0, s/p]}$ et $\varphi \colon \Robba_A \to \Robba_A$) qui envoie $T$ sur $(1 + T)^p - 1$, et une action continue du groupe $\Gamma$, agissant par $\sigma_a(T) = (1 + T)^a - 1$, $a \in \zpe$, sur tous les anneaux définis ci-dessus.

Si $0 < r < s$, l'anneau $\Robba_A^{[r/p, s/p]}$ est un $\varphi(\Robba_A^{[r, s]})$-module libre de rang $p$, ce qui permet de définir un inverse à gauche de $\varphi$ \[ \psi := p^{-1} \varphi^{-1} \circ \mathrm{Tr}_{\Robba_A^{[r/p, s/p]} / \varphi(\Robba_A^{[r, s]})} \colon \Robba_A^{[r/p, s/p]} \to \Robba_A^{[r, s]}. \] On note aussi $\psi \colon \Robba_A \to \Robba_A$ l'application qui s'en déduit. 

Si $D^{]0, s]}$ est un $\Robba_A^{]0, s]}$-module et $0 < s' < s$, on note $D^{]0, s']} := D^{]0, s]} \otimes_{\Robba_A^{]0, s]}} \Robba_A^{]0, s']}$ et $\varphi^*(D^{]0, s]}) := D^{]0, s]} \otimes_{\Robba_A^{]0, s]}, \varphi} \Robba_A^{]0, s/p]}$.

\begin{definition}
Un \textit{$(\varphi, \Gamma)$-module sur $\Robba_A^{]0, s]}$} est un module projectif de type fini $D^{]0, s]}$ sur $\Robba_A^{]0, s]}$, muni d'un isomorphisme $\Robba^{]0, s/p]}$-linéaire $\widetilde{\varphi} \colon \varphi^*(D^{]0, s]}) \to D^{]0, s/p]}$ et d'une action semi-linéaire de $\Gamma$, commutant avec $\widetilde{\varphi}$ dans le sens évident.

Un \textit{$(\varphi, \Gamma)$-module sur $\Robba_A$} est un $\Robba_A$-module projectif de type fini $D$ tel qu'il existe $s > 0$ et un $(\varphi, \Gamma)$-module $D^{]0, s]}$ sur $\Robba_A^{]0, s]}$ tel que $D \cong D^{]0, s]} \otimes_{\Robba_A^{]0, s]}} \Robba_A$.
\end{definition}

Soit $D^{]0, s]}$ un $(\varphi, \Gamma)$-module sur $\Robba_A^{]0, s]}$. L'isomorphisme $\widetilde{\varphi} \colon \varphi^*(D^{]0, s]}) \to D^{]0, s/p]}$ induit, pour tout $0 < s' < s$, des opérateurs semi-linéaires $\varphi \colon D^{]0, s']} \to D^{]0, s'/p]}$, définis comme la composée
\[ \varphi \colon D^{]0, s']} = D^{]0, s]} \otimes_{\Robba_A^{]0, s]}} \Robba_A^{]0, s']} \to \varphi^*(D^{]0, s]}) \otimes_{\Robba_A^{]0, s]}} \Robba_A^{]0, s']} \to D^{]0, s/p]} \otimes_{\Robba_A^{]0, s/p]}} \Robba_A^{]0, s'/p]} = D^{]0, s'/p]}, \]
la première flèche étant celle induite par $D^{]0, s]} \to \varphi^*(D^{]0, s]}) \colon x \mapsto x \otimes 1$ et la deuxième étant $\widetilde{\varphi} \otimes \varphi$. L'isomorphisme $\varphi^*(D^{]0, s']}) \cong D^{]0, s'/p]}$ induit un morphisme $A$-linéaire $\psi \colon D^{]0, s'/p]} \cong \varphi(D^{]0, s']}) \otimes_{\varphi(\Robba_A^{]0, s']})} \Robba_A^{]0, s'/p]} \to D^{]0, s']}$ pour tout $0 < s' \leq s$, défini par $\psi(\varphi(z) \otimes f) = z \otimes \psi(f)$. Si $D$ est un $(\varphi, \Gamma)$-module sur $\Robba_A$, on a de même des opérateurs $\varphi \colon D \to D$ et $\psi \colon D \to D$.

Plus généralement, si $X$ est un espace rigide analytique sur $\qp$ et $r \geq 0$, on définit $\Robba_X^{]0, r]}$ comme le faisceau des fonctions rigides analytiques sur $X \times B_{]0, r]}$, $\Robba_X = \varinjlim_{r < 0} \Robba_X^{]0, r]}$ et un $(\varphi, \Gamma)$-module sur $\Robba_X$ est une collection compatible de $(\varphi, \Gamma)$-modules sur $\Robba_A$ pour chaque ouvert admissible $\mathrm{Sp}(A)$ de $X$.

\subsection{Cohomologie des $(\varphi, \Gamma)$-modules} \label{cohomphigamma}

Soient $A$ une algèbre aff\"inoide sur $\qp$ et $D \in \Phi \Gamma(\Robba_A)$. On note $\Gamma' \subseteq \Gamma$ la partie de $p$-torsion de $\Gamma$ (qui est triviale si $p \neq 2$ et cyclique d'ordre $2$ quand $p = 2$). Soit $\gamma \in \Gamma$ tel que son image dans $\Gamma / \Gamma'$ en est un générateur topologique. On pose $\gamma_0 = \gamma$ et, pour $n \geq 1$, $\gamma_n$ un générateur topologique de $\Gamma_n$. Pour $\delta \in \{ \varphi, \psi \}$ et $\gamma' \in \{ \gamma_n : n \geq 0\}$, on note $D' = D^{\Gamma'}$ si $n = 0$ et $D' = D$ si $n \geq 1$, et on définit le complexe $$ \mathscr{C}^\bullet_{\delta, \gamma'}(D): 0 \to D' \to D' \oplus D' \to D' \to 0, $$ où les flèches sont données, respectivement, par $x \mapsto ((\delta - 1)x, (\gamma' - 1)x)$ et $(x, y) \mapsto (\gamma' - 1)x - (\delta - 1) y$. Les modules $H_{\delta, \gamma'}^\bullet(\qp, D)$ sont définis comme les groupes de cohomologie de ce complexe.

\begin{proposition} [{\cite[Prop. 2.3.6 et Th. 4.4.2]{KPX}}] \label{KPX}
Soient $A$ une algèbre aff\"inoide sur $\qp$ et $D$ un $(\varphi, \Gamma)$-module sur $\Robba_A$. Les complexes $\mathscr{C}_{\varphi, \gamma'}(D)$ et $\mathscr{C}_{\psi, \gamma'}(D)$ sont quasi-isomorphes et les groupes de cohomologie $H^i_{\varphi, \gamma'}(D)$ sont des $A$-modules de type fini, compatibles au changement de base. On a une dualité locale et une formule de Euler-Poincaré.
\end{proposition}

\subsection{Cohomologie d'Iwasawa des $(\varphi, \Gamma)$-modules} \label{cohomIw}

%Soit $\Lambda = \zp[[\Gamma]]$ l'algèbre d'Iwasawa de $\Gamma$. On décompose $\Gamma = \Gamma^{\rm tors} \times \Gamma^{\rm st}$, où $\Gamma^{\rm tors}$ et $\Gamma^{\rm st}$ désignent, respectivement, la partie de torsion et sans torsion de $\Gamma$, ce qui fournit un isomorphisme $\Lambda \cong \zp[\Gamma^{\rm tors}] \otimes_\zp \zp[[\Gamma^{\rm st}]]$. Soit $\gamma$ un générateur topologique de $\Gamma^{\rm st}$ et notons $[\gamma]$ son image dans $\Lambda$. On obtient un isomorphisme $\Lambda \cong \zp[\Gamma^{\rm tors}] \otimes_\zp \zp[[T]]$ en envoyant $[\gamma]$ sur $1+T$.

Soit $\Lambda = \zp[[\Gamma]]$ l'algèbre d'Iwasawa de $\Gamma$. Si l'on choisit un isomorphisme $\Gamma = \Gamma^{\rm tors} \times \widetilde{\Gamma}$, où $\Gamma^{\rm tors}$ désigne la partie de torsion de $\Gamma$ et $\widetilde{\Gamma} \cong 1 + 2p \zp$ via le caractère cyclotomique, on obtient un isomorphisme $\Lambda \cong \zp[\Gamma^{\rm tors}] \otimes_\zp \zp[[\widetilde{\Gamma}]]$. Soit $\gamma$ un générateur topologique de $\widetilde{\Gamma}$ et notons $[\gamma]$ son image dans $\Lambda$. On obtient un isomorphisme $\Lambda \cong \zp[\Gamma^{\rm tors}] \otimes_\zp \zp[[T]]$ en envoyant $[\gamma]$ sur $1+T$.

Soit $\Lambda_\infty := \Robba^+(\Gamma)$ l'algèbre de distributions sur $\Gamma$. Précisément, on obtient $\Robba^+(\Gamma)$ en remplaçant la variable $T$ par $[\gamma] - 1$ dans la définition de $\Robba^+$. On peut de la sorte définir les anneaux $\Lambda_n = \Robba^{[r_n, +\infty]}(\Gamma)$ et on a $ \Lambda_\infty = \varprojlim_n \Lambda_n. $ Le choix de l'isomorphisme $\Lambda \cong \zp[\Gamma^{\rm tors}] \otimes_\zp \zp[[T]] $ fait que $\Lambda_n$ s'identifie aux fonction analytiques sur la boule $v_p(T) \geq r_n$ et $\Lambda_\infty$ à celui des fonctions analytiques sur la boule ouverte unité. Ce dernier espace s'identifie à l'anneau $H^0(\mathfrak{X}, \mathcal{O})$ des sections globales sur l'espace des poids $p$-adiques, et aussi, par le théorème d'Amice (cf. \cite[Th. 2.2]{S-T}) à l'espace des distributions $\mathscr{D}(\zpe, L)$ sur $\zpe$ à valeurs dans $L$ \footnote{Soit $\mathrm{LA}(\zpe, L) = \varinjlim_{n > 0} \mathrm{LA}_n(\zpe, L)$ l'espace des fonctions localement analytiques sur $\zpe$ à valeurs dans $L$, où $\mathrm{LA}_n(\zpe, L)$ dénote l'espace des fonctions continues $f \colon \zpe \to L$ admettant, pour tout $x \in \zpe$, un développent en séries des puissances autour $x$ convergeant sur la boule $B(x, n) = \{ z \in \cp : v_p(z - x) \geq n \}$. Chaque $\mathrm{LA}_n(\zpe, L)$, muni de la valuation $v_{\mathrm{LA}_n}(f) := \inf_{x \in \zpe} \inf_{z \in B(x, n)} v_p(f(z))$, est un espace de Banach, et l'on munit $\mathrm{LA}(\zpe, L)$ de la topologie de la limite inductive. L'espace des distributions $\mathscr{D}(\zpe, L)$ est défini comme le dual continu de $\mathrm{LA}(\zpe, L)$, muni de la topologie forte. Si $\mu \in \mathscr{D}(\zpe, L)$ et $\phi \in \mathrm{LA}(\zpe, L)$, on note $\int_\zpe \phi \cdot \mu$ l'évaluation de $\mu$ en $\phi$. Cf. \cite[\S I.4]{Colmezfoncdunevariablepadique} pour plus de détails. }

\subsubsection{Déformation cyclotomique} On note (cf. \cite[Def. 4.4.7]{KPX}) $\Lambda_n^\iota$ le module $\Lambda_n$ muni de l'action de $\Gamma$ via $\gamma(f) = [\gamma^{-1}] f$, $\gamma \in \Gamma$ et $f \in \Lambda_n$. On définit 
\[ \mathbf{Dfm} = \varprojlim_{n} \Robba \widehat{\otimes}_\qp  \Lambda_n^\iota = \varprojlim_{n} \varinjlim_{s > 0} \varprojlim_{r < s} \Robba^{[s, r]} \widehat{\otimes} \Lambda_n^\iota . \] Plus généralement, si $D$ est un $(\varphi, \Gamma)$-module sur $\Robba$, on définit sa déformation cyclotomique par
\[ \mathbf{Dfm}(D) = D \widehat{\otimes}_\qp \Lambda_\infty^\iota = \varprojlim_n \varinjlim_{s > 0} \varprojlim_{r < s} D^{[r, s]} \widehat{\otimes} \Lambda^\iota_n. \]
Les actions $\varphi$, $\psi$ et $\Gamma$ sont données par les formules 
\[ \varphi(x \otimes \lambda) = \varphi(x) \otimes \lambda, \;\;\; \psi(x \otimes \lambda) = \psi(x) \otimes \lambda, \;\;\; \gamma(x \otimes \lambda) = \gamma(x) \otimes [\gamma^{-1}] \lambda, \]
pour $x \in D$, $\lambda \in \Lambda_\infty$ et $\gamma \in \Gamma$. Le module $\mathbf{Dfm}(D)$ est un $(\varphi, \Gamma)$-module sur l'anneau $\Robba_\mathfrak{X} = \varprojlim_{n} \Robba_{\mathfrak{X}_n}$ de Robba relatif à l'espace des poids \footnote{Rappelons d'abord que $\Robba_{\mathfrak{X}_n} = \Robba \widehat{\otimes} \Lambda_n = \varinjlim_s \varprojlim_r \Robba^{[r, s]} \widehat{\otimes} \Lambda_n$ et que $\Gamma$ agit trivialement sur le deuxième facteur. Si $x \otimes \lambda \in D \widehat{\otimes} \Lambda_n^\iota$, $f \otimes \mu \in \Robba_{\mathfrak{X}_n} \cong \Robba \widehat{\otimes} \Lambda_n$ et $\gamma \in \Gamma$, alors on a
\begin{eqnarray*} \gamma((f \otimes \mu) (x \otimes \lambda)) &=& \gamma(f x \otimes \mu \lambda) = \gamma(fx) \otimes [\gamma^{-1}] \lambda \mu \\
&=& (\gamma(f) \otimes \mu) (\gamma(x) \otimes [\gamma^{-1}] \lambda) = \gamma(f \otimes \mu) \gamma(x \otimes \lambda),
\end{eqnarray*}
ce qui montre que l'action de $\Gamma$ est semi-linéaire, et donc $\mathbf{Dfm}(D)$ est bien un $(\varphi, \Gamma)$-module sur $\Robba_{\mathfrak{X}}$ comme défini dans \S \ref{phiGammarel}.}.

\subsubsection{Cohomologie d'Iwasawa analytique} On définit la cohomologie d'Iwasawa (analytique) de $D$ comme $$ H^i_{\rm Iw}(D) = H^i_{\psi, \gamma}(\mathbf{Dfm}(D)). $$ Ce sont, d'après la proposition \ref{KPX}, des $\Lambda_\infty$-modules de type fini.

%Les groupes de cohomologie d'Iwasawa admettent une interprétation en termes des groupes de cohomologie à valeurs dans $\mathscr{D}(\zpe, D)$. Plus précisément, si $M = \varinjlim_{s} \varprojlim_r M^{[r, s]}$ est une limite inductive des limites projectives d'espaces de Banach $p$-adiques, on définit l'algèbre des distributions sur $\zpe$ à valeurs dans $M$ comme \[ \mathscr{D}(\zpe, M) := \varprojlim_n \varinjlim_s \varprojlim_r \mathscr{D}(\zpe, L) \widehat{\otimes} M^{[r, s]}, \] 

Si $\eta \colon \Gamma \to L^\times$ est un caractère, le changement de base par rapport à $f_\eta \colon \Lambda_\infty \to L : [\gamma] \mapsto \eta^{-1}(\gamma)$ fournit un isomorphisme $$\mathbf{Dfm}(D) \otimes_{\Lambda_\infty, f_\eta} L \xrightarrow{\sim} D(\eta). $$Si $\mu \in H^i_{\rm Iw}(D)$, cet isomorphisme induit des morphismes de spécialisation $$ H^i_{\rm Iw}(D) \to H^i_{\psi, \gamma}(D(\eta)): \; \; \; \mu \mapsto \int_\Gamma \eta \cdot \mu, $$ la notation étant justifiée par l'interprétation classique de la cohomologie d'Iwasawa en termes des distributions \footnote{Rappelons que, si $V \in \mathrm{Rep}_L \mathscr{G}_\qp$ et si $D = \D_{\rm rig}(V) \in \Phi \Gamma^{\text{ét}}(\Robba)$ est le $(\varphi, \Gamma)$-module associé à $V$ par l'équivalence de catégories de Fontaine-Kedlaya, alors $H^i_{\rm Iw}(D) = H^i(\qp, \mathscr{D}(\Gamma, V))$ (cf. \cite[Cor. 4.4.11]{KPX} et \cite[Prop. II.1.8]{ColmezIw1}). Dans ce cas-là, on peut voir un élément dans $H^i_{\rm Iw}(D)$ comme un cocycle à valeurs dans $\mathscr{D}(\zpe, V)$ et, si $\eta \colon \Gamma \to L^\times$ est un caractère continu, on en déduit une application de spécialisation $H^i_{\rm Iw}(D) \to H^i(\qp, V(\eta)) \cong H^i(D(\eta))$ envoyant un cocycle $g \mapsto \mu_g$ vers le cocycle $g \mapsto \int_\Gamma \eta \cdot \mu_g$ (cf. \cite[\S II.1]{ColmezIw1} pour plus de détails).}

\subsubsection{Cohomologie d'Iwasawa et $(\varphi, \Gamma)$-modules} Si $D \in \Phi\Gamma(\Robba)$, on définit le complexe $\mathscr{C}^\bullet_\psi(D)$, concentré en $[1, 2]$, par $$ [D \xrightarrow{\psi - 1} D]. $$ Le complexe $\mathscr{C}^\bullet_{\psi}(D)$ appartient à $\D^{[0, 2]}_{\rm perf}(\Lambda_\infty)$ et calcule la cohomologie d'Iwasawa de $D$ (cf. \cite[Th. 4.4.8]{KPX}). On a, en particulier, un isomorphisme $$ \mathrm{Exp}^* \colon H^1_{\rm Iw}(\qp, D) \to D^{\psi = 1} $$ dont l'inverse est donnée par $ z \mapsto [\frac{p-1}{p} \log(\chi(\gamma)) (z \otimes 1), 0]$ \footnote{Si $p = 2$, il faut appliquer à $z$ le projecteur naturel sur le sous espace d'éléments $\Gamma'$-invariants. On évitera ce cas-ci, se traitant de la même manière mais avec une complication technique supplémentaire. }.

Si $z \in D^{\psi = 1}$, on note, afin d'alléger les notation, $\mu_z$ l'élément $(\mathrm{Exp}^*)^{-1}(z) \in H^1_{\rm Iw}(\qp, D)$. Si $n \geq 0$, on a (cf. \cite[\S VIII.1.3]{ColmezPhiGamma} ou \cite[\S 2.2.3, Eq. (6)]{Nakamura2}) la formule pour la spécialisation
\[ \int_{\Gamma} \eta \cdot \mu_z = [ \frac{p - 1}{p} \log(\gamma) (z \otimes e_\eta), 0] \in H^1_{\psi, \gamma}(D(\eta)). \]
% où $\tau_n(\gamma_n) = p^{-n} \log(\chi(\gamma_n))$, si $n \geq 1$ et $\tau_0(\gamma_0) = \frac{p - 1}{p} \log(\chi(\gamma_0))$.

Terminons en mentionnant que $\eta$ induit un automorphisme sur $\Lambda_\infty$ donné par $\eta([\gamma]) = \eta(\gamma)^{-1} \cdot [\gamma]$, ce qui induit un isomorphisme de $(\varphi, \Gamma)$-modules $\mathbf{Dfm}(D) \cong \mathbf{Dfm}(D(\eta))$, donné par $x \otimes [\gamma] \mapsto (x \otimes e_\eta) \otimes \eta(\gamma)^{-1} [\gamma]$, et donc un isomorphisme de $\Lambda_\infty$-modules \[ H^i_{\rm Iw}(\qp, D) \to H^i_{\rm Iw}(\qp, D(\eta)): \;\;\; \mu \mapsto \mu \otimes e_\eta. \] On a, par exemple, $$ \int_\Gamma \eta \cdot \mu = \int_\Gamma 1 \cdot (\mu \otimes e_\eta). $$

\subsection{Dual de Tate et résidus} \label{dual1}

Soit $\mathscr{A} \in \{ \E, \Robba \}$, soit $D \in \Phi \Gamma(\mathscr{A})$ et notons $\omega_D = (\det_D) \chi^{-1}$. On définit le dual de Tate de $D$ comme $\check{D} = \mathrm{Hom}_{\varphi, \Gamma}(D, \mathscr{A} \frac{dT}{1 + T})$, où $\mathscr{A} \frac{dT}{1 + T}$ est le $(\varphi, \Gamma)$-module étale libre de rang $1$ de base $\frac{dT}{1 + T}$ sur lequel $\varphi$ et $\Gamma$ agissent par les formules $\gamma (\frac{dT}{1 + T}) = \chi(\gamma) \frac{dT}{1 + T}$ si $\gamma \in \Gamma$, $\varphi(\frac{dT}{1 + T}) = \frac{dT}{1 + T} $. On note \[ \langle \; , \; \rangle \colon \check{D} \times D \to \mathscr{A} \frac{dT}{1+T} \] l'accouplement naturel. Si $D$ est de dimension $2$, on a une identification $\check{D} = D \otimes \omega_D^{-1}$. 

Si $f =\sum_{n \in \Z} a_n T^n \in \mathscr{A}$, on pose $\mathrm{r\acute{\mathrm{e}}s}_0(f dT) = a_{-1}$. Si $x \in \check{D}$ et $y \in D$, la formule $$ \{x,y\} = \mathrm{r\acute{\mathrm{e}}s}_0(\langle \sigma_{-1} x, y \rangle) $$ définit un accouplement parfait $ \{ \; , \; \}: \check{D} \times D \to L$ identifiant $\check{D}$ au dual topologique $D^*$ de $D$. Si $D$ est de dimension $2$, on a, sous l'identification $\check{D} = D \otimes \omega_D^{-1}$, la formule \footnote{On voit ici l'élément $e_{\omega_D^{-1}}$ comme $e_{\det_D}^{-1} dT$, de sorte que, si $\sigma_{-1}(x) \wedge y = f \cdot e_{\det_D}$, alors $(\sigma_{-1}(x) \wedge y) \otimes e_{\omega_D^{-1}} = f dT$.} \[ \{ x \otimes e_{\omega_D^{-1}}, y \} = \mathrm{r\acute{\mathrm{e}}s}_0((\sigma_{-1}(x) \wedge y) \otimes e_{\omega_D^{-1}}), \] et on note $[x, y] := \{ x \otimes e_{\omega_D^{-1}}, y \}$.

\subsection{Dualité locale}

Considérons, d'ici jusqu'à la fin de cette section, $D \in \Phi \Gamma^{\text{ét}}(\E)$ et $D_{\rm rig} \in \Phi \Gamma^{\text{ét}}(\Robba)$ correspondant à $D$ par l'équivalence de catégories de Kedlaya entre $\Phi \Gamma^{\text{ét}}(\E)$ et $\Phi \Gamma^{\text{ét}}(\Robba)$. Les accouplements introduits ci-dessous seront utilisés dans la dernière partie de cet article. Pour une référence pour tout ce qui suit, cf. \cite{ColmezMirabolique} et \cite{ColmezPhiGamma}.
%Soit $n \geq 0$ et posons $\tau_n(\gamma_n) = p^{-n} \log \chi(\gamma_n)$, $\tau_0(\gamma_0) = \frac{p-1}{p} \log \chi(\gamma_0)$.
Notons $\gamma = \gamma_0$ dans la suite. Notons, pour $i \in \{ 0, 1\}$, $$ \cup: H^i_{\varphi, \gamma}(\check{D}_{(\rm rig)}) \times H^{2- i}_{\varphi, \gamma}(D_{(\rm rig)}) \to L $$ l'accouplement local. Un calcul avec le complexe de Herr et l'isomorphisme de Fontaine nous permet d'exprimer l'accouplement local en termes des $(\varphi, \Gamma)$-modules:

\begin{lemme} [{\cite[\S VIII.1.3]{ColmezPhiGamma}}] \label{lemmeIwacc} Soient $\check{z} \in \check{D}_{(\rm rig)}^{\psi = 1}$, $z \in D_{(\rm rig)}^{\psi = 1}$ et $\eta: \Gamma \to L^\times$ est un caractère continu. Alors
%\[  (\int_{\Gamma_n} \eta^{-1} \cdot \mu_{\check{z}}) \cup (\int_{\Gamma_n} \eta \cdot \mu_{z}) = \{ (1 - \varphi) \check{z}, \frac{- \tau_n(\gamma_n)}{\eta(\chi(\gamma_n)) \gamma_n - 1} (1 - \varphi) z \}. \]
\[  (\int_{\Gamma} \eta^{-1} \cdot \mu_{\check{z}}) \cup (\int_{\Gamma} \eta \cdot \mu_{z}) = \{ (1 - \varphi) \check{z}, - \frac{p - 1}{p} \frac{\log(\gamma)}{\eta(\gamma) \gamma - 1} (1 - \varphi) z \}. \]
\end{lemme}

\subsection{Applications exponentielles} \label{appsexp}

Si $D \in \Phi\Gamma(\Robba)$ est de Rham et $n \geq 0$, on note $$ \exp_{D} \colon \DdR(D) \to \H^1_{\varphi, \gamma_n}(D), $$ $$ \exp^*_{D} \colon \H^1_{\varphi, \gamma_n}(D) \otimes \DdR(D), $$ les applications exponentielle et exponentielle duale de Bloch-Kato comme définies dans \cite[\S 2.3, \S 2.4.]{Nakamura}. Quand cela ne pose pas de problèmes, on omettra les indices dans les notations des applications exponentielles.

\subsubsection{Description de l'application exponentielle} Explicitement (\cite[Lem. 2.12(1)]{Nakamura}), si $x \in \DdR(D) = (\D_{\mathrm{dif}}(D))^\Gamma$, alors il existe $m$ tel que $x \in \D_{\mathrm{dif}, m}(D)$, et il existe aussi $\tilde{x} \in D^{]0, r_m]}[1/t]$ tel que $\varphi^{-k}(\tilde{x}) - x \in \D_{\mathrm{dif}, k}^+(D)$ pour tout $k \geq m$. Alors on a $$ \exp_{D}(x) = [(\gamma - 1) \tilde{x}, (\varphi - 1) \tilde{x} ] \in H^1_{\varphi, \gamma}(D). $$ On remarque que l'application $\exp_D$ est nulle sur $\mathrm{Fil}^0 \, \DdR(D)$.

\subsubsection{Description de l'application exponentielle duale} Si $M$ est un module muni d'une action de $\Gamma$, et $\gamma' \in \{ \gamma_n: n \geq 0 \}$, on pose \footnote{Comme précédemment, $M' = M^{\Gamma'}$ si $\gamma' = \gamma_0$ et $M' = M$ autrement.} $$ C^\bullet_{\gamma'}(M) = [M' \xrightarrow{\gamma' - 1} M'], $$ et on définit les groupes de cohomologie $H^i_{\gamma'}(M) = H^i(C^\bullet_{\gamma'}(M))$. Par exemple, si $D \in \Phi\Gamma(\Robba)$ et $n \geq 0$, alors $$ H^0_{\gamma_n}(\Ddif(D)) = \Ddif(D)^{\Gamma_n} = L_n \otimes \D_{\rm dR}(D). $$

Si $D \in \Phi\Gamma(\Robba)$ est de Rham et $n \geq 0$, on a un isomorphisme $L_\infty((t)) \otimes_{L_n} (L_n \otimes \DdR(D)) \cong \D_{\mathrm{dif}}(D)$ et l'application $L_n \otimes \DdR(D) \to H^1_{\gamma_n}(\D_{\mathrm{dif}}(D))$ qui envoie $x \in L_n \otimes \DdR(D)$ vers la classe de cohomologie $[\log \chi(\gamma_n) (1 \otimes x)] \in H^1_{\gamma_n}(\D_{\mathrm{dif}}(D))$ est un isomorphisme (cf. \cite{Nakamura}, juste avant le lemme 2.14). De plus, on a une application naturelle $H^1_{\varphi, \gamma_n}(D) \to H^1_{\gamma_n}(\Ddifp(D)) \to H^1_{\gamma_n}(\Ddif(D))$ définie par $[(x, y)] \mapsto [\varphi^{-m} x]$, où $m \gg 0$ et $[ \cdot ]$ dénote la classe de cohomologie correspondante.
%Cette application est bien définie: si $(x, y) \in D \oplus D$ est un cocycle représentant une classe de cohomologie dans $H^1_{\varphi, \gamma_n}(D)$, alors $(1 - \varphi) x = (1 - \gamma_n) y$ et donc $\varphi^{-m} x - \varphi^{-(m - 1)} x = (1 - \gamma_n) \varphi^{-m} y$, ce qui implique que les images de $\varphi^{-m} x$ et $\varphi^{-(m - 1)}x$ dans $H^1_{\gamma_n}(\D_{\rm dif}^+(D)) = \D_{\rm dif}^+(D)' / (1 - \gamma_n) \D_{\rm dif}^+(D)'$ \footnote{Rappelons une dernière fois que si $M$ est un $\Gamma$-module et $n \geq 0$, on note $M' = M^{\Gamma'}$ si $n = 0$ et $M'$ si non.} coïncident.
On définit $$ \exp^*_{D} \colon H^1_{\varphi, \gamma_0}(D) \to \DdR(D) $$ comme la composition $H^1_{\varphi, \gamma_0}(D) \to H^1_{\gamma_0}(\Ddifp(D)) \to H^1_{\gamma_0}(\Ddif(D))  \xrightarrow{\sim} \DdR(D)$. Remarquons que, par construction, l'image de $\exp^*_{D}$ tombe dans $\mathrm{Fil}^0 \, \DdR(D)$.

Les applications exponentielle et exponentielle duale sont l'une l'adjointe de l'autre par rapport à la dualité de Tate et à la dualité de Poincaré. Plus précisément, on note
\[ \{ \; , \; \}_{\rm dR} \colon \DdR(\check{D}) \times \DdR(D) \to L \]
l'accouplement induit par la dualité entre $\check{D}_{(\rm rig)}$ et $D_{(\rm rig)}$ composé avec la trace de Tate normalisée. On a alors le résultat suivant.

\begin{lemme} [{\cite[Prop. 2.16]{Nakamura}}] \label{adjonction} Soit $D \in \Phi \Gamma(\Robba)$ de de Rham et soient $x \in \DdR(\check{D})$, $y \in H^1_{\varphi, \gamma}(D)$. Alors
\[ \exp_{\check{D}}(x) \cup y = \{ x , \exp^*_D(y) \}_{\rm dR}. \]
\end{lemme}

%Si $D = \D_{\rm rig}(V)$, on a $\exp_{D, F_n} = \exp_{V, F_n}$ et $\exp^*_{D, F_n} = \exp^*_{V, F_n}$ via les isomorphismes $L_n \otimes \DdR(D) = L_n \otimes \DdR(V)$ et $H^i_{\varphi, \gamma_n}(D) \cong H^i(F_n, V)$. Enfin, les applications $\exp$ et $\exp^*$ sont l'une l'adjointe de l'autre: si $x \in \DdR(D)$ et $y \in H^1_{\varphi, \gamma}(D^*(1))$, on a $$ \langle x, \exp^*_{D^*(1)}(y) \rangle = \exp_{D}(x) \cup y, $$ où l'accouplement de gauche $ \langle \;,\; \rangle \colon \DdR(D) \times \DdR(D^*(1)) \to \qp$ est induit par $\DdR(D) \times \DdR(D^*(1)) \to \DdR(\Robba(1)) \cong \qp$ et celui de droite est l'accouplement local de cohomologie. Quand cela ne pose pas de problèmes, on omettra les indices dans les notations des applications exponentielles.

\subsection{Lois de réciprocité} \label{loisrec}

On rappelle à continuation les formules permettant d'exprimer l'image par les applications exponentielle et exponentielle duale des différentes spécialisations d'un élément dans la cohomologie d'Iwasawa en termes des $(\varphi, \Gamma)$-modules, normalement connues sous le nom de `lois de réciprocité explicite'.

Soit $D \in \Phi \Gamma(\Robba)$ de Rham. Si $z \in D^{\psi = 1}$, le diagramme commutatif
\[ \xymatrixcolsep{4pc} \xymatrix{
    D^{]0, r_{n+1}]} \ar[d]_{\psi} \ar[r]^{\varphi^{-(n+1)}} & L_{n+1}((t)) \otimes \DdR(D) \ar[d]^{p^{-1} \mathrm{Tr}_{L_{n+1} / L_n}} \\
    D^{]0, r_n]} \ar[r]^{\varphi^{-n}} & L_n((t)) \otimes \DdR(D), }
    \]
où $n \geq m(D)$, montre que les valeurs 
\[ p^{-m} \, \mathrm{Tr}_{L_m / L}([\varphi^{-m} z]_0) \in \DdR(D) \]
ne dépendent pas de $m$ assez grand, où, pour $\sum_{n \in \Z} a_n t^n \in \DdR(D) \otimes L_n((t))$, l'on note $[\sum_{n \in \Z} a_n t^n]_0 = a_0$.

\begin{proposition} \label{loirecexpd}
Soit $h \in \Z$ tel que $\mathrm{Fil}^{-h} \, \DdR(D) = \DdR(D)$ et soient $z \in D^{\psi = 1}$ et $j \geq -h$. On a alors l'égalité suivante dans $\DdR(D(-j))$: 
\[ \exp^* (\int_\Gamma \chi^{-j} \cdot \mu_z) = \lim_{m \to +\infty} p^{-m} \, \mathrm{Tr}_{L_m / L}([\varphi^{-m} (z \otimes e_{-j})]_0). \]
\end{proposition}

\begin{proof}
Voir \cite[Th. IV.2.1]{ColmezIw2} ou \cite[Th. II.6]{Berger03} ou \cite[Th. 3.10(2)]{Nakamura}.
\end{proof}

\begin{remarque}
On a une identification $\DdR(D(-j)) = \DdR(D) \otimes \mathbf{e}^{\rm dR}_{-j}$, où $\mathbf{e}^{\rm dR}_{-j} = t^j e_{-j}$ est l'élément défini dans l'introduction (cf. aussi \S \ref{basesdR} plus tard). Observons que $\varphi^{-m} (z \otimes e_{-j}) = \varphi^{-m} z \otimes e_{-j}$ et que, si $\varphi^{-n} z = \sum_{l \geq 0} a_l t^l$, $a_l \in L_m \otimes \DdR(D)$, alors \[ [\varphi^{-m} z \otimes e_{-j}]_0 = a_j \otimes t^j e_{-j} = a_j \otimes \mathbf{e}^{\rm dR}_{-j} \in L_m \otimes \DdR(D) \otimes \mathbf{e}^{\rm dR}_{-j}. \]
\end{remarque}

\begin{proposition} \label{loirecexp}
Supposons que $D$ est à poids de Hodge-Tate $0$ et $k \geq 0$ et soient $z \in D^{\psi = 1}$ et $j \geq 1$. On a alors l'égalité suivante dans $\DdR(D(j))$: 
\[ \exp^{-1}(\int_\Gamma \chi^j \cdot \mu_z) = (-1)^j (j-1)! \lim_{m \to +\infty} p^{-m} \, \mathrm{Tr}_{L_m / L} ( [\varphi^{-m} (\partial^{-j} z \otimes t^{-j} e_j)]_0). \]
\end{proposition}

\begin{proof}
Voir \cite[Th. II.3]{Berger03} ou \cite[Th. 3.10(1)]{Nakamura}.
\end{proof}

\begin{remarque} %\leavevmode
%\begin{itemize}
On a une identification $\DdR(D(j)) = \DdR(D) \otimes \mathbf{e}^{\rm dR}_{j}$. Si $\varphi^{-m} (\partial^{-j} z) = \sum_{l \geq 0} b_l t^l$, $b_l \in L_m \otimes \DdR(D)$, on a \[ [ \varphi^{-m}(\partial^{-j} z \otimes t^{-j} e_j)]_0 = p^{mj} \, b_0 \otimes \mathbf{e}^{\rm dR}_j \in L_m \otimes \DdR(D) \otimes \mathbf{e}^{\rm dR}_j . \]
%\item Si $k - 1 \leq j \leq 0$, l'application $\exp_{D(j)}$ n'est pas inversible car $\mathrm{Fil}^0 \, D(j) = \mathrm{Fil}^j \, D = \mathscr{L}$. L'égalité de la proposition s'interprète soit en passant l'exponentielle du côté droite ou bien en choisissant la base $f_1, f_2$  de $\DdR(D)$ de sorte que $f_1 \in \mathrm{Fil}^0 \, \DdR(D)$, ce qui définit une section $a \overline{f}_2 \mapsto a f_2$ de la projection $\DdR(D) \to \DdR(D) / \mathrm{Fil}^0 \, \DdR(D)$ (cf. \cite[\S 3.4.1]{Nakamura2}).
%\end{itemize}
\end{remarque}

\subsection{Accouplement d'Iwasawa} \label{accIw} Soit $D \in \Phi \Gamma^{\text{ét}}(\E)$ et soient $\mathscr{C} = \mathscr{C}(D) = (1 - \varphi) D^{\psi = 1} \subseteq D^{\psi = 0}$ le c\oe{}ur de $D$ et $\check{\mathscr{C}} = (1 - \varphi) \check{D}^{\psi = 1}$ celui de $\check{D}$. Ce sont des sous-$\Lambda$-modules libres de rang $d$ et l'application naturelle $\E(\Gamma) \otimes_\Lambda \mathscr{C} \to D^{\psi = 0}$ (resp. $\E(\Gamma) \otimes_\Lambda \check{\mathscr{C}} \to \check{D}^{\psi = 0}$) est un isomorphisme (cf. \cite[Corollaire VI.1.3(ii)]{ColmezMirabolique}). Ces espaces sont les orthogonaux l'un de l'autre pour l'accouplement $\{ \; , \; \} \colon \check{D} \times D \to L$ défini dans \S \ref{dual1} (\cite[Lem. VI.1.1]{ColmezMirabolique}). Pour $x \in \check{\mathscr{C}}$, $y \in \mathscr{C}$, la formule $$ \langle x, y \rangle_{\mathrm{Iw}} = \lim_{n \to + \infty} \sum_{i \in (\Z / p^n \Z)^\times}\{ \frac{\tau_n(\gamma_n) \sigma_i}{\gamma_n - 1} \cdot x , y \} [\sigma_i]$$ définit un accouplement
\[ \langle \; , \; \rangle_{\rm Iw}: \check{\mathscr{C}} \times \mathscr{C} \to \Lambda \]
qui identifie $\check{\mathscr{C}}$ à $\mathrm{Hom}_\Lambda(\mathscr{C}, \Lambda)$ (\cite[Prop. VI.1.2]{ColmezMirabolique}).

L'accouplement défini ci-dessus est l'analogue de l'accouplement usuel dans la théorie d'Iwasawa $(\; , \; )_{\rm Iw} : H^1_{\rm Iw}(\qp, \check{V}) \times H^1_{\rm Iw}(\qp, V) \to \Lambda$ dans le sens que (\cite[Rem. VIII.1.5]{ColmezPhiGamma}) 
\[ ((\mathrm{Exp}^*)^{-1} z, (\mathrm{Exp}^*)^{-1} z')_{\rm Iw} = \langle (1 - \varphi)z, (1 - \varphi)z'\rangle_{\rm Iw} \]
pour tous $z \in \check{D}^{\psi = 1}$, $z' \in D^{\psi = 1}$ et où $\mathrm{Exp}^*: H^1_{\rm Iw}(\qp, V) \to D^{\psi = 1}$ dénote l'isomorphisme de Fontaine (\cite[Th. II.1.3]{ColmezIw2}).

Si $\eta \in \mathfrak{X}(L)$, l'application \footnote{Les éléments $e_\eta$ et $e_{\eta^{-1}}$ dénotent des bases des modules $L(\eta)$ $L(\eta^{-1})$, munies d'une action triviale de $\varphi$ et d'une action de $\Gamma$ donnée par $\sigma_a(e_{\eta}) = \eta(a) \cdot e_\eta$, $\sigma_a(e_{\eta^{-1}}) = \eta^{-1}(a) \cdot e_{\eta^{-1}}$, $a \in \zpe$. Ces éléments apparaîtront très souvent dans le texte.} $x \mapsto x \otimes e_\eta$ (resp. $x \mapsto x \otimes e_\eta^{-1}$) induit un isomorphisme de $\mathscr{C}(D)$ (resp. $\check{\mathscr{C}}(D)$) sur $\mathscr{C}(D \otimes \eta)$ (resp. $\check{\mathscr{C}}(D \otimes \eta))$ et on a (\cite[Prop. VI.1.4]{ColmezMirabolique})
\[ \langle x \otimes e_{\eta^{-1}}, y \otimes e_\eta \rangle_{\rm Iw} = m_{\eta^{-1}} \cdot \langle x, y \rangle_{\rm Iw}, \] où $m_{\eta^{-1}}$ dénote la `multiplication par $\eta^{-1}$' d'une mesure, i.e., si $\big( \sum_{i \in \Z / p^n \Z} \alpha_{n, i} [\sigma_i] \big)_{n \in \N}  \in \Lambda$, alors on a $m_{\eta^{-1}} \big( \sum_{i \in \Z / p^n \Z} \alpha_{n, i} [\sigma_i] \big)_{n \in \N} = \big( \sum_{i \in \Z / p^n \Z} \eta^{-1}(i) \alpha_{n, i} [\sigma_i] \big)_{n \in \N}$.
%Les isomorphismes $D \boxtimes \zpe \cong \E(\Gamma) \otimes_\Lambda \mathscr{C}$ et $\check{D} \boxtimes \zpe \cong \E(\Gamma) \otimes_\Lambda \check{\mathscr{C}}$ permettent d'étendre l'accouplement défini ci-dessus en un accouplement $$ \langle \; , \; \rangle_{\rm Iw}: (\check{D} \boxtimes \zpe) \times (D \boxtimes \zpe) \to \E(\Gamma). $$ 

Notons $\mathscr{C}_{\rm rig} = (1 - \varphi) D_{\rm rig}^{\psi = 1}$ et  $\check{\mathscr{C}}_{\rm rig} = (1 - \varphi) \check{D}_{\rm rig}^{\psi = 1}$. On a des isomorphismes $\mathscr{C}_{\rm rig} \cong \Robba^+(\Gamma) \otimes_\Lambda \mathscr{C}$, $\mathscr{\check{C}}_{\rm rig} \cong \Robba^+(\Gamma) \otimes_\Lambda \mathscr{\check{C}}$ (\cite[Prop. V.1.18]{ColmezPhiGamma}), qui permettent d'étendre par linéarité l'accouplement d'Iwasawa en un accouplement 
%$$ \langle \; , \; \rangle_{\zpe}: (\check{D}_{\rm rig} \boxtimes \zpe) \times (D_{\rm rig} \boxtimes \zpe) \to \Robba \frac{dT}{1+T} \boxtimes \zpe $$ 
%$$ \wedge_{\zpe}: (D_{\rm rig} \boxtimes \zpe) \times (D_{\rm rig} \boxtimes \zpe) \to (\wedge D_{\rm rig} \boxtimes \zpe). $$ 
$$ \langle \; , \; \rangle_{\rm Iw}: \mathscr{\check{C}} \times \mathscr{C} \to \Robba^+(\Gamma). $$

% CHANGED
%Notons $\mathscr{C}_{\rm rig} = (1 - \varphi) D_{\rm rig}$ et  $\check{\mathscr{C}}_{\rm rig} = (1 - \varphi) \check{D}_{\rm rig}$. On a des isomorphismes $\mathscr{C}_{\rm rig} \cong \Robba^+(\Gamma) \otimes_\Lambda \mathscr{C}$ (\cite[Prop. V.1.18]{ColmezPhiGamma}) et $D_{\rm rig} \boxtimes \zpe \cong \Robba(\Gamma) \otimes_{\Robba^+(\Gamma)} \mathscr{C}_{\rm rig} \cong \Robba(\Gamma) \otimes_\Lambda \mathscr{C}$ (et ses correspondants $\check{\mathscr{C}}_{\rm rig} \cong ... $) (\cite[Corollaire V.1.13(iii)]{ColmezPhiGamma}) qui permettent d'étendre par linéarité l'accouplement d'Iwasawa en un accouplement 
%%$$ \langle \; , \; \rangle_{\zpe}: (\check{D}_{\rm rig} \boxtimes \zpe) \times (D_{\rm rig} \boxtimes \zpe) \to \Robba \frac{dT}{1+T} \boxtimes \zpe $$ 
%%$$ \wedge_{\zpe}: (D_{\rm rig} \boxtimes \zpe) \times (D_{\rm rig} \boxtimes \zpe) \to (\wedge D_{\rm rig} \boxtimes \zpe). $$ 
%$$ \langle \; , \; \rangle_{\rm Iw}: (\check{D}_{\rm rig} \boxtimes \zpe) \times (D_{\rm rig} \boxtimes \zpe) \to \Robba(\Gamma). $$

Le lemme suivant nous fournit une description des différentes spécialisations de l'accouplement d'Iwasawa en termes des $(\varphi, \Gamma)$-modules:

\begin{lemme} [{\cite[Corollaire VI.1.5]{ColmezMirabolique}}] \label{accIwphiGamma} Soient $x \in \check{\mathscr{C}}_{\rm rig}$, $y \in \mathscr{C}_{\rm rig}$ et $\eta \colon \Gamma \to L^\times$ un caractère continu. Alors
\[ \int_\Gamma \eta \cdot \langle x , y \rangle_{\rm Iw} = \{  x, - \frac{p - 1}{p} \frac{ \log(\gamma)}{\eta(\gamma) \gamma  - 1} y \}. \]
\end{lemme}

%si $\eta: \Gamma \to L^\times$ est un caractère continu, on note $$ \langle \; , \; \rangle_{\rm dR}: \DdR(\check{D}_{(\rm rig)}(\eta^{-1})) \times \DdR(D_{(\rm rig)}(\eta)) \to L $$ l'accouplement induit par la dualité entre $\check{D}_{(\rm rig)}$ et $D_{(\rm rig)}$ composé avec la trace de Tate normalisée.

On aura besoin plus tard du lemme suivant:

\begin{lemme} \label{lemme26} Soient $\check{z} \in \check{D}_{\rm rig}^{\psi = 1}$, $z \in D_{\rm rig}^{\psi = 1}$ et $\eta: \Gamma \to L^\times$ un caractère localement constant. On a
\small{ \[ \int_\Gamma \eta^{-1} \chi^j \cdot \langle (1 - \varphi) \check{z}, (1 - \varphi) z \rangle_{\rm Iw} = \left\{ \begin{array}{l l} \{ \exp^*(\int_\Gamma \eta \chi^{-j} \cdot \mu_{\check{z}}), \exp^{-1} (\int_\Gamma \eta^{-1} \chi^j \cdot \mu_z ) \}_{\rm dR} & \quad \text{si $j \gg 0$} \\
\{ \exp^{-1}(\int_\Gamma \eta \chi^{-j} \cdot \mu_{\check{z}}), \exp^* (\int_\Gamma \eta^{-1} \chi^j \cdot \mu_z )\}_{\rm dR} & \quad \text{si $j \ll 0$}. \\
\end{array} \right. \]} \normalsize
\end{lemme}

\begin{proof}
Par le lemme \ref{accIwphiGamma}, on a
\[ \int_\Gamma \eta^{-1} \chi^j \cdot \langle (1 - \varphi) \check{z} , (1 - \varphi) z \rangle_{\rm Iw} = \{ (1 - \varphi) \check{z}, -\frac{p - 1}{p} \frac{ \log(\gamma)}{\eta^{-1}(\gamma) \chi(\gamma)^j \gamma  - 1} (1 - \varphi) z \}. \]
En utilisant le lemme \ref{lemmeIwacc}, on obtient
\[  \{ (1 - \varphi) \check{z}, -\frac{p - 1}{p} \frac{ \log(\gamma)}{\eta^{-1}(\gamma) \chi(\gamma)^j \gamma  - 1} (1 - \varphi) z \} = (\int_\Gamma \eta \chi^{-j} \cdot \mu_{\check{z}}) \cup (\int_\Gamma \eta^{-1} \chi^j \cdot \mu_z). \]
Or, si $j \ll 0$ (resp. $j \gg 0$), l'application exponentielle (resp. exponentielle duale) de Bloch-Kato pour $\check{D}(\eta \chi^{-j})$ est bijective. Si, par exemple, $j \ll 0$, ceci nous permet d'écrire
\begin{eqnarray*} (\int_\Gamma \eta \chi^{-j} \cdot \mu_{\check{z}}) \cup (\int_\Gamma \eta^{-1} \chi^j \cdot \mu_z) &=& \big( \exp_{\check{D}(\eta \chi^{-j})} (\exp^{-1}_{\check{D}(\eta \chi^{-j})} (\int_\Gamma \eta \chi^{-j} \cdot \mu_{\check{z}} )) \big) \cup \big( \int_\Gamma \eta^{-1} \chi^j \cdot \mu_z \big) \\
&=& \{ \exp^{-1}_{\check{D}(\eta \chi^{-j})} (\int_\Gamma \eta \chi^{-j} \cdot \mu_{\check{z}}), \exp^*_{D(\eta^{-1} \chi^j)} (\int_\Gamma \eta^{-1} \chi^j \cdot \mu_z ) \}_{\rm dR},
\end{eqnarray*}
où la dernière égalité suit de l'adjonction (lemme \ref{adjonction} appliqué à $D(\eta^{-1} \chi^j)$, dont le dual de Tate s'identifie à $\check{D}(\eta \chi^{-j})$) entre les applications exponentielles. Le cas $\j \gg 0$ se traitant de la même manière, ceci nous permet de conclure.
\end{proof}

%\begin{proof}
%On a
%\begin{eqnarray*}
%\int_\Gamma \eta^{-1} \chi^j \cdot \langle \check{z} , z \rangle_{\rm Iw} &=& \{ \check{z} , \frac{ - \log(\gamma)}{\eta^{-1}(\gamma) \chi(\gamma)^j \gamma  - 1} \cdot z \} \\
%&=& \{ (1 - \varphi) \check{z}, \frac{ - \log(\gamma)}{\eta^{-1}(\gamma) \chi(\gamma)^j \gamma  - 1} \cdot (1 - \varphi) z \} \\
%&=& (\int_\Gamma \eta \chi^{-j} \cdot \mu_{\check{z}}) \cup (\int_\Gamma \eta^{-1} \chi^j \cdot \mu_z),
%\end{eqnarray*}
%où la dernière égalité suit de la description explicite de l'accouplement de Tate en termes de $(\varphi, \Gamma)$-modules. Le résultat suit du fait que, quand $j \gg 0$ (resp. $j \ll 0$), l'application exponentielle (resp. exponentielle duale) de Bloch-Kato est bijective, ce qui nous permet, en utilisant la loi de réciprocité de Kato, d'exprimer l'accouplement de Tate en termes de l'accouplement $\langle \; , \; \rangle_{\rm dR}$.
%\end{proof}

%\newpage
\section{Correspondance de Langlands $p$-adique et équation fonctionnelle en dimension $2$} \label{equationfonctionnelle}

Nous nous inspirons des idées de Nakamura (\cite{Nakamura2}) et des techniques de changement de poids de Colmez (\cite{ColmezPoids}) pour démontrer, dans cette section, une équation fonctionnelle au niveau de la théorie d'Iwasawa d'un $(\varphi, \Gamma)$-module de rang $2$.

\subsection{Notations} \label{basesdeRham3} Soit $D \in \Phi\Gamma(\Robba)$ de dimension $2$, de Rham à poids de Hodge-Tate $0$ et $k$ que l'on suppose non triangulin \footnote{Rappelons qu'un $(\varphi, \Gamma)$-module sur $\Robba$ de rang $2$ est dit \textit{triangulin} si il est obtenu comme extension de deux $(\varphi, \Gamma)$-modules de rang $1$. D'après \cite{Kedlaya}, tout $(\varphi, \Gamma)$ module de rang $2$ sur $\Robba$ est étale à torsion près par un caractère ou triangulin. Comme on le verra dans le texte (cf. la preuve du théorème \ref{conjepsilon}), le théorème principal de cet article (théorème \ref{eqfonct1}) est équivalent au fait que l'isomorphisme proposé par Nakamura interpole l'isomorphisme de de Rham (cf. Conjecture \ref{eaeaconj}(5) ainsi que \S \ref{constructionisom}); pour le cas d'un $(\varphi, \Gamma)$-module triangullin, la conjecture, et donc nos résultats, sont connus (\cite[Corollay 3.12]{Nakamura3}, \cite[Corollary 3.10]{Nakamura2}) pour n'importe quel $(\varphi, \Gamma)$-module triangulin, ce qui fait que l'hypothèse de non triangularité ne soit pas vraiment restrictive.} et notons $\Delta = \Nrig(D)$, qui est à poids de Hodge-Tate tous nuls. Comme les poids de Hodge-Tate de $D$ sont $0$ et $k$, et ceux de $\Delta$ sont nuls, le caractère $\det_\Delta$ est localement constant et $\omega_D = \omega_\Delta x^k$ (et $\det_D = x^k \det_\Delta$). Notons $e_D = e_{\det_D}$.

%Soit $\check{D} = \mathrm{Hom}_{\varphi, \Gamma}(D, \Robba \frac{dT}{1 + T})$ le dual de Tate de $D$, où $\Robba \frac{dT}{1 + T}$ est le $(\varphi, \Gamma)$-module étale libre de rang $1$ de base $\frac{dT}{1 + T}$ sur lequel $\varphi$ et $\Gamma$ agissent par les formules $\gamma (\frac{dT}{1 + T}) = \chi(\gamma) \frac{dT}{1 + T}$ si $\gamma \in \Gamma$, $\varphi(\frac{dT}{1 + T}) = \frac{dT}{1 + T} $. On note \[ \langle \; , \; \rangle \colon \check{D}_{\rm rig} \times D_{\rm rig} \to \Robba \frac{dT}{1+T} \] l'accouplement naturel. Soient $\omega_D = (\det_D) \chi^{-1}$ et $\omega_\Delta = (\det_\Delta) \chi^{-1}$. Le fait que $D$ soit de dimension $2$ nous permet d'identifier $\check{D} = D \otimes \omega_D^{-1}$. Comme les poids de Hodge-Tate de $D$ sont $0$ et $k$, et ceux de $\Delta$ sont nuls, le caractère $\det_\Delta$ est localement constant et $\omega_D = \omega_\Delta x^k$ (et $\det_D = x^k \det_\Delta$). Notons $e_D = e_{\det_D}$.

\subsubsection{Bases et modules de de Rham} \label{basesdR} On aura besoin de jongler un peu avec des éléments habitant dans le module de de Rham des différents tordus de $D$ et de son dual de Tate et les identifications suivantes permettront de voir tous ces éléments dans $\DdR(D)$. Soit $\eta \colon \zpe \to L^\times$ un caractère localement constant (vu comme un caractère de $\qpe$ en posant $\eta(p) = 1$). L'élément $G(\eta) e_\eta \in L_\infty$ est fixé par l'action de $\Gamma$ et on a donc un isomorphisme $\DdR(\Robba(\eta)) = (L_\infty((t)) \cdot e_\eta)^\Gamma = L \cdot G(\eta) e_\eta$, ce qui nous fournit un générateur $\mathbf{e}^{\rm dR}_\eta = G(\eta) e_\eta$ de ce module. Si $j \in \Z$, on rappelle que l'on a noté $e_j = e_{\chi^j}$ une base du module $L(\chi^j)$, de sorte que $\mathbf{e}^{\rm dR}_j =  t^{-j} e_j$ est une base de $\DdR(\Robba(\chi^j))$. Notons \[ \mathbf{e}_{\eta, j} = e_\eta \otimes e_j, \] qui est une base de $L(\eta \chi^j)$. L'élément \[ \mathbf{e}^{\rm dR}_{\eta, j} =  \mathbf{e}^{\rm dR}_\eta \otimes \mathbf{e}^{\rm dR}_j = G(\eta) t^{-j} \cdot \mathbf{e}_{\eta, j} = G(\eta) e_\eta \otimes t^{-j} e_j \] constitue une base du module $\DdR(\Robba(\eta \chi^j))$.

Si $D \in \Phi\Gamma(\Robba)$ est de Rham, alors $D(\eta \chi^j)$ l'est aussi et on a, par ce qui précède, \[ \DdR(D(\eta \chi^j)) = \Ddif( D \otimes L(\eta \chi^j))^\Gamma = \DdR(D) \otimes L \cdot \mathbf{e}^{\rm dR}_{\eta, j}, \] de sorte que l'application $x \mapsto x \otimes \mathbf{e}^{\rm dR}_{\eta, j}$ induit un isomorphisme $\DdR(D) \xrightarrow{\sim} \DdR(D(\eta \chi^j))$.

\subsubsection{Bases et duaux} Enfin, on note $e_\eta^\vee = e_{\eta^{-1}}$, $e_j^\vee = e_{-j}$ les éléments duaux, respectivement, de $e_\eta$ et $e_j$, ainsi que \[ \mathbf{e}^{\rm dR, \vee}_{\eta, j} = G(\eta)^{-1} \cdot e_\eta^\vee \otimes t^j e_j^\vee, \] base du module $\DdR(\Robba(\eta^{-1} \chi^{-j}))$. L'application $x \mapsto x \otimes \mathbf{e}^{\rm dR, \vee}_{\eta, j}$ induit un isomorphisme $\DdR(D(\eta \chi^j)^*) \xrightarrow{\sim} \DdR(D)$ et on a $x \otimes \mathbf{e}^{\rm dR}_{\eta, j} \otimes \mathbf{e}^{\rm dR, \vee}_{\eta, j} = x$ pour tout $x \in \DdR(D)$.

%Si $\eta \colon \zpe \to L^\times$ est un caractère localement constant, vu comme un caractère de $\qpe$ en posant $\eta(p) = 1$, rappelons que l'on a un générateur $\mathbf{e}^{\rm dR}_\eta = G(\eta) e_\eta$ du module $\DdR(\Robba(\eta))$, que $e_\eta^\vee = e_{\eta^{-1}}$ dénote la base de $L(\eta)^*$ duale de $e_\eta$, et que $\DdR(\Robba(\eta)^*) =  L \cdot G(\eta)^{-1} e_\eta^\vee = L \cdot G(\eta^{-1}) e_{\eta^{-1}}$. Ceci fournit deux bases $\mathbf{e}^{\rm dR, \vee}_\eta = G(\eta)^{-1} e_\eta^\vee$ et $\mathbf{e}^{\rm dR}_{\eta^{-1}} = G(\eta^{-1}) e_{\eta^{-1}}$ du module $\DdR(\Robba(\eta)^*) = \DdR(\Robba(\eta^{-1}))$, reliées par la formule $p^n \eta(-1) \mathbf{e}^{\rm dR, \vee}_\eta = \mathbf{e}^{\rm dR}_{\eta^{-1}}$. 

Fixons une base $f_1, f_2$ de $\DdR(D)$ et notons $$ \langle \;,\; \rangle_{\rm dR} \colon \DdR(D) \times \DdR(D) \to L, $$ le produit scalaire défini par la formule $\langle a_1 f_1 + a_2 f_2, b_1 f_1 + b_2 f_2 \rangle_{\rm dR} = a_1 b_1 + a_2 b_2$.

L'isomorphisme $\wedge^2 D = (\Robba \frac{dT}{1 + T}) \otimes \omega_D$ induit un isomorphisme $\wedge^2 \DdR(D) = \DdR((\Robba \frac{dT}{1 + T}) \otimes \omega_D) = (t^{-k} L_\infty e_D)^\Gamma$. On définit $\Omega \in L_\infty$ par la formule $f_1 \wedge f_2 = (t^k \Omega)^{-1} e_{D}$, ce qui nous permet de fixer les bases $(t^k \Omega)^{-1} e_D$ et $t^k \Omega e_D^\vee$ du module $\DdR(\wedge^2 D)$ et de son dual. On fixe aussi les bases $\mathrm{e}^{\rm dR}_{\omega_D} = (t^{k-1} \Omega)^{-1} e_{\omega_D}$ et $\mathbf{e}^{\rm dR, \vee}_{\omega_D} = (t^{k-1} \Omega) e_{\omega_D}^\vee$ du module $\DdR(\Robba(\omega_D))$ et de son dual.

%Avec les notations introduites ci-dessus, on a des isomorphismes $$ \DdR(D) \xrightarrow{\sim} \DdR(\check{D}(\eta \chi^{-{\j}})); \;\;\; x \mapsto x \otimes G(\eta) e_{\eta} \otimes t^{{\j}} e_{-{\j}} \otimes  \Omega t^{k-1}  e_{\omega_D}^\vee $$ $$ \DdR(D) \xrightarrow{\sim} \DdR(D(\eta^{-1} \chi^{\j})) ; \;\;\; x \mapsto x \otimes G(\eta^{-1}) e_{\eta^{-1}} \otimes t^{-{\j}} e_{\j}. $$ 

Afin d'alléger les notations dans les calculs futurs, notons, pour $\eta$ comme ci-dessus et $j \in \Z$, $$ \mathbf{e}_{\eta, j, \omega_D^\vee} = e_\eta \otimes e_j \otimes e_{\omega_D}^\vee, \; \; \; \mathbf{e}_{\eta, j} = e_\eta \otimes e_j, $$ bases de $L(\eta \chi^j \omega_D^{-1})$ et $L(\eta \chi^j)$ respectivement, et leurs duales $$ \mathbf{e}^\vee_{\eta, j, \omega_D^\vee} = e_\eta^\vee \otimes e_{-j} \otimes e_{\omega_D}, \; \; \; \mathbf{e}^\vee_{\eta, j} = e_\eta^\vee \otimes e_{-j}, $$ ainsi que des bases des module $\DdR(\Robba(\eta \chi^j \omega_D^{-1}))$ et $\DdR(\Robba(\eta \chi^j))$ $$ \mathbf{e}^{\rm dR}_{\eta, j, \omega_D^\vee} = G(\eta) e_\eta \otimes t^{-j} e_j \otimes t^{k-1} \Omega e_{\omega_D}^\vee = \mathbf{e}^{\rm dR}_\eta \otimes \mathbf{e}^{\rm dR}_j \otimes \mathbf{e}^{\rm dR, \vee}_{\omega_D}, $$ $$ \mathbf{e}^{\rm dR}_{\eta, j} = G(\eta) e_\eta \otimes t^{-j} e_j = \mathbf{e}^{\rm dR}_\eta \otimes \mathbf{e}^{\rm dR}_j $$ et leurs duales $$ \mathbf{e}^{\rm dR, \vee}_{\eta, j, \omega_D^\vee} = G(\eta)^{-1} e_\eta^\vee \otimes t^{j} e_{-j} \otimes (t^{k-1} \Omega)^{-1} e_{\omega_D} = \mathbf{e}^{\rm dR, \vee}_\eta \otimes \mathbf{e}^{\rm dR}_{-j} \otimes \mathbf{e}^{\rm dR, \vee}_{\omega_D^\vee}, $$ $$ \mathbf{e}^{\rm dR, \vee}_{\eta, j} = G(\eta)^{-1} e_\eta^\vee \otimes t^j e_{-j} = \mathbf{e}^{\rm dR, \vee}_\eta \otimes \mathbf{e}^{\rm dR}_{-j}, $$ et les variantes évidentes que l'on puisse imaginer.

Par exemple, si $\eta \colon \zpe \to L^\times$ est un caractère d'ordre fini, si $j \geq 0$ et si $x \in \DdR(\check{D}(\eta \chi^{-{\j}}))$, on écrira $x \otimes \mathbf{e}^{\rm dR, \vee}_{\eta, -j, \omega_D^\vee} \in \DdR(D)$ l'image de $x$ par l'isomorphisme $$ \DdR(\check{D}(\eta \chi^{-{\j}})) \xrightarrow{\sim} \DdR(D); \;\;\; x \mapsto x \otimes G(\eta)^{-1} e_\eta^\vee \otimes t^{-j} e_{j} \otimes (t^{k - 1} \Omega)^{-1} e_{\omega_D} $$ et de même, si $x \in \DdR(D(\eta^{-1} \chi^{\j}))$, on notera $x \otimes \mathbf{e}^{\rm dR, \vee}_{\eta^{-1}, j} \in \DdR(D)$ l'image de $x$ par l'isomorphisme $$ \DdR(D(\eta^{-1} \chi^j)) \xrightarrow{\sim} \DdR(D); \;\;\; x \mapsto x \otimes G(\eta^{-1})^{-1} e_{\eta^{-1}}^\vee \otimes t^j e_{-j}. $$

\begin{remarque} \leavevmode
\begin{itemize}
\item Les bases des modules de de Rham, vues comme des éléments dans $L_\infty((t)) \cdot e_{\xi}$ pour un certain caractère $\xi \colon \qpe \to L^\times$, héritent une action de l'opérateur $\varphi$ (y agissant sur $L_\infty((t))$ $L_\infty$-linéairement et via $\varphi(t) = p t$ et sur $e_\xi$ par $\varphi(e_\xi) = \xi(p) e_\xi$). On a, par exemple, \[ \varphi(\mathbf{e}^{\rm dR}_\eta) = \mathbf{e}^{\rm dR}_\eta, \;\;\; \varphi(\mathbf{e}^{\rm dR}_j) = p^{-j} \mathbf{e}^{\rm dR}_j, \] \[ \varphi(\mathbf{e}^{\rm dR}_{\eta, j, \omega_D^\vee}) = p^{-j + k - 1} \omega_D^{-1}(p) \cdot \mathbf{e}^{\rm dR}_{\eta, j,\omega_D^\vee} = p^{-j - 1} \omega_\Delta^{-1}(p) \cdot \mathbf{e}^{\rm dR}_{\eta, j,\omega_D^\vee}. \]
\item Il faut faire un peu d'attention et distinguer le caractère identité $x$ et le caractère cyclotomique $\chi = x |x|$. Les deux coïncident sur $\zpe$ mais $\chi(p) = 1$, tandis que le premier prend la valeur $p$. Par exemple, $\Gamma$ agit trivialement sur l'élément $e_j \otimes e_{x^j}^\vee$ mais $\varphi(e_j \otimes e_{x^j}^\vee) = p^{-j} \; e_j \otimes e_{x^j}^\vee$.
\end{itemize}
\end{remarque}

\subsection{Représentations lisses de $\mathrm{GL}_2(\qp)$}

Dans cette section, on énonce quelques résultats classiques de la théorie des représentations lisses de $\mathrm{GL}_2(\qp)$ dont on aura besoin dans la suite. La référence principale est \cite{BH}.

\subsubsection{Facteurs epsilon pour $\mathrm{GL}_1$} \label{replisses}

Commençons par rappeler la définition des facteurs locaux associés à un caractère. Soit $\eta \colon \qpe \to L^\times$ un caractère continu. On dit que $\eta$ est non ramifié si sa restriction à $\zpe$ est triviale et il est ramifié dans le cas contraire. On définit son conducteur par $1$ s'il est non ramifié, et par $p^n$, où $n$ est le plus petit entier tel que la restriction $\eta \mid_{1 + p^n \zp}$ soit triviale, dans le cas contraire. Le choix d'un système compatible $(\zeta_{p^n})_{n \geq 0}$ de racines de l'unité nous permet de fixer un caractère additif $\psi \in \mathrm{Hom}(\qp, L^\times_\infty)$ de niveau $0$ (i.e $\ker \psi= \zp$) par la formule $\psi(x) = \zeta_{p^n}^{p^n x}$ pour n'importe quel $n \geq - v_p(x)$. Fixons aussi $\mu$ la mesure de Haar sur $\qp$ telle que $\mu(\zp) = 1$. Soit $\mu^*$ une mesure de Haar de $\qpe$. Pour $\phi \in \mathrm{LC}_{\rm c}(\qp, L)$ une fonction localement constante à support compact dans $\qp$, la fonction \footnote{On fixe un isomorphisme $\overline{\Q}_p \cong \C$, de sorte que l'on puisse voir une extension finie $L$ de $\qp$ comme un sous-corps de $\C.$}
$$ \zeta(\phi, \eta, s) = \int_{\qpe} \phi(x) \eta(x) |x|^s \cdot d\mu^*(x) $$
converge pour $\mathrm{Re} \; s \gg 0$ et admet un prolongement analytique à $\C$ tout entier. Les facteurs $L$ et $\epsilon$ associés à $\eta$ sont donnés par les formules \footnote{ cf. \cite[\S 23.4]{BH} pour la première formule et \cite[\S 23.5 Th., Lem. 1]{BH} pour se ramener au cas du caractère $\psi$ de niveau zero pour la deuxième.} $$ L(\eta, s) = \left\{
  \begin{array}{cc}
    (1 - \eta(p) p^{-s})^{-1} & \quad \text{ si $\eta$ n'est pas ramifié} \\
    1 & \quad \text{si $\eta$ est ramifié}  \\
  \end{array}
\right. $$

$$ \epsilon(\eta, s) = \left\{
  \begin{array}{cc}
    1 & \quad \text{ si $\eta$ n'est pas ramifié} \\
    p^{-ns} \eta(p)^n G(\eta^{-1}) & \quad \text{si $\eta$ est ramifié}  \\
  \end{array}
\right. $$
Le facteur epsilon satisfait l'équation fonctionnelle $$ \epsilon(\eta, s) \epsilon(\eta^{-1}, 1 - s) = \eta(-1). $$ Enfin, on a une équation fonctionnelle pour tout $\phi \in \mathrm{LC}_{\rm c}(\qp, L)$, où les deux membres sont des polynômes en $p^{-s}$, $$ \frac{\zeta(\hat{\phi}, \eta^{-1}, 1 - s)}{L(\eta^{-1}, 1 - s)} = \epsilon(\eta, s) \frac{\zeta(\phi, \eta, s)}{L(\eta, s)}, $$ où $\hat{\phi} = \int_\qp \phi(y) \psi(xy) d \mu(y)$ dénote la transformée de Fourier de $\phi$. On notera dans la suite $\epsilon(\eta) := \epsilon(\eta, 1/2) = p^{-n/2} \eta(p)^n G(\eta^{-1})$. 
 
\subsubsection{Facteurs epsilon pour $\mathrm{GL}_2$} Soit $\pi$ une représentation lisse irréductible de $G = \mathrm{GL}_2(\qp)$ à coefficients dans $\C$ et notons $\check{\pi}$ sa contragrédiente. Notons $K = \mathrm{M}_2(\zp)$ et $A = \mathrm{M}_2(\qp)$, $\mathrm{LC}_{\rm c}(A, L)$ les fonctions localement constantes à support compact dans $A$ et $\mathfrak{C}(\pi)$ l'espace des coefficients de la représentation $\pi$: c'est le $L$ espace vectoriel engendré par les fonctions $g \in G \mapsto \langle \check{v}, g \cdot v \rangle$, $v \in \pi, \check{v} \in \check{\pi}$. Si $\eta \colon \qpe \to L^\times$ est un caractère localement constant, on note $\pi \otimes \eta$ \footnote{On voit $e_\eta$ comme un générateur du $L$ espace vectoriel de dimension $1$ que l'on munit d'une action de $G$ par la formule $g \cdot e_\eta = \eta(\det g) \; e_\eta$. On a un isomorphisme d'espaces vectoriels $v \mapsto c \otimes e_\eta$ de $\pi$ sur $\pi \otimes \eta$.} la tordue de $\pi$ par $\eta$: si $v \in \pi$ et $g \in G$ alors $g (v \otimes e_\eta) = \eta(\det g) (gv \otimes e_\eta)$. On suppose dans la suite que \textit{$\pi$ est une représentation supercuspidale} \footnote{Rappelons qu'une représentation lisse irréductible $\pi$ de $\mathrm{GL}_2(\qp)$ est dite supercuspidale si elle n'est pas une sous-représentation d'une induite parabolique ou, ce qui revient au même (\cite[\S 9.1 Prop.]{BH}), si son module de Jacquet $V_N := V / V(N)$, où $V$ dénote l'espace ambiant de $\pi$ et $V(N)$ est le sous espace engendré par les vecteurs de la forme $v - {\matrice 1 x 0 1} v$, $x \in \qp$, est nul.}.

Fixons $d\mu$ la mesure de Haar sur $A$ normalisée par $\mu(K) = 1$. La transformée de Fourier, définie par $$ \phi \in \mathrm{LC}_{\rm c}(A, L) \mapsto \hat{\phi}(g) = \int_A \phi(h) \psi(\mathrm{tr}(gh)) d \mu(h), $$ est alors auto-duale: $\hat{\hat{\phi}}(x) = \phi(-x)$. Enfin, si $f \in \mathfrak{C}(\pi)$, la formule $\check{f}(g) = f(g^{-1})$ définit un élément $\check{f}$ dans $\mathfrak{C}(\hat{\pi})$. La fonction $$ \zeta(\phi, f, s) = \int_G \phi(g) f(g) || \det g ||^s \cdot d\mu^*(g), \;\;\; \phi \in \mathrm{LC}_{\rm c}, f \in \mathfrak{C}(\pi), $$ où $\mu^*$ est une mesure de Haar sur $G$, converge pour $\mathrm{Re} \; s \gg 0$ et définit une fonction rationnelle en la variable $p^{-s}$ (\cite[\S 24.2 Th. 1]{BH}).

Il existe (\cite[\S 24.2 Th. 2 + Corollary]{BH}) une constante $\epsilon(\pi) \in L$ tel que, si l'on pose $\epsilon(\pi, s) = p^{- c(\pi) (s-1/2)} \epsilon(\pi), $ où $c(\pi)$ est le conducteur de $\pi$ \footnote{Le niveau $c(\pi)$ de $\pi$ est défini comme le plus petit entier $n$ tel que $\pi$ possède un élément fixe par les matrices de la forme $K_n = \{ {\matrice a b c d} \in K$, $c = d - 1 = 0$ mod $p^n \}$. On a $\dim_L \pi^{K_{c(\pi)}} = 1$ et on appelle $p^{c(\pi)}$ le \textit{conducteur} de la représentation $\pi$. }, alors pour tous $\phi \in \mathrm{LC}_{\rm c}(A, L)$ et $f \in \mathfrak{C}(\pi)$, on a une équation fonctionnelle $$ \zeta(\hat{\phi}, \check{f}, \frac{3}{2} - s) = \epsilon(\pi, s) \zeta(\phi, f, \frac{1}{2} + s). $$ Observons que, si $j \in \Z$, alors $\epsilon(\pi, j + 1/2) = p^{-c(\pi) j} \epsilon(\pi) = \epsilon(\pi \otimes | \cdot |^j)$. La fonction $\epsilon(\pi, s)$ est le facteur epsilon associé à la représentation $\pi$ et elle satisfait l'équation fonctionnelle $$ \epsilon(\pi, s) \epsilon(\check{\pi}, 1 - s) = \omega_\pi(-1), $$ où $\omega_\pi$ est le caractère central de la représentation $\pi$.

%\begin{proposition} \label{BHepsilon} Soit $\eta$ un caractère de $\qpe$ de conducteur $p^n$ pour $n > c(\pi)$ et soit $c_\eta \in \qpe$ satisfaisant $\eta(1 + x) = \psi(c_\eta x)$ pour $x \in p^{\left\lfloor \frac{n}{2} \right \rfloor + 1}$. Alors $$ \epsilon(\pi \otimes \eta, s) = \omega_\pi(c_\eta)^{-1} \epsilon(\eta, s)^2. $$
%\end{proposition}
%
%\begin{proof}
%Par \cite{BH}, 25.7 thm., on a $\epsilon(\pi \otimes \eta^{-1}, s) = \omega_\pi(c_\eta)^{-1} \epsilon(\eta \circ \det, s)$. Or, la représentation $\eta \circ \det$ apparaît comme facteur irréductible de l'induite unitaire lisse $\mathbf{\iota}_B^G(|x|^{1/2} \eta \otimes |x|^{-1/2} \eta) = \mathrm{Ind}_B^G(\eta \otimes \eta)$ de la représentation $|x|^{1/2} \eta \otimes |x|^{-1/2} \eta$ de $B$ (cf. \cite{BH}, 9.11), et \cite{BH}, 26.1, thm. implique que $\epsilon(\eta \circ \det, s) = \epsilon(|x|^{1/2} \eta, s) \epsilon(|x|^{-1/2} \eta, s)$. On conclut en remarquant que $\epsilon(|x|^{1/2} \eta, s) = p^{-n/2} \epsilon(\eta, s)$  et $\epsilon(|x|^{-1/2} \eta, s) = p^{-n/2} \epsilon(\eta, s)$, comme on le voit de la formule $\epsilon(|x|^{\pm 1/2} \eta, s) = p^{-ms} \eta(p)^m |p|^{\pm m/2} G(|x|^{\pm 1/2} \eta)$, où $m = \mathrm{cond}(|x|^{\pm 1/2} \eta)$, et en remarquant que $m = \mathrm{cond}(\eta) = n$ et $G(|x|^{\pm 1/2} \eta) = G(\eta)$ car la restriction à $\zpe$ de $|x|^{\pm 1/2}$ est triviale.
%\end{proof}

\subsubsection{Modèle de Kirillov et facteurs locaux} \label{Kirillovaa}

Rappelons qu'un modèle de Kirillov d'une représentation lisse $\pi$ de caractère central $\omega_\pi$ est une injection $B$-équivariante de $\pi$ dans l'espace $\mathrm{LC}_{\rm rc}(\qpe, L_\infty)$ des fonctions localement constantes sur $\qpe$ à support compact dans $\qp$ (i.e $\phi(p^{-n} \zpe) = 0$ pour tout $n \gg 0$) \footnote{$\mathrm{rc}$ dénote "relativement compact".}. L'action du Borel sur ce dernier espace est donnée par la formule $$ ({\matrice a b 0 d} \cdot \phi)(x) = \omega_\pi(d) \psi(bx/d) \phi(ax/d). $$ La décomposition de Bruhat $G = B \cup BwB$ montre que l'action de $G$ est déterminée par celle du Borel et par l'action de l'involution $w = {\matrice 0 1 1 0}$, et c'est un fait remarquable que ce dernier opérateur puisse être décrit en termes des facteurs locaux de la représentation, comme on le verra à continuation.

Si $\eta \colon \qpe \to L^\times$ est un caractère localement constant et $m \in \Z$, on définit une fonction $\xi_{\eta, m} \in \mathrm{LC}_{c}(\qpe, L_\infty)$ par la formule $$ \xi_{\eta, m}(x) = \left\{
  \begin{array}{ccc}
    \eta(x) & \quad \text{ si } v_p(x) = m \\
    0 & \quad \text{sinon}  \\
  \end{array}
\right. $$ 
%Dans la formule, $\eta(x)$ est considéré comme un élément de $L$ (que l'on peut toujours supposer assez grand de sorte qu'il contient ces valeurs) et $G(\eta)^{-1}$ comme élément de $F_\infty$. On a donc $({\matrice a 0 0 1} \xi_{\eta, m})(x) = \xi_{\eta, m}(ax) = G(\eta)^{-1} \eta(ax) = \eta(a) G(\eta)^{-1} \eta(x)= \sigma_a(G(\eta)^{-1} \eta(x))$, d'où $\xi_{\eta, m} \in \mathrm{LA}_{c}(\qpe, L_\infty)^\Gamma$ (elle est en fait localement constante).

\begin{prop} [{\cite[Th. 37.3]{BH}}] \label{BHeqfonct}
Soit $\eta \colon \qpe \to L^\times$ un caractère et soit $m \in \Z$. Alors $$ w \cdot \xi_{\eta, m} = \eta(-1) \epsilon(\pi \otimes \eta^{-1}) \, \xi_{\eta^{-1}\omega_\pi, -c(\pi \otimes \eta^{-1}) - m}. $$
\end{prop}

\begin{proof}
La formule de \cite[Th. 37.3]{BH} donne \[ {\matrice 0 1 {-1} 0} \cdot \xi_{\eta, m} = \epsilon(\pi \circ \eta^{-1}) \; \xi_{\eta^{-1} \omega_\pi, -c(\pi \otimes \eta^{-1}) - m}. \]
Observons que ${\matrice 0 1 {-1} 0} = {\matrice {-1} 0 0 {-1}} {\matrice {-1} 0 0 1} w $, et que les matrices ${\matrice {-1} 0 0 {-1}}$ et ${\matrice {-1} 0 0 1}$ agissent, respectivement, via multiplication par $\omega_\pi(-1)$ et par $({\matrice {-1} 0 0 1} \cdot \phi)(x) = \phi(-x)$, d'où 
\begin{eqnarray*}
w \cdot \xi_{\eta, m} &=& \epsilon(\pi \circ \eta^{-1}) \; {\matrice {-1} 0 0 {-1}} {\matrice {-1} 0 0 1} \cdot \xi_{\eta^{-1} \omega_\pi, -c(\pi \otimes \eta^{-1}) - m} \\
&=& \epsilon(\pi \circ \eta^{-1}) \omega_\pi(-1) \eta(-1) \; {\matrice {-1} 0 0 {-1}} \cdot \xi_{\eta^{-1} \omega_\pi, -c(\pi \otimes \eta^{-1}) - m} \\
&=& \eta(-1) \epsilon(\pi \circ \eta^{-1}) \; \xi_{\eta^{-1} \omega_\pi, -c(\pi \otimes \eta^{-1})}.
\end{eqnarray*}
\end{proof}

\subsection{Autour de la correspondance de Langlands $p$-adique pour $\mathrm{GL}_2(\qp)$} \label{LLpadique}

Dans cette section, on rappelle certains résultats concernant la correspondance de Langlands $p$-adique pour $\mathrm{GL}_2(\qp)$, notamment, la théorie du modèle de Kirillov-Colmez et les techniques de `changement de poids' introduites dans \cite{ColmezPoids}, qui seront cruciales dans la preuve de notre résultat principal.

\subsubsection{La correspondance}

Rappelons brièvement la construction de la correspondance de Langlands $p$-adique (\cite{ColmezMirabolique}, \cite{ColmezPhiGamma}, \cite{ColmezPoids}). Soit $D \in \Phi\Gamma^{\text{ét}}(\mathscr{O}_\E)$ de dimension $2$. Rappelons que $D^\sharp$ désigne le plus grand sous-$\mathscr{O}_L[[T]]$-module compact de $D$ stable par $\psi$ et sur lequel $\psi$ est surjectif (\cite[\S II.4]{ColmezMirabolique}), et $D^\natural$ est le plus petit sous-$\mathscr{O}_L[[T]]$-module compact de $D$ stable par $\psi$ et engendrant $D$ (\cite[\S II.5]{ColmezMirabolique}). Si $D \in \Phi\Gamma^{\text{ét}}(\E)$ et $D_0 \subseteq D$ est un $\mathscr{O}_\E$-réseau stable par $\varphi$ et $\Gamma$, on pose $D^\sharp = L \otimes_{\mathscr{O}_L} D_0^\sharp$, $D^\natural = L \otimes_{\mathscr{O}_L} D_0^\natural$. Si $D$ est absolument irréductible de dimension $\geq 2$, on a (\cite[Cor. II.5.21]{ColmezMirabolique}) $D^\natural = D^\sharp$.

Soit $G = \mathrm{GL}_2(\qp)$. Rappelons (\cite[\S III.1.2]{ColmezMirabolique}) que, si $U \subseteq \zp$ est un ouvert compact, en traduisant les opérations de restriction d'une mesure, on peut définir des opérateurs de restriction $\mathrm{Res}_U$ sur $D$. Ces opérateurs nous fournissent un faisceau ${\matrice {\zp - \{0\}} \zp 0 1}$-équivariant $U \mapsto D \boxtimes U$ sur $\zp$, et on a $D = D \boxtimes \zp$, $D \boxtimes \zpe = D^{\psi = 0}$. Soient $w_D = m_{\omega_D^{-1}} \circ w_* \colon D \boxtimes \zpe \to D \boxtimes \zpe$, où $m_{\omega_D^{-1}}$ est la multiplication par $\omega_D^{-1}$, et $w_* \colon D \boxtimes \zpe \to D \boxtimes \zpe$ sont définies dans \cite[\S V.2.1 et \S V.1]{ColmezPhiGamma}. En s'inspirant des formules provenant de l'analyse $p$-adique, on construit (\cite[\S II.1]{ColmezPhiGamma}) un faisceau $G$-équivariant $U \mapsto D \boxtimes U$ sur $\PP^1(\qp)$ dont les sections globales sont données par $$ D \boxtimes \PP^1 = \{ (z_1, z_2) \in D \times D : w_D(\mathrm{Res}_\zpe(z_1)) = \mathrm{Res}_\zpe(z_2) \} $$ et les sections sur $\zp$ sont $D \boxtimes \zp = D$. On définit (\cite[\S II.2]{ColmezPhiGamma}) le module $D^\natural \boxtimes \PP^1 = \{ z \in D \boxtimes \PP^1: \mathrm{Res}_\zp( {\matrice {p^n} 0 0 1} \cdot z) \in D^\natural \;\;\; \forall n \in \Z \}$. Ce module est stable par l'action de $G$ (\cite[Th. II.3.1(i)]{ColmezPhiGamma}) et on pose $$ \Pi(D) = D \boxtimes \PP^1 / D^\natural \boxtimes \PP^1, $$ qui est une $L$-représentation de Banach admissible et topologiquement irréductible de $G$. La représentation $\Pi(D)$ est celle associée à $D$ par la correspondance de Langlands $p$-adique. L'accouplement $[\; , \;]$ sur $D$ s'étend (\cite[Lem. II.1.12]{ColmezPoids}) en un accouplement parfait et $G$-équivariant $[\; , \; ]_{\PP^1}$ sur $D \boxtimes \PP^1$ en posant
\[ [z, z']_{\PP^1} = [\mathrm{Res}_\zp(z), \mathrm{Res}_\zp (z')] + \omega_\Delta(-1) [\mathrm{Res}_{p \zp} (w \cdot z), \mathrm{Res}_{p \zp} (w \cdot z')], \]
pour lequel $\Pi(D)^* \otimes \omega_D$ est son propre orthogonal (\cite[\S 2.1]{ColmezPoids}). On a (\cite[Th. II.3.1(ii)]{ColmezPhiGamma}) une suite exacte de $G$-modules topologiques
\[ 0 \to \Pi(D)^* \otimes \omega_D \to D \boxtimes \PP^1 \to \Pi(D) \to 0. \]

Soient $D^\dagger \in \Phi \Gamma^{\text{ét}}(\E^\dagger)$ et $D_{\rm rig} = \Robba \otimes_{\E^\dagger} D^\dagger \in \Phi \Gamma^{\text{ét}} (\Robba)$ les modules correspondants à $D$ par l'équivalence de catégories de la proposition \ref{equivPhiGamma} (cf. \cite[\S V.1.1]{ColmezPhiGamma} pour des rappels des constructions). Le sous-module $D^\dagger \boxtimes \PP^1 = \{ (z_1, z_2) \in D \boxtimes \PP^1 : z_1, z_2 \in D^\dagger\}$ de $D \boxtimes \PP^1$ est stable par l'action de $G$ (\cite[Prop. V.2.8]{ColmezPhiGamma}) et l'action de $G$ s'étend par continuité (\cite[Prop. V.2.9]{ColmezPhiGamma}) en une action sur \[ D_{\rm rig} \boxtimes \PP^1 := \{ (z_1, z_2) \in D_{\rm rig} \times D_{\rm rig} : w_D(\mathrm{Res}_\zpe(z_1)) = \mathrm{Res}_\zpe(z_2) \}, \] où l'action de $w_D$ sur $D_{\rm rig} \boxtimes \zpe = D_{\rm rig}^{\psi = 0}$ est définie par $\Gamma$-antilinéarité en utilisant l'identité $D_{\rm rig} \boxtimes \zpe = \Robba(\Gamma) \otimes_{\E^\dagger(\Gamma)} D^\dagger \boxtimes \zpe$ . Ceci permet de définir un faisceau $G$-équivariant $U \mapsto D_{\rm rig} \boxtimes U$ sur $\PP^1(\qp)$ et on a (\cite[Th. V.2.20]{ColmezPhiGamma}) une suite exacte
\[ 0 \to (\Pi(D)^{\rm an})^* \otimes \omega_D \to D_{\rm rig} \boxtimes \PP^1 \to \Pi(D)^{\rm an} \to 0, \]
où $\Pi(D)^{\rm an}$ dénote les vecteurs localement analytiques de la représentation $\Pi(D)$, ce qui rend naturel de poser $\Pi(D_{\rm rig}) = \Pi(D)^{\rm an}$.

Plus généralement, si $D \in \Phi \Gamma(\Robba)$ n'est pas triangulin, il est isocline et il existe un caractère $\delta$ tel que $D(\delta)$ soit étale. On définit alors $U \mapsto D \boxtimes U$ et $\Pi(D)$ par torsion à partir de $D(\delta) \boxtimes U$ et $\Pi(D(\delta))$; en particulier $\Pi(D) = \Pi(D(\delta)) \otimes (\delta^{-1} \circ \det)$. On obtient ainsi une $L$-représentation localement analytique $\Pi = \Pi(D)$ de caractère central $\omega_D$ et une suite exacte comme ci-dessus.

\subsubsection{Le modèle de Kirillov, d'après Colmez} \label{Kirillov}

Décrivons le modèle de Kirillov construit dans \cite[\S 2.4.1-2]{ColmezPoids}. Soit $D \in \Phi \Gamma(\Robba)$ irréductible de rang $2$ et soit $\Pi = \Pi(D)$ la représentation de $G$ décrite dans la section précédente. Notons $B = {\matrice * * 0 *} \subseteq \mathrm{GL}_2(\qp)$ le Borel supérieur et soit $\mathscr{Y}$ un $L_\infty[[t]]$-module muni d'une action de l'algèbre de distributions $\mathscr{D}(\Gamma) = \mathscr{Robba}^+(\Gamma)$ de $\Gamma$. On définit $$ \mathrm{LA}_{\mathrm{rc}}(\qpe, \mathscr{Y})^\Gamma $$ comme l'espace des fonctions $\phi \colon \qpe \to \mathscr{Y}$ localement analytiques à support compact dans $\qp$, vérifiant $\sigma_a(\phi(x)) = \phi(ax)$ pour tout $a \in \zpe$. On munit $\mathrm{LA}_{\mathrm{rc}}(\qp, \mathscr{Y})^\Gamma$ d'une action de $B$ par la formule \footnote{Si $r \in \qp$ et $n \in \N$ est tel que $rp^n \in \zp$, on pose $[\epsilon^{r}] = \varphi^{-n}((1 + T)^{p^n r}) \in L_n[[t]]$.} $$ ({\matrice a b 0 d} \cdot \phi)(x) = \omega_D(d) [\epsilon^{bx / d}] \phi(ax/d). $$

Soit $Y$ un $B$-module de caractère central $\omega_D$. On dit que $Y$ admet un modèle de Kirillov s'il existe un $B$-module $\mathscr{Y}$ comme ci-dessus et une injection $B$-équivariante de $Y$  dans $\mathrm{LA}_{\mathrm{rc}}(\qpe, \mathscr{Y})^\Gamma$. Si $Y$ admet un modèle de Kirillov, on note $Y_{\rm c}$ l'image inverse dans $Y$ du sous-espace des fonctions $\phi \in \mathrm{LA}_{\mathrm{c}}(\qpe, \mathscr{Y})^\Gamma \subseteq \mathrm{LA}_{\mathrm{rc}}(\qpe, \mathscr{Y})^\Gamma$ dont le support est compact dans $\qpe$ ($\phi(p^n \zpe) = 0$ pour tout $|n| \gg 0$).

Rappelons que l'algèbre de Lie $\mathfrak{g} = \mathfrak{gl}_2$ de $G$ agit sur le module $D \boxtimes \PP^1$ (cf. \cite{Dospinescu}) et l'action de $u^+ = {\matrice 0 1 0 0} \in \mathfrak{g}$ sur $D$ est donnée par $ u^+ z = tz$. Notons $\Pi^{u^+-\mathrm{fini}}$ l'ensemble des $v \in \Pi$ tués par une puissance de $u^+$, qui en est un sous $B$-module. Si $v \in \Pi^{u^+-\mathrm{fini}}$ et $\tilde{v} \in D \boxtimes \PP^1$ en est un relèvement, la condition $v \in \Pi^{u^+-\mathrm{fini}}$ se traduit donc par l'existence de $N, k \geq 0$ tels que $({\matrice 1 {p^N} 0 1} - 1)^k \tilde{v} \in \Pi^* \otimes \omega_D$.

Notons $D_{\rm dif} = \Ddif(D)$, $D_{\rm dif}^+ = \Ddifp(D)$ et $D_{\rm dif}^- = D_{\rm dif} / D_{\rm dif}^+$ et rappelons que l'on dispose des applications de localisation $\varphi^{-n} \colon D^{]0, r_n]} \to D_{\rm dif}$. L'image de $\Pi^* \otimes \omega_D$ dans $D$ par l'application $\mathrm{Res}_\zp$ est incluse dans $D^{]0, r_{m(D)}]}$ (en effet, elle est incluse dans $D^{]0, r_a]}$ pour un certain $a$, car elle est l'image d'un Fréchet dans une limite inductive de Fréchets, et stable par $\psi$ car $\Pi$ l'est par ${\matrice {p^{-1}} 0 0 1}$). Ceci nous permet de poser, pour $n \geq m(D)$ et $N, k$ comme ci-dessus, 
\[ \varphi^{-n}(\mathrm{Res}_\zp( {\matrice {{p^j} a} 0 0 1} \tilde{v})) = \frac{1}{\varphi^{N + j - n}(\sigma_a(T))^k} \varphi^{-n}(\mathrm{Res}_\zp({\matrice {{p^j} a} 0 0 1} ({\matrice 1 {p^N} 0 1} - 1)^k \tilde{v})) \in t^{-k} D_{\mathrm{dif}, n}^+. \]
Si $x \in \qpe$, l'image de $\varphi^{-n}(\mathrm{Res}_\zp( {\matrice {{p^n} x} 0 0 1} \tilde{v}))$ dans $D_{\mathrm{dif}}^-$ ne dépend ni du choix de $\tilde{v}$ ni du choix de $n$ assez grand (cf. \cite[\S 2.4.2]{ColmezPoids}). On a donc une application bien définie $$ x \mapsto \mathscr{K}_v(x) \in \mathrm{LA}_{\rm rc}(\qpe, D_{\rm dif}^-)^\Gamma. $$

%Si $x \in \qpe$, l'image de $\varphi^{-n}(\mathrm{Res}_\zp( {\matrice {{p^n} x} 0 0 1} \tilde{v}))$ dans $D_{\mathrm{dif}}^-$ ne dépend ni du choix de $\tilde{v}$ (car un autre relèvement de $v$ dans $D \boxtimes \PP^1$ diffère de $\tilde{v}$ pour un élément dans $\Pi(D)^* \otimes \omega_\Delta$ et l'image par $\varphi^{-n}$ de sa restriction à $\zp$ appartient à $\varphi^{-n}(D^{]0, r_n]}) \subseteq D_{\rm dif}^+$) ni du choix de $n$ assez grand (car l'action de $\Gamma$ est localement analytique). On a donc une application bien définie $$ x \mapsto \mathscr{K}_v(x) \in \mathrm{LA}_{\rm rc}(\qpe, D_{\rm dif}^-)^\Gamma. $$

\begin{prop} [{\cite[Prop. 2.10]{ColmezPoids}}] L'application $\mathscr{K}_v$ ci-dessus définit un modèle de Kirillov pour $\Pi^{u^+ - \text{fini}}$. 
%De plus, l'algèbre enveloppante $U(\mathfrak{gl})$ de $\mathfrak{gl}_2$ stabilise $\Pi^{u^+-\text{fini}}$ et, si $\lambda \in U(\mathfrak{gl})$, alors $$ \mathscr{K}_{\lambda v}(x) = {\matrice x 0 0 1} \lambda {\matrice {x^{-1}} 0 0 1} \cdot \mathfrak{K}_v(x). $$
\end{prop}

\subsubsection{Dualité et modèle de Kirillov} \label{dualKir} Le choix d'un isomorphisme $\wedge^2 D = (\Robba \frac{dT}{1 + T}) \otimes \omega_D$ induit un isomorphisme $\wedge^2 D_{\rm dif} = (L_\infty((t)) dt) \otimes \omega_D$  de $L_\infty((t))$-espaces vectoriels munis d'une action de $\Gamma$ \footnote{L'action de $\Gamma$ sur $dt$ est donnée par $\sigma_a(dt) = a \, dt$.}. On note 
\[ [\;,\;]_{\rm dif} \colon D_{\rm dif} \times D_{\rm dif} \to L \]
l'accouplement défini par la formule $$ [x, y]_{\rm dif} = \text{rés}_L( (\sigma_{-1} (x) \wedge y) \otimes e_{\omega_D}^\vee ), $$ où $\text{rés}_L \colon L_\infty((t)) dt \to L$ est défini comme la composée de l'application résidu $L_\infty((t)) \to L_\infty : \sum_{n \in \Z} a_n t^n \mapsto a_{-1}$ et la trace de Tate normalisée $T_L := \lim_{n \to +\infty} \frac{1}{[L_n:L]} \mathrm{Tr}_{L_n / L} \colon L_\infty \to L$. L'accouplement $[\;,\;]_{\rm dif}$ satisfait $[\sigma_a x, \sigma_a y]_{\rm dif} = \omega_D(a) [x, y]_{\rm dif}$ et on a donc aussi 
\begin{equation} \label{accdif}
[x, y]_{\rm dif} = \omega_D(-1) \text{rés}_L( (x  \wedge \sigma_{-1} (y)) \otimes e_{\omega_D}^\vee ).
\end{equation}
Comme $[D_{\rm dif}^+, D_{\rm dif}^+]_{\rm dif} = 0$, l'accouplement $[\;,\;]_{\rm dif}$ induit un accouplement $$ [\;,\;]_{\rm dif} \colon D_{\rm dif}^+ \times D_{\rm dif}^- \to L. $$

Si $\phi \in \mathrm{LA}_{\rm c}(\qpe, D_{\rm dif}^-)^\Gamma$, $z \in \Pi^* \otimes \omega_D$ et $N \gg 0$, la formule (cf. \cite[\S 2.4.3]{ColmezPoids}) 
\begin{equation} \label{formulemagique} [z, \phi] = \sum_{i \in \Z} \omega_D(p^{-i}) [\varphi^{-N} \big( \mathrm{Res}_\zp \big( {\matrice {p^{i + N}} 0 0 1} z \big) \big), \phi(p^i)]_{\rm dif}
\end{equation} est bien définie pour $N \geq m(D)$, ne dépend pas du choix $N$ et définit une forme linéaire continue (car la somme est finie, $\phi$ étant à support compact) sur $\Pi^* \otimes \omega_D$ et fournit donc un plongement $ \iota \colon \mathrm{LA}_{\rm c}(\qpe, D_{\rm dif}^-)^\Gamma \to \Pi $ caractérisé par $$ [z, \iota(\phi)]_{\PP^1} = [z, \phi] $$ pour tout $z \in \Pi^* \otimes \omega_D$ (cf. \cite[\S 2.4.3]{ColmezPoids}).

\begin{prop} [{\cite[Prop. $2.13$]{ColmezPoids}}]
L'image de $\iota$ est incluse dans $\Pi^{u^+-\text{fini}}$ et la composition \[ \mathscr{K} \circ \iota \colon \mathrm{LA}_{\rm c}(\qpe, D_{\rm dif}^-)^\Gamma \to \mathrm{LA}_{\rm rc}(\qpe, D_{\rm dif}^-)^\Gamma \] est l'inclusion naturelle.
\end{prop}

Terminons en remarquant (cf. \cite[Rem. 2.14]{ColmezPoids}) que, si $D$ est de Rham non triangulin, alors $$ \Pi(D)^{u^+-\text{fini}} = \mathrm{LA}_{\rm c}(\qpe, D_{\rm dif}^-)^\Gamma. $$ 

\subsubsection{$(\varphi, \Gamma)$-modules de de Rham non triangulins de dimension $2$} \label{phiGammadR}

Rappelons que, d'après \S \ref{eqdiff}, on a une recette permettant de reconstruire, à partir de l'équation différentielle $\Nrig(D)$ et de la filtration de Hodge, tout $(\varphi, \Gamma)$-module $D \in \Phi \Gamma(\Robba)$ de rang $2$ qui est de Rham à poids de Hodge-Tate $0$ et $k \geq 1$. Précisément, on a
\[ D = \{ z \in \Nrig(D): \varphi^{-n}(z) \in L_n[[t]] \otimes \mathscr{L} + t^k L_n[[t]] \otimes \DdR(D) \text{ pour } n \gg 0 \}, \] où $\mathscr{L} := \mathrm{Fil}^0 \, \DdR(D) \subseteq \DdR(D)$ dénote la droite de la filtration de Hodge de $D$.

%Soient $K$ une extension finie de $\qp$ et $K_0 = K \cap \Q^{\rm nr}_p$ la sous extension maximale non ramifiée de $K$. Soit $M$ un $(\varphi, \mathscr{G}_\qp)$-module irréductible de rang $2$: $M$ est un $L \otimes K_0$-module libre de rang $2$ muni d'une action semi-linéaire de $\varphi$ et d'une action semi-linéaire de $\mathscr{G}_\qp$ (agissant à travers le groupe fini $\mathscr{G}_\qp / \mathscr{G}_K$), commutant entre elles et telle que l'action du groupe d'inertie de $\mathscr{G}_\qp$ soit absolument irréductible. Soit 
%\[ \Delta = \Delta(M) = (\Robba_K \otimes_{K_0} M)^{\mathrm{Gal}(K_\infty / \qp(\mu_{p^\infty}))}, \]
%qui est un $(\varphi, \Gamma)$-module de rang $2$ sur $\Robba$, de Rham à poids de Hodge-Tate $0$ et $0$.

Réciproquement, soit $\Delta \in \Phi \Gamma(\Robba)$ de rang $2$, de Rham, non triangulin à poids de Hodge-Tate $0$ et $0$, et notons dorénavant $$ M_{\rm dR} := \DdR(\Delta), $$ qui est un $L$-espace vectoriel de dimension $2$ (sans filtration). On a $\D_{\rm dif, n}^+(\Delta) = L_n[[t]] \otimes M_{\rm dR}$ pour $n \geq m(\Delta)$. Si $k \geq 1$ et $\mathscr{L} \subseteq M_{\rm dR}$ est une droite, on pose 
\[ \Delta_{k, \mathscr{L}} = \{ z \in \Delta: \varphi^{-n}(z) \in L_n[[t]] \otimes \mathscr{L} + t^k L_n[[t]] \otimes M_{\rm dR} \text{ pour } n \gg 0 \}, \]
qui est un $(\varphi, \Gamma)$-module de rang $2$ sur $\Robba$, de Rham à poids de Hodge-Tate $0$ et $k$ tel que $\Nrig(\Delta_{k, \mathscr{L}}) = \Delta$. Tout $(\varphi, \Gamma)$-module $D$ de rang $2$ sur $\Robba$, de Rham, pas triangulin et à poids de Hodge-Tate distincts est (\cite[\S 3.1]{ColmezPoids}), à torsion près par une puissance du caractère cyclotomique, de la forme $\Delta_{k, \mathscr{L}}$ pour un unique choix de $\Delta$, $k$ et $\mathscr{L}$.

\begin{remarque} \label{dpst}
D'après \cite[Th. V.2.3]{Berger08} (cf. aussi \cite[\S 3.1]{ColmezPoids}), la donnée de $\Delta = \Nrig(D)$ est équivalente à la donnée d'un $(\varphi, N, \mathscr{G}_\qp)$-module $M$ (dit \textit{l'espace de solutions} de $\Delta$, dans la terminologie de \textit{loc. cit.}). Si $D$ est étale et $V = \mathbf{V}(D) \in \mathrm{Rep}_L \mathscr{G}_\qp$ est la représentation galoisienne associée à $D$ par l'équivalence de catégories de la proposition \ref{equivPhiGamma}, alors le $(\varphi, N, \mathscr{G}_\qp)$-module $M$ en question n'est rien d'autre que le module potentiellement semi-stable $\D_{\rm pst}(V)$ de $V$ de la théorie de Hodge $p$-adique. Le $(\varphi, \Gamma)$-module $\Delta$ est non triangulin si et seulement si l'action de l'inertie de $\mathscr{G}_\qp$ sur $M$ est absolument irréductible.
\end{remarque}

%Notons $$ M_{\rm dR} = (K \otimes_{K_0} M)^{\mathscr{G}_\qp}, $$ qui est un $L$-espace vectoriel de dimension $2$. On a $\D_{\rm dif, n}^+(\Delta) = L_n[[t]] \otimes M_{\rm dR}$ pour $n \geq m(\Delta)$. Si $k \geq 1$ et $\mathscr{L} \subseteq M_{\rm dR}$ est une droite, on pose $$ \Delta_{k, \mathscr{L}} = \{ z \in \Delta[1/t]: \varphi^{-n}(z) \in L_n[[t]] \otimes \mathscr{L} + t^k L_n[[t]] \otimes M_{\rm dR} \text{ pour } n \gg 0 \}, $$ qui est un $(\varphi, \Gamma)$-module de rang $2$ sur $\Robba$, de Rham à poids de Hodge-Tate $0$ et $k$.

%Soit $D \in \Phi\Gamma(\Robba)$ de Rham, non-triangulin, de dimension $2$ et à poids de Hodge-Tate $0, k \geq 1$, et notons $\Delta = \Nrig(D)$, $M_{\rm dR} = \DdR(D)$. Comme on l'a déjà remarqué (cf. \ref{eqdiff}), le $(\varphi, \Gamma)$-module $D$ est uniquement déterminé par la donnée de son deuxième poids $k \geq 1$ et de sa filtration de Hodge $\mathscr{L} \subseteq M_{\rm dR}$ par la formule $$ D = \{ z \in \Delta[1/t]: \varphi^{-n}(z) \in L_n[[t]] \otimes \mathscr{L} + t^k L_n[[t]] \otimes M_{\rm dR} \text{ pour } n \gg 0 \}. $$ Notons, pour $\Delta$, $\mathscr{L}$ et $k$ comme ci-dessus, $\Delta_{k, \mathscr{L}}$ le $(\varphi, \Gamma)$-module (de dimension $2$, de Rham et à poids de Hodge-Tate $0, k \geq 1$) que l'on obtient en appliquant cette recette.

\subsubsection{Techniques de changement de poids} \label{Changement}

Nous aurons besoin dans la suite de certains résultats de Colmez (\cite{ColmezPoids}) concernant l'étude des vecteurs localement algébriques des représentations de $\mathrm{GL}_2(\qp)$ associées aux représentations de de Rham par la correspondance de Langlands $p$-adique.

Soit $\Delta \in \Phi \Gamma(\Robba)$ de rang $2$, de Rham, non triangulin à poids de Hodge-Tate nuls comme dans \S \ref{phiGammadR}. Dans \cite{ColmezPoids} (cf., par exemple, \cite[Th. 0.6]{ColmezPoids}), on construit, pour $k \in \Z$, une représentation localement analytique irréductible $\Pi(\Delta, k)$ de $G$, à caractère central $x^k \omega_\Delta$, en tordant convenablement l'action de $G$ sur le module $\Delta \boxtimes \PP^1$. Plus précisément, si $a, c \in \qp$ ne sont pas tous les deux nuls, on montre que l'opérateur $\partial$ sur $\Delta$ s'étend en un opérateur sur $\Delta \boxtimes \PP^1$ (\cite[Lem. 3.3(i)]{ColmezPoids}) et que $c \partial + a$ y est bijectif (\cite[Rem. 3.5]{ColmezPoids}) et on définit (\cite[Lem. 3.9]{ColmezPoids}), pour n'importe quel entier $k \in \Z$, une action de $G$ sur $\Delta \boxtimes \PP^1$ par la formule \footnote{Dans \cite{ColmezPoids}, la formule donnée pour l'action de $G$ est $ {\matrice a b c d} \ast_k v = (-c \partial + a)^k \cdot ({\matrice a b c d } \cdot v)$. Le changement de signe ci-dessous est justifié par les différentes normalisations entre \cite{ColmezPoids} et ce travail (qui reprend plutôt les normalisations de \cite{ColmezPhiGamma}!) pour la représentation $\mathrm{Sym}^{k-1}$. Ce changement élimine certaines nuisances des signes. } $$ {\matrice a b c d} \ast_k v = (c \partial + a)^k \cdot ({\matrice a b c d } \cdot v). $$ Par exemple, si $w = {\matrice 0  1 1 0}$ dénote l'involution, on a \[ w \ast_k v = \partial^k \cdot (w \cdot v). \] On note $(\Delta \boxtimes \PP^1)[k]$ le $G$-module $\Delta \boxtimes \PP^1$ muni de l'action $\ast_k$. Le sous-module $\Pi(\Delta)^* \otimes \omega_\Delta$ de $\Delta \boxtimes \PP^1$ est stable par l'action $\ast_k$ de $G$ (\cite[Corollaire 3.17]{ColmezPoids}) et on note $\Pi(\Delta, k)$ le quotient de $(\Delta \boxtimes \PP^1)[k]$ par $\Pi(\Delta)^* \otimes \omega_\Delta$, qui est une représentation de $G$ de type analytique \footnote{Une représentation de type analytique de $G$ est un espace de type \textrm{LF} (limite inductive d'espaces de Fréchet), muni d'une action continue de $H$ qui s'étend en une action de l'algèbre $\mathscr{D}(H)$ des distributions sur $H$.}.

La représentation $\Pi(\Delta, k)$ vérifie les propriétés suivantes (cf. \cite[Lem. 3.19]{ColmezPoids} pour le premier point et \cite[Th. 3.31]{ColmezPoids} pour les autres):
\begin{itemize}
\item $\Pi(\Delta)^* \otimes \omega_\Delta$, vu comme sous-module de $(\Delta \boxtimes \PP^1)[k]$, est isomorphe à $\Pi(\Delta, -k)^* \otimes \omega_\Delta$, d'où une suite exacte de $G$-modules $$ 0 \to \Pi(\Delta, -k)^* \otimes \omega_\Delta \to (\Delta \boxtimes \PP^1)[k] \to \Pi(\Delta, k) \to 0. $$
%\item Les injections naturelles $t^k \Delta \boxtimes_{x^k \omega_\Delta} \PP^1 \subseteq D_i \boxtimes_{\delta_{D_i}} \PP^1$ et $\D_i \boxtimes_{\delta_{D_i}} \PP^1$ envoient $\Pi(\Delta, -k)^* \otimes x^k \omega_\Delta$ dans $\Pi(D_i)^* \otimes \delta_{D_i}$ et $\Pi(D_i)^* \otimes \delta_{D_i}$ dans $\Pi(\Delta, k)^* \otimes x^k \omega_\Delta$.
\item Il existe une représentation lisse  $\mathrm{LL}_p(\Delta)$ de $G$, qui ne dépend ni du poids ni de la filtration, de caractère central $x \omega_\Delta$, telle que $$ \Pi(\Delta, k)^{\rm alg} = (\mathrm{LL}_p(\Delta) \otimes \mathrm{Sym}^{k - 1}) \otimes M_{\rm dR}, $$ où $\Pi(\Delta, k)^{\rm alg}$ dénote les vecteurs localement algébriques de la représentation $\Pi(\Delta, k)$, et $\mathrm{Sym}^{k-1}$ est la puissance symétrique de la représentation standard de $G$ (cf. remarque \ref{kiralg} ci-dessous).
\item $\Pi(\Delta_{k, \mathscr{L}}) = \Pi(\Delta, k) / ((\mathrm{LL}_p(\Delta) \otimes \mathrm{Sym}^{k - 1}) \otimes \mathscr{L}) $.
\item $\Pi(\Delta_{k, \mathscr{L}})^{\rm alg} = (\mathrm{LL}_p(\Delta) \otimes \mathrm{Sym}^{k - 1}) \otimes (M_{\rm dR} / \mathscr{L})$.
\end{itemize}

\subsubsection{Modèles de Kirillov} \label{Kirillovtordu}

Les représentations $\Pi(\Delta)$ et $\Pi(\Delta, k)$ sont isomorphes, à torsion par un caractère près, en tant que $B$-représentations et l'application $v \mapsto \mathscr{K}_v$ de la section précédente fournit donc un modèle de Kirillov pour $\Pi(\Delta, k)^{u^+-\text{fini}}$ dans $\mathrm{LA}_{\rm rc}(\qpe, \Delta_{\rm dif}^-)^\Gamma$ muni d'une action de $B$ définie par la formule
\begin{equation} \label{actionk}
({\matrice a b 0 d} \ast_k \phi)(x) = a^k \omega_\Delta(d) [\epsilon^{bx / d}] \phi(ax/d).
\end{equation} Ceci induit aussi des modèles de Kirillov pour $\Pi(\Delta, k)^{\rm alg}$, ainsi que pour $\mathrm{LL}_p(\Delta)$ et $\Pi(D)^{\rm alg}$.

Soient $\mathscr{L}_1, \mathscr{L}_2 \subseteq M_{\rm dR}$ deux droites distinctes et soit $W_{\mathscr{L}_i} = \ker(\Pi(\Delta, k) \to \Pi(\Delta_{k, \mathscr{L}_i})) = (\mathrm{LL}_p(\Delta) \otimes \mathrm{Sym}^{k - 1}) \otimes \mathscr{L}_i \subseteq \Pi(\Delta, k)^{\rm alg}$, qui est une représentation localement algébrique de $G$. On a $\Pi(\Delta, k)^{\rm alg} = W_{\mathscr{L}_1} \oplus W_{\mathscr{L}_2}$ et $W_{\mathscr{L}_i} = \mathrm{LL}_p(\Delta) \otimes \mathrm{Sym}^{k - 1} \otimes \mathscr{L}_i$. Enfin, on a le diagramme commutatif suivant, où les flèches horizontales sont des injections et les verticales des isomorphismes de $G$-modules, qui met en scène tous les personnages introduits ci-dessus et qui sera très utile pour nos calculs futurs.

\begin{equation} \label{Kirillovdiag}
\xymatrix{
    W_{\mathscr{L}_i} \ar[r] \ar[d]^\wr & \Pi(\Delta, k)^{\rm alg} \ar[r] \ar[d]^\wr & \Pi(\Delta, k)^{u^+-\text{fini}} \ar[d]^\wr \\
    \mathrm{LA}_{\rm c}(\qpe, L^{-k}_\infty \otimes \mathscr{L}_i) ^\Gamma\ar[r] & \mathrm{LA}_{\rm c}(\qpe, L^{-k}_\infty \otimes M_{\rm dR}) ^\Gamma \ar[r] & \mathrm{LA}_{\rm c}(\qpe, \Delta_{\rm dif}^-) ^\Gamma, }
\end{equation}
où l'on a noté $L^{-k}_\infty = (t^{-k} L_\infty[t] / L_\infty[t])$.

%Remarquons les identités: $$ (t^k \Delta)_{\rm dif}^- = \Delta_{\rm dif} / t^k \Delta^+_{\rm dif} = (L_\infty((t)) / t^k L_\infty[[t]]) \otimes M_{\rm dR}, \;\;\; \Delta_{\rm dif}^+ / t^k \Delta^+_{\rm dif} = (L_\infty[t] / t^k) \otimes M_{\rm dR}. $$

\begin{remarque} \label{kiralg} Si $e_1, e_2$ la base canonique de $L^2$ sur $L$, on a $ \mathrm{Sym}^{k - 1} = \oplus_{j = 0}^{k - 1} L \cdot e_1^j e_2^{k - 1 - j}$ et l'action de $G$ est donnée par $$ {\matrice a b c d} \cdot e_1^j e_2^{k-1-j} = (a e_1 + c e_2)^j (b e_1  + d e_2)^{k-1-j}. $$ En particulier, on a \[ w \cdot e_1^j e_2^{k - 1 - j} = e_1^{k - 1 - j} e_2^j. \] On a (cf. \cite[\S VI.2.5]{ColmezPhiGamma}), pour chaque $i = 1, 2$, un isomorphisme $B$-équivariant \[ \iota_i \colon \mathrm{LC}_{\rm c}(\qpe, L_\infty)^\Gamma \otimes \mathrm{Sym}^{k - 1} \otimes \det^{-k} \xrightarrow{\sim} \mathrm{LA}_{\rm c}(\qpe, L^{-k}_\infty \otimes \mathscr{L}_i)^\Gamma  , \] \[ \phi \otimes e_1^j e_2^{k - 1 - j} \mapsto [x \mapsto (k - 1 - j)! \phi(x) (xt)^{j-k} \otimes f_i], \] où $B$ agit sur $\mathrm{LC}_{\rm c}(\qpe, L_\infty)^\Gamma$ via la formule $$ ({\matrice a b 0 d} \cdot \phi)(x) = (x \omega_\Delta)(d) \psi(bx/d) \phi(ax/d) $$ et sur le module de droite par l'action $\ast_k$ décrite dans la formule \eqref{actionk}. Cet isomorphisme devient un isomorphisme $G$-équivariant si l'on munit ces espaces d'une action de $G$ via les bijections $\mathrm{LL}_p(\Delta) \xrightarrow{\sim} \mathrm{LC}_{\rm c}(\qpe, L_\infty)^\Gamma$ et $W_{\mathscr{L}_i} \xrightarrow{\sim} \mathrm{LA}_{\rm c}(\qpe, L^{-k}_\infty \otimes \mathscr{L}_i)^\Gamma$.
%\footnote{Observons que $\mathrm{LA}_{\rm c}(\qpe, L^{-k}_\infty \otimes \mathscr{L}_i)^\Gamma = \mathrm{LP}_{\rm c}(\qpe, L^{-k}_\infty \otimes \mathscr{L}_i)^\Gamma$ sont des fonctions localement polynomiales.}.
\end{remarque}

\begin{lemme} \leavevmode \label{propinv} Soient $k, j \in \Z$. Alors
\begin{enumerate}
\item On a $w \cdot \partial = \partial^{-1} \cdot w$ sur $\Delta \boxtimes \PP^1$.
\item $[g \ast_{-k} x, y ]_{\PP^1} = \omega_\Delta(\det \; g) [x, g^{-1} \ast_k y ]_{\PP^1}, $ pour $x, y \in \Delta \boxtimes \PP^1$, $g \in G$.
\end{enumerate}
\end{lemme}

\begin{proof}
Le premier point est \cite[Prop. 3.6]{ColmezPoids}. Le deuxième point est une réécriture de  la formule $[g \ast_{-k} x, g \ast_{k} y]_{\PP^1} = [g \cdot x, g \cdot y]_{\PP^1}$ de \cite[Prop. 3.13.(iii)]{ColmezPoids} à l'aide de l'identité $[g \cdot x, g \cdot y]_{\PP^1} = \omega_\Delta(\det \; g) [x, y]_{\PP^1}$. En effet, on a
\begin{align*}
[ g \ast_{-k} x, y]_{\PP^1} &= [ g \ast_{-k} x, g \ast_k (g^{-1} \ast_k y)]_{\PP^1} \\
&= [g \cdot x, g \cdot (g^{-1} \ast_k y)]_{\PP^1} \\
&= \omega_\Delta(\det \; g) [ x, g^{-1} \ast_{-k} y]_{\PP^1}.  \qedhere \\
\end{align*}
\end{proof}

\subsection{Une équation fonctionnelle locale: préliminaires}

Cette section contient les calculs techniques qui sont à la base des démonstrations des résultats principaux de cet article (théorèmes \ref{eqfonct1} et \ref{eqfonct1b}).

\subsubsection{Vecteurs propres de $\psi$} \label{psiinvsec}

Dans toute cette sous-section, on considère $D \in \Phi \Gamma^{\text{ét}}(\Robba)$ de rang $2$, de Rham non triangulin et à poids de Hodge-Tate $0$ et $k \geq 1$ et on note $\Delta = \Nrig(D)$. On aura besoin d'établir le lien entre $\Delta^{\psi = 1}$ et la représentation $\Pi(\Delta)^* \otimes \omega_\Delta$.

\begin{lemme} \label{dualiteinv}
L'image de $(1 - \varphi) \Delta^{\psi = 1}$ par l'involution $w_\Delta$ est égale à $(1 - \omega_\Delta(p)^{-1} \varphi) \Delta^{\psi = \omega_\Delta(p)^{-1}}$ \footnote{Ce dernier module s'identifiant à $(1 - \varphi) \check{\Delta}^{\psi = 1}$ via l'identification entre $\check{\Delta}$ et $\Delta \otimes \omega_\Delta^{-1}$.}.
\end{lemme}

\begin{proof}
Les ingrédients pour démontrer ce lemme se trouvent déjà dans \cite[\S VI]{ColmezPhiGamma}, où on renvoie le lecteur pour les détails techniques des résultats cités. Rappelons d'abord que l'on a (cf. \S \ref{eqdiff}) des inclusions $t^k \Delta \subseteq D \subseteq \Delta$.

Soit $\nabla_k = \prod_{j = 0}^{k - 1} (\nabla - j)$, où $\nabla$ est l'opérateur introduit dans \S \ref{generalites}. Cet opérateur joue un grand rôle dans l'étude des vecteurs localement algébriques de la représentation de $\mathrm{GL}_2(\qp)$ associée à $D$ dans \cite[\S VI]{ColmezPhiGamma} et donc dans la preuve de la compatibilité entre la correspondance de Langlands classique et $p$-adique. On a les propriétés suivantes:
\begin{itemize}
\item $w_D \circ \nabla_k = (-1)^k \nabla_k \circ w_D$ sur $D \boxtimes \zpe$ (\cite[Lem. VI.4.3(ii)]{ColmezPhiGamma}).
\item D'après \cite[Lem. VI.6.14]{ColmezPhiGamma}. l'opérateur $\nabla_k$ induit des isomorphismes $(1 - \varphi) \Delta^{\psi = 1} \xrightarrow{\sim} (1 - \varphi) (t^k \Delta)^{\psi = 1}$ et $(1 - \omega_D(p)^{-1} \varphi) \Delta^{\psi = \omega_D(p)^{-1}} \xrightarrow{\sim} (1 - \varphi) (t^k \Delta)^{\psi = \omega_D(p)^{-1}}$. Notons que les modules $(1 - \varphi) (t^k \Delta)^{\psi = 1}$ et $(1 - \varphi) (t^k \Delta)^{\psi = \omega_D(p)^{-1}}$ sont inclus dans $D^{\psi = 0}$.
\item On a $w_D((1 - \varphi) (t^k \Delta)^{\psi = 1}) = (1 - \varphi) (t^k \Delta)^{\psi = \omega_D(p)^{-1}}$ (\cite[Th. VI.6.8 et Lem. VI.6.14]{ColmezPhiGamma}).
\item Dans les notations de \S \ref{Changement}, on a $D = \Delta_{k, \mathscr{L}}$, où $\mathscr{L} = \mathrm{Fil}^0 \, \DdR(D)$, et on sait que $w_\Delta = \partial^k \circ w_D$, $\omega_D = \omega_\Delta x^k$, et que $\partial^k$ induit (car il est bijectif sur $\Delta$, cf. \cite[\S 3.2]{ColmezPoids}) un isomorphisme
\[ (1 - \omega_D(p)^{-1} \varphi) \Delta^{\psi = \delta_D(p)^{-1}} \xrightarrow{\sim} (1 - \omega_D(p)^{-1} p^k \varphi) \Delta^{\psi = \omega_D(p)^{-1} p^k} = (1 - \omega_\Delta(p)^{-1} \varphi) \Delta^{\psi = \omega_\Delta(p)^{-1}}. \]
\end{itemize}

Si $z \in (1 - \varphi) \Delta^{\psi = 1}$, l'action de $w_\Delta$ sur $z$ est donc donnée par la composée
\begin{align*} 
w_\Delta \colon (1 - \varphi) \Delta^{\psi = 1} &\xrightarrow{(-1)^k \nabla_k} (1 - \varphi) (t^k \Delta)^{\psi = 1} \xrightarrow{w_D} (1 - \omega_D(p)^{-1} \varphi) (t^k \Delta)^{\psi = \omega_D(p)^{-1}} \\
&\xrightarrow{\nabla_k^{-1}} (1 - \omega_D(p)^{-1} \varphi) \Delta^{\psi = \omega_D(p)^{-1}} \xrightarrow{\partial^k} (1 - \omega_\Delta(p)^{-1} \varphi) \Delta^{\psi = \omega_\Delta(p)^{-1}},
\end{align*}
ce qui permet de conclure.
%Soient \[ \mathscr{C}_{\rm e} = \nabla_k (1 - \varphi) \nabla^{\psi = 1}; \;\;\; \mathscr{C}'_{\rm e} = \nabla_k (1 - \omega_D(p)^{-1} \varphi) \Delta^{\psi = \omega_D(p)^{-1}}, \] des sous-modules de $\nabla^{\psi = 1}$. 
\end{proof}

\begin{lemme} \label{phiinv2}
On a $\Pi(D)^{{\matrice p 0 0 1} = p^{-r}} = 0$ pour tout $r > 0$.
\end{lemme}

\begin{proof}
Par dualité, il suffit de montrer que l'adhérence de $({\matrice p 0 0 1} - p^{-r}) (\Pi(D)^* \otimes \omega_D)$ dans $\Pi(D)^* \otimes \omega_D$ est égal au module $\Pi(D)^* \otimes \omega_D$ tout entier. Il suffit donc de montrer que $\Pi(D)^* \otimes \omega_D$ contient un sous-espace dense sur lequel ${\matrice p 0 0 1} - p^{-r}$ agit de manière surjective.

Considérons (cf. \cite[\S V.1.2]{ColmezPhiGamma}) le module \footnote{Soit $\widetilde{\mathbf{E}}_\qp^+$ le complété de la clôture radicielle de $\F_p((T))$, $\widetilde{\mathbf{A}}_\qp^+ = W(\widetilde{\mathbf{E}}_\qp^+)$ son anneau des vecteurs de Witt et posons $\widetilde{\mathscr{O}}^+_\E = \mathscr{O}_L \cdot \widetilde{\mathbf{A}}^+_\qp$. L'anneau $\widetilde{\mathbf{E}}_\qp^+$ est muni naturellement d'actions de $\varphi$ et $\Gamma$ commutant entre elles (et coïncidant avec celles sur $\F_p((T))$ inclus dans le corps résiduel $k_L((T))$ de $\mathscr{O}_\E$), qui s'étendent de manière unique à $\widetilde{\mathbf{A}}_\qp^+$ et par $\mathscr{O}_L$-linéarité à $\widetilde{\mathscr{O}}^+_\E$, l'action de $\varphi$ devenant bijective, et $\mathscr{O}_\E$ s'identifie naturellement au sous-anneau de $\widetilde{\mathscr{O}}^+_\E$ engendré topologiquement par $[1 + T] - 1$ (que l'on identifie à $T \in \mathscr{O}_\E$) et son inverse. De la même manière que l'on obtient $\Robba$ à partir de $\mathscr{O}^+_\E = \mathscr{O}_L[[T]]$ ($\Robba = \cup_{b \geq 1} \cap_{a \geq b} \E^{[r_a, r_b]}$, où $\E^{[r_a, r_b]} = \mathscr{O}_\E^+[\frac{T^{n_a}}{p}, \frac{p}{T^{n_b}}]^\wedge[1/p]$, où $n_a = r_a^{-1}, n_b = r_b^{-1}$ et la complétion est prise par rapport à la topologie $p$-adique, cf. \S \ref{anneaux} ou bien \cite[\S I.1.2]{ColmezPhiGamma}), l'on définit l'anneau $\widetilde{\Robba}$ à partir de $\widetilde{\mathscr{O}}^+_\E$. L'intérêt de ces objets réside dans le fait qu'à partir d'eux l'on peut construire des objets où $\varphi$ dévient inversible. cf. \cite[\S V.1.1]{ColmezPhiGamma} pour plus de détails et des exemples.} \footnote{Un ensemble $X$ dans un espace vectoriel topologique sur $L$ est borné si pour tout voisinage $U$ de zero il existe $\alpha \in L$ tel que $X \subseteq \alpha U$. En particulier, si une suite $(x_n)_{n \in \N}$ est bornée, alors $p^n x_n \to 0$ quand $n \to +\infty$.} \[ \widetilde{D}^+ := \{ x \in \widetilde{\Robba} \otimes D : (\varphi^n(x))_n \text{ est une suite bornée } \} \subseteq \Pi(D)^* \otimes \omega_D. \] Alors $\widetilde{D}^+$ est dense (cf. \cite[note (65)]{ColmezPhiGamma}, ceci suit essentiellement de l'irréductibilité de $D$ et de \cite[Lem. IV.2.2]{ColmezMirabolique}), et ${\matrice p 0 0 1} - p^{-r}$, qui agit comme $\varphi - p^{-r}$, y est surjectif, car il admet un inverse défini par $\sum_{i \geq 0} p^{ir} \varphi^i$ (qui converge sur $\widetilde{D}^+$ car le terme général de la série tend vers $0$, les $\varphi^i(x)$ étant bornés). Ceci permet de conclure.
\end{proof}

\begin{lemme} \label{phiinv}
On a $\Pi(\Delta)^{{\matrice p 0 0 1} = 1} = 0$.
\end{lemme}

\begin{proof}
Soit $\mathscr{L} = \mathrm{Fil}^0 \, \DdR(D)$ de sorte que $D = \Delta_{k, \mathscr{L}}$. L'opérateur $\partial^{-k - 1}$ induit une bijection entre $\Pi(\Delta)^{{\matrice p 0 0 1} = 1}$ et $\Pi(\Delta)^{{\matrice p 0 0 1} = p^{-k - 1}}$ et il suffit donc de montrer que ce dernier module est nul. D'après \S \ref{Changement} (voir l'avant dernier point dans la liste de propriétés), on a une suite exacte
\[ 0 \to \mathscr{L} \otimes (\mathrm{LL}_p(\Delta) \otimes \mathrm{Sym}^{k-1}) \to \Pi(\Delta, k) \to \Pi(D) \to 0. \] Comme $ \Pi(\Delta, k)$ s'identifie (\cite[Rem. 3.18]{ColmezPoids}) à $\Pi(\Delta)$ muni de l'action définie par ${\matrice a b c d} \cdot_k v := (\partial c + a)^k ({\matrice a b c d} \cdot v)$, on a une identification $\Pi(\Delta)^{{\matrice p 0 0 1} = p^{-k - 1}} = \Pi(\Delta, k)^{{\matrice p 0 0 1} = p^{-1}}$. Il suffit donc de montrer que $W_\mathscr{L}^{{\matrice p 0 0 1} = p^{-1}} = \Pi(D)^{{\matrice p 0 0 1} = p^{-1}} = 0$, où l'on a noté $W_\mathscr{L} = \mathscr{L} \otimes (\mathrm{LL}_p(\Delta) \otimes \mathrm{Sym}^{k-1})$.

Le modèle de Kirillov induit un isomorphisme entre $W_\mathscr{L}$ et $\mathrm{LA}_{\rm c}(\qpe, L^{-k}_\infty \otimes \mathscr{L}_i)$, qui n'a évidement pas de vecteurs satisfaisant ${\matrice p 0 0 1} \cdot v = p v$ (si $\phi \in \mathrm{LA}_{\rm c}(\qpe, L^{-k}_\infty \otimes \mathscr{L}_i)$ est telle que ${\matrice p 0 0 1} \cdot \phi = p \phi$, alors $\phi(x) = p^n ({\matrice p 0 0 1}^n \phi)(x) = p^{2n} \phi(p^n x) = 0$ si $n \gg 0$ car $\phi$ est à support compact). On conclut en utilisant le lemme \ref{phiinv2}.
\end{proof}

\begin{cor} \label{invphicor}
On a $(\Pi(\Delta)^* \otimes \omega_\Delta)^{{\matrice p 0 0 1} = 1} = (\Delta \boxtimes \PP^1)^{{\matrice p 0 0 1} = 1}$.
\end{cor}

\begin{proof}
Il suffit de prendre les invariants par ${\matrice p 0 0 1}$ de la suite exacte \[ 0 \to \Pi(\Delta)^* \otimes \omega_\Delta \to \Delta \boxtimes \PP^1 \to \Pi(\Delta) \to 0 \] et utiliser le lemme \ref{phiinv} ci-dessus.
\end{proof}

Rappelons que l'action de $\mathrm{GL}_2(\qp)$ sur $\Delta \boxtimes \PP^1$ est déterminée par le \textit{squelette d'action} (\cite[II.1.2]{ColmezPhiGamma}). En particulier, si $z = (z_1, z_2) \in \Delta \boxtimes \PP^1$, on a
\[ {\matrice p 0 0 1} \cdot z = (\varphi(z_1) + \omega_\Delta(p) w_\Delta( \mathrm{Res}_\zpe (\psi(z_2)), \omega_\Delta(p) \psi(z_2)). \]
On en déduit que $z$ est invariant par ${\matrice p 0 0 1}$ si et seulement si $z_2 \in \Delta^{\psi = \omega_\Delta(p)^{-1}}$ et $(1 - \varphi) z_1 = w_\Delta ( (1 - \omega_\Delta(p)^{-1} \varphi) z_2)$.

\begin{lemme} \label{psidist}
L'application $\mathrm{Res}_{\zp}$ induit un isomorphisme $$ (\Pi(\Delta)^* \otimes \omega_\Delta)^{{\matrice p 0 0 1} = 1} \xrightarrow{\sim} \Delta^{\psi = 1}. $$
\end{lemme}

\begin{proof}
Comme l'on a déjà mentionné, l'injectivité suit de \cite[Prop. 2.20]{ColmezPoids}. Si $z \in \Delta^{\psi = 1}$ alors, par le lemme \ref{dualiteinv} ci-dessus, $w_\Delta((1 - \varphi) z) \in (1 - \omega_\Delta(p)^{-1} \varphi) \Delta^{\psi = \omega_\Delta(p)^{-1}}$ et il existe donc $z' \in \Delta^{\psi = \omega_\Delta(p)^{-1}}$ tel que $w_\Delta ((1 - \varphi) z) = (1 - \omega_\Delta(p)^{-1} \varphi) z'$. Alors $\tilde{z} = (z, z') \in (\Delta \boxtimes \PP^1)^{{\matrice p 0 0 1} = 1} = (\Pi(\Delta)^* \otimes \omega_\Delta)^{{\matrice p 0 0 1} = 1}$ (corollaire \ref{invphicor}) et satisfait $\mathrm{Res}_\zp(\tilde{z}) = z$, ce qui permet de conclure.

%Par injectivité de $(1 - \varphi) \colon \Delta^{\psi = 1} \to \Delta^{\psi = 0}$, il suffit de montrer que l'application $\mathrm{Res}_\zpe$ induit une surjection de $(\Pi(\Delta)^* \otimes \omega_\Delta)^{{\matrice p 0 0 1} = 1}$ sur $(1 - \varphi) \Delta^{\psi = 1}$.
%
%Soit $z \in \Delta^{\psi = 1}$, alors $\nabla_k (1 - \varphi) z \in (1 - \varphi) (t^k \Delta)^{\psi = 1} \subseteq (1 - \varphi) D^{\psi = 1}$ et, pour le résultat dans les cas étale, il existe $\tilde{z}' \in (\Pi(D)^* \otimes \omega_D)^{{\matrice p 0 0 1} = 1} \subseteq (\Pi(\Delta)^* \otimes \omega_\Delta)^{{\matrice p 0 0 1} = 1}$ tel que $\mathrm{Res}_\zpe (\tilde{z}') = \nabla_k (1 - \varphi) z$. Si $\tilde{z} = \nabla_k^{-1} \tilde{z}'$, alors $\mathrm{Res}_\zpe (\tilde{z}) = (1 - \varphi) z$, ce qui permet de conclure.
\end{proof}

\begin{remarque} 
Soit $\alpha \in L^\times$ tel que $v_p(\alpha) \in \Z$. En appliquant $\partial^{-r}$, $r = v_p(\alpha)$, et en tordant par le caractère non-ramifié $\delta \colon \qpe \to \mathscr{O}_L^\times$ tel que $\delta(p) = \alpha^{-1} p^r$ à l'égalité du lemme \ref{psidist} (ou bien en répétant la même preuve), on montre que $\mathrm{Res}_{\zp}$ induit un isomorphisme
\[ (\Pi(\Delta)^* \otimes \omega_\Delta)^{{\matrice p 0 0 1} = \alpha} \xrightarrow{\sim} \Delta^{\psi = \alpha^{-1}}. \] Ceci s'applique en particulier à $\alpha = \omega_\Delta(p)^{-1} = \omega_D(p)^{-1} p^k$ car $\omega_D(p) \in \mathscr{O}_L^\times$ (car $D$ est étale et donc son caractère centrale est unitaire). Pour un tel $z \in \Delta^{\psi = \alpha^{-1}}$, on notera $\tilde{z} \in (\Pi(\Delta)^* \otimes \omega_\Delta)^{{\matrice p 0 0 1} = \alpha}$ l'élément correspondant.
\end{remarque}

%\begin{remarque}
%La condition du lemme est vrai dès que l'on sait qu'il existe $D \in \Phi\Gamma^{\text{ét}}(\Robba)$ tel que $\Delta = \Nrig(D)$, ce qui est toujours le cas que l'on considère dans ce texte.
%\end{remarque}

%Si $D \in \Phi\Gamma^{\text{ét}}(\E)$ et si $z \in D^{\psi = 1}$, l'élément $(z)_{n \geq 0}$ appartient à $(D^\natural \boxtimes \qp)^{{\matrice p 0 0 1} = 1}$ et l'application $x \mapsto \mathrm{Res}_\zp(x)$ définit un isomorphisme (cf. \cite{ColmezDosp}, prop. III.23 ainsi que ainsi que la remarque III.8) $$ (D^\natural \boxtimes \qp)^{{\matrice p 0 0 1} = 1} \to D^{\psi = 1}. $$ Si $D$ est absolument irréductible, alors $D^\natural = D^\sharp$ et l'application $x \mapsto (\mathrm{Res}_\zp({\matrice {p^n} 0 0 1} x))_{n \geq 0}$ induit un isomorphisme $$ D^\natural \boxtimes \PP^1 \to D^\natural \boxtimes \qp. $$ On note $\tilde{z} \in D^\natural \boxtimes_{\omega_D} \PP^1$ la préimage de $z$ par la composition de ces isomorphismes.
%
%Si $D_{\rm rig} \in \Phi\Gamma^{\text{ét}}(\Robba)$ correspond à $D$, on utilise l'identité $D_{\rm rig}^{\psi = 1} = D^{\psi = 1} \otimes_{\zp[[\Gamma]]} \mathscr{D}(\Gamma)$ et l'action de l'algèbre de distributions sur $\Pi(D_{\rm rig})^*$ pour relever un élément $z \in D_{\rm rig}^{\psi = 1}$ en un élément $\tilde{z} \in (\Pi(D_{\rm rig})^*)^{{\matrice p 0 0 1} = 1}.$ Plus précisément, si $z = \sum z_i \otimes \lambda_i$, $z_i \in D^{\psi = 1}, \lambda_i \in \mathscr{D}(\Gamma)$, on a $\tilde{z} = \sum \tilde{z}_i \otimes \lambda_i$.

\subsubsection{Lois de réciprocité (encore)}

Dans cette section, on paraphrase les lois des réciprocité de \S \ref{loisrec} sous une forme appropriée pour les calculs futurs. Plus précisément, on tient compte de toutes les identifications faites entre les modules de de Rham d'un $(\varphi, \Gamma)$, de son dual de Tate et de ses différents tordus par des caractères localement algébriques. Ceci nous permettra de comparer tous ces éléments sans peine dans un même module de de Rham.

Soit $\Delta \in \Phi \Gamma(\Robba)$ de rang $2$, de Rham non triangulin à poids de Hodge-Tate tous nuls comme dans \S \ref{phiGammadR}.

\begin{lemme} \label{expd1}
Soient $z \in \Delta^{\psi = \omega_\Delta(p)^{-1}}$, $\eta \colon \zpe \to L^\times$ un caractère d'ordre fini, $k \geq 1$ un entier, $j \in \Z$ tel que $j \geq -k$ et $m \in \Z$. Notons $\check{z} = z \otimes e_{\omega_\Delta}^\vee \in \check{\Delta}^{\psi = 1}$. Alors
\[ \exp^*(\int_\Gamma \eta \chi^{-j-k} \cdot \mu_{\check{z}}) \otimes \mathbf{e}^{\rm dR, \vee}_{\eta, -j-k, \omega_\Delta^\vee} = \frac{\omega_\Delta(p)^n}{(j + k -1)!} p^{-n({\j} + k)} \, \mathrm{Tr}_{L_n / L} ( G(\eta)^{-1} \Omega^{-1} [ \varphi^{-n} \partial^{{\j} + k - 1} z ]_0). \]
\end{lemme}

\begin{proof}
Par la proposition \ref{loirecexpd} on a, pour $n \gg 0$,
\begin{eqnarray*} \exp^*(\int_\Gamma \eta \chi^{-j-k} \cdot \mu_{\check{z}}) &=& \exp^*(\int_\Gamma 1 \cdot (\mu_z \otimes e_\eta \otimes e_{-j-k} \otimes e_{\omega_\Delta}^\vee )) \\
&=& p^{-n} \, \mathrm{Tr}_{L_n / L}( [\varphi^{-n} (z \otimes \mathbf{e}_{\eta, -j-k, \omega_\Delta^\vee} )]_0), \\
\end{eqnarray*}
Comme $ \varphi^{-n}(\mathbf{e}_{\eta, -j-k, \omega_\Delta^\vee}) = \omega_\Delta(p)^n \mathbf{e}_{\eta, -j-k, \omega_\Delta^\vee}$, on a $$ [ \varphi^{-n} (z \otimes \mathbf{e}_{\eta, -j-k, \omega_\Delta^\vee}) ]_0 = \omega_\Delta(p)^n [ \varphi^{-n} z \otimes \mathbf{e}_{\eta, -j-k, \omega_\Delta^\vee}]_0. $$ Clarifions la notation utilisée. L'élément $\varphi^{-n} z \otimes \mathbf{e}_{\eta, -j-k, \omega_\Delta^\vee}$ appartient à $$ L_n[[t]] \otimes \DdR(\Delta(\eta \chi^{-j-k} \omega_\Delta^{-1})) = L_n[[t]] \otimes M_{\rm dR} \otimes L \cdot \mathbf{e}^{\rm dR}_{\eta, -j-k, \omega_\Delta^\vee}$$ et peut être exprimé sous la forme $ x \otimes \mathbf{e}^{\rm dR}_{\eta, -j-k, \omega_\Delta^\vee}$, $x = G(\eta)^{-1} t^{1-k-j} \Omega^{-1} \varphi^{-n} z \in L_n((t)) \otimes M_{\rm dR}$. Si on écrit $\varphi^{-n} z$ sous la forme $\sum_{l \geq 0} a_l t^l$, $a_l \in L_n \otimes M_{\rm dR}$, on a alors $$[\varphi^{-n} z \otimes \mathbf{e}_{\eta, -j-k, \omega_\Delta^\vee}]_0 = G(\eta)^{-1} \Omega^{-1} a_{j+k-1} \otimes \mathbf{e}^{\rm dR}_{\eta, -j-k, \omega_\Delta^\vee}. $$ Notons que l'élément $\mathbf{e}^{\rm dR}_{\eta, -j-k, \omega_\Delta^\vee}$ est invariant par $\Gamma$ et commute donc à $\mathrm{Tr}_{L_n / L}$. On en déduit
$$ \mathrm{Tr}_{L_n / L} ( [\varphi^{-n} (z \otimes \mathbf{e}_{\eta, -j-k, \omega_\Delta^\vee})]_0) = \omega_\Delta(p)^n \, \mathrm{Tr}_{L_n / L} ( G(\eta)^{-1} \Omega^{-1} [ t^{-j -k + 1} \varphi^{-n} z ]_0) \otimes \mathbf{e}^{\rm dR}_{\eta, -j-k, \omega_\Delta^\vee} $$
En utilisant la formule $[ t^{-{\j} -k + 1} \varphi^{-n} z ]_0 = [ \varphi^{-n} z ]_{{\j} + k - 1} = \frac{p^{-n({\j} + k - 1)}}{({\j} + k -1)!} [\varphi^{-n} \partial^{{\j} + k - 1} z]_0$ on en déduit 
$$ \exp^*(\int_\Gamma \eta \chi^{-j-k} \cdot \mu_{\check{z}}) \otimes \mathbf{e}^{\rm dR, \vee}_{\eta, -j-k, \omega_\Delta^\vee} = \frac{\omega_\Delta(p)^n}{(j + k -1)!} p^{-n({\j} + k)} \, \mathrm{Tr}_{L_n / L} ( G(\eta)^{-1} \Omega^{-1} [ \varphi^{-n} \partial^{{\j} + k - 1} z ]_0), $$
ce qui permet de conclure.
\end{proof}

\begin{lemme} \label{exp1}
Soient $z \in \Delta^{\psi = 1}$, $\eta \colon \zpe \to L^\times$ un caractère d'ordre fini et ${\j} \geq 1$. Alors
\[ \exp^{-1}(\int_\Gamma \eta \chi^{\j} \cdot \mu_z) \otimes \mathbf{e}^{\rm dR, \vee}_{\eta, j} = (-1)^{\j} ({\j}-1)! p^{n(j-1)}  \mathrm{Tr}_{L_n / L} ( G(\eta)^{-1} [\varphi^{-n} \partial^{-{\j}} z]_0). \]
\end{lemme}

\begin{proof}
Par la proposition \ref{loirecexp} on a, pour $n \gg 0$,
\begin{eqnarray*}
\exp^{-1}(\int_\Gamma \eta \chi^{\j} \cdot \mu_z) &=& \exp^{-1}(\int_\Gamma \chi^{\j} \cdot (\mu_z \otimes e_\eta)) \\
&=& (-1)^{\j} ({\j}-1)! p^{-n} \, \mathrm{Tr}_{L_n / L} ( [\varphi^{-n} (\partial^{-j} z \otimes e_\eta \otimes t^{-j} e_j)]_0) \\
\end{eqnarray*}
et, comme $\varphi^{-n}(e_\eta \otimes t^{-j} e_j) = p^{nj} \, e_\eta \otimes t^{-j} e_j$, on a $$ \mathrm{Tr}_{L_n / L} ( [\varphi^{-n} (\partial^{-{\j}} z \otimes e_\eta \otimes t^{-j} e_j)]_0) = p^{nj} \, \mathrm{Tr}_{L_n / L} ( [\varphi^{-n} \partial^{-{\j}} z \otimes e_\eta \otimes t^{-j} e_j]_0). $$ 
Comme précédemment, l'élément $\varphi^{-n} \partial^{-{\j}} z \otimes e_\eta \otimes t^{-j} e_j$ appartient à $L_n[[t]] \otimes \DdR(\Delta(\eta\chi^{\j})) = L_n[[t]] \otimes M_{\rm dR} \otimes L \cdot \mathbf{e}^{\rm dR}_{\eta, j} $ et peux donc être exprimé sous la forme $ x \otimes \mathbf{e}^{\rm dR}_{\eta, j}$, avec $x = G(\eta)^{-1} \varphi^{-n} \partial^{-j} z \in L_n[[t]] \otimes M_{\rm dR}$. On a alors $$ [\varphi^{-n} \partial^{-{\j}} z \otimes e_\eta \otimes t^{-{\j}} e_{\j}]_0 = G(\eta)^{-1} [\varphi^{-n} \partial^{-j} z]_0 \otimes \mathbf{e}^{\rm dR}_{\eta, j}. $$ Notons finalement que l'élément $\mathbf{e}^{\rm dR}_{\eta, j}$ est invariant par $\Gamma$ et commute donc à la trace, ce qui donne
$$ \mathrm{Tr}_{L_n / L} ( [\varphi^{-n} (\partial^{-{\j}} z \otimes e_\eta \otimes t^{-{\j}} e_{\j})]_0) = p^{n{\j}} \, \mathrm{Tr}_{L_n / L}( G(\eta)^{-1} [\varphi^{-n} \partial^{-{\j}} z]_0) \otimes \mathbf{e}^{\rm dR}_{\eta, j}, $$ ce qui donne
\[ \exp^{-1}(\int_\Gamma \eta \chi^{\j} \cdot \mu_z) \otimes \mathbf{e}^{\rm dR, \vee}_{\eta, j} = (-1)^{\j} ({\j}-1)! p^{n(j-1)}  \mathrm{Tr}_{L_n / L} ( G(\eta)^{-1} [\varphi^{-n} \partial^{-{\j}} z]_0), \]
et permet de conclure.
\end{proof}

\subsubsection{Dualité et modèle de Kirillov (encore)}

On revient à la situation de \S \ref{psiinvsec} et on considère $D \in \Phi \Gamma^{\text{ét}}(\Robba)$ de rang $2$, de Rham non triangulin à poids de Hodge-Tate distincts et on note $\Delta = \Nrig(D)$. Considérons dans la suite le modèle de Kirillov de $\Pi(\Delta)^{u^+-\text{fini}}$ à valeurs dans $\mathrm{LA}_{\rm c}(\qpe, \Delta_{\rm dif}^+)^\Gamma$. Soit $\alpha \in L^\times$ et soit $z \in \Delta^{\psi = \alpha}$ obtenu par restriction à $\zp$ d'un élément $ \tilde{z} \in (\Pi(\Delta)^* \otimes \omega_\Delta)^{{\matrice p 0 0 1} = \alpha^{-1}}$, i.e. $z = \mathrm{Res}_\zp (\tilde{z})$. On a donc $\mathrm{Res}_\zp({\matrice {p^{i + n}} 0 0 1} \tilde{z}) = \mathrm{Res}_\zp(\alpha^{-(i+n)} \tilde{z}) = \alpha^{-(i+n)} z$ et, si $\phi \in \mathrm{LA}_{\rm c}(\qpe, \Delta_{\rm dif}^-)^\Gamma$, la formule \eqref{formulemagique} pour l'accouplement prend la forme très simple suivante:
\[ [\tilde{z}, \phi] = \sum_{i \in \Z} \omega_\Delta(p^{-i}) \alpha^{-(i+n)} [\varphi^{-n}(z), \phi(p^i) )]_{\rm dif} \;\;\; (n \gg 0). \]
Plus généralement, si $j \in \Z$, on peut appliquer $\partial^{\j}$ à $\tilde{z}$ pour obtenir un élément dans $(\Pi(\Delta)^* \otimes \omega_\Delta)^{{\matrice p 0 0 1} = \alpha^{-1}p^{-j}}$ et on a
\begin{equation} \label{formulemagique2}
[\partial^{\j} \tilde{z}, \phi] = \sum_{i \in \Z} \omega_\Delta(p^{-i}) \alpha^{-(i + n)} p^{-j (i + n)} [\varphi^{-n} \partial^{\j} z, \phi(p^i) )]_{\rm dif}.
\end{equation}

Rappelons que l'on a fixé une base $f_1, f_2$ de $M_{\rm dR}$. Les éléments $1 \otimes f_1, 1 \otimes f_2$ forment une base du $L_\infty((t))$-espace vectoriel $\Delta_{\rm dif} = L_\infty((t)) \otimes M_{\rm dR}$ et l'on identifie $(1 \otimes f_1) \wedge (1 \otimes f_2) = \Omega^{-1} dt \otimes e_{\omega_\Delta}$ sous l'isomorphisme $\wedge^2 \Delta_{\rm dif} = (L_\infty((t)) dt) \otimes \omega_\Delta$. Nous étendons l'accouplement $\langle \;,\; \rangle_{\rm dR}$ par $L_\infty((t))$-linéarité en un accouplement $\langle \;,\; \rangle_{\rm dR} \colon \Delta_{\rm dif} \times \Delta_{\rm dif} \to L_\infty((t))$ satisfaisant
\begin{equation} \label{accdR} x \wedge (h \otimes f_i) = h \langle x, 1 \otimes f_{3-i} \rangle_{\rm dR} ((1 \otimes f_{3-i}) \wedge (1 \otimes f_i)) = h \langle x, 1 \otimes f_{3-i} \rangle_{\rm dR} (-1)^{i} \Omega^{-1} dt \otimes e_{\omega_\Delta},
\end{equation}
pour $x \in \Delta_{\rm dif}$, $h \in L_\infty((t))$ et $i = 1, 2$.

\subsubsection{Exponentielle duale et accouplement $[\;,\;]_{\PP^1}$} \label{accexpd}

En s'inspirant de Nakamura (\cite[\S 3.4.5]{Nakamura2}) et de la formule de la proposition \ref{BHeqfonct} permettant de retrouver les facteurs epsilon d'une représentation lisse à partir de son modèle de Kirillov, on définit, pour $\eta \colon \zpe \to L^\times$ un caractère d'ordre fini, $k \geq 1$, $i \in \{1, 2\}$ et $m \in \Z$, la fonction $f^i_{\eta, k, m} \in \mathrm{LA}_{\mathrm{c}}(\qpe, \Delta_{\rm dif}^-)^\Gamma$ par la formule $$ f^i_{\eta, k, m}(p^n a) = \left\{
  \begin{array}{c c}
   (-1)^{i + 1} (k-1)! \sigma_a (G(\eta)^{-1} t^{-k}) \otimes f_i & \quad \text{ si $n = m$} \\
    0 & \quad \text{si $n \neq m$,}  \\
  \end{array}
\right. $$ où $a \in \zpe$, $n \in \Z$. Notons que, pour $a \in \zpe$, on a tout simplement \[ f^i_{\eta, k, m}(p^m a) = (-1)^{i+1} (k-1)! \eta(a) a^{-k} G(\eta)^{-1} t^{-k} \otimes f_i. \]

\begin{lemme} \label{expd2}
Soient $z \in \Delta^{\psi = \omega_\Delta(p)^{-1}}$, $\eta \colon \zpe \to L^\times$ un caractère d'ordre fini, $k \geq 1$ un entier, $j \in \Z$ tel que $j > -k$, $i \in \{ 1, 2 \}$ et $m \in \Z$. Notons $\check{z} = z \otimes e_{\omega_\Delta}^\vee \in \check{\Delta}^{\psi = 1}$. Alors \footnote{Rappelons que $\tilde{z} \in (\Pi(\Delta)^* \otimes \omega_\Delta)^{{\matrice p 0 0 1} = \omega_\Delta(p)}$ l'élément correspondant par le lemme \ref{psidist}}
%$$ \langle \exp^*(\int_\Gamma \eta \chi^{-j-k} \cdot \mu_{\check{z}}) \otimes \mathbf{e}^{\rm dR, \vee}_{\eta, -j-k, \omega_\Delta^\vee}, f_i \rangle_{\rm dR} = \frac{p-1}{p} \alpha^m p^{mj} \eta(-1) \frac{\omega_\Delta(-1) (-1)^k}{(j+k-1)!} [ \partial^{\j} \tilde{z}, f^{3-i}_{\eta, k, m}]_{\PP^1}. $$
\small{
\begin{equation*} 
[ \partial^{\j} \tilde{z}, f^{3 - i}_{\eta, k, m}]_{\PP^1} = (-1)^{i + 1} \frac{p}{p-1} \omega_\Delta(p)^n p^{-j(m+n)} \omega_\Delta(-1) \eta(-1)(-1)^k p^{-n k} \mathrm{Tr}_{L_n / L} \big(  G(\eta)^{-1} \Omega^{-1} \langle [  \varphi^{-n} \partial^{{\j}+k-1} z]_0, f_i \rangle_{\rm dR} \big).
\end{equation*}
} \normalsize
\end{lemme}

\begin{proof}
Traitons le cas $i = 1$, l'autre étant évidement analogue. Rappelons que $\partial^{\j} \tilde{z} \in (\Pi(\Delta)^* \otimes \omega_\Delta)^{{\matrice p 0 0 1} = \omega_\Delta(p) p^{-{\j}}}$ et que, si $n \gg 0$, on a
\small{\begin{eqnarray*}
[ \partial^{\j} \tilde{z}, f^2_{\eta, k, m}]_{\PP^1} &=& \sum_{i \in \Z} \omega_\Delta(p)^n p^{-j(i+n)} [\varphi^{-n} \partial^{\j} z,  f^2_{\eta, k, m}(p^i) )]_{\rm dif} \\
&=& - \omega_\Delta(p)^n p^{-j(m+n)} (k-1)! [ \varphi^{-n} \partial^{\j} z, G(\eta)^{-1} t^{-k} \otimes f_2 ]_{\rm dif},
%&=& \omega_\Delta(p)^n \alpha^{-m-n} p^{-j(m+n)} \omega_\Delta(-1) \eta(-1) (-1)^k (k-1)! \text{rés}_L( (\varphi^{-n} \partial^{\j} z \wedge (G(\eta)^{-1} t^{-k} \otimes f_2)) \otimes e_{\omega_\Delta}^\vee), \\
\end{eqnarray*}}\normalsize
où la première égalité suit de la formule \eqref{formulemagique2} et la deuxième égalité du fait que $f_{\eta,k, m}^2(p^i) = 0$ dès que $i \neq m$. Par définition de l'accouplement $[ \;,\;]_{\rm dif} $ (cf. formule \eqref{accdif}), et en notant que $\sigma_{-1}(G(\eta)^{-1}t^{-k} \otimes f_2) = \eta(-1)(-1)^k (G(\eta)^{-1} t^{-k} \otimes f_2$, on a aussi
\[ [ \varphi^{-n} \partial^{\j} z, G(\eta)^{-1} t^{-k} \otimes f_2 ]_{\rm dif} = \omega_\Delta(-1) \eta(-1) (-1)^k \text{rés}_L( (\varphi^{-n} \partial^{\j} z \wedge (G(\eta)^{-1} t^{-k} \otimes f_2)) \otimes e_{\omega_\Delta}^\vee). \] On en déduit
\small \begin{equation} \label{eaea11}
[ \partial^{\j} \tilde{z}, f^2_{\eta, k, m}]_{\PP^1} = - \omega_\Delta(p)^n p^{-j(m+n)} \omega_\Delta(-1) \eta(-1) (-1)^k (k-1)! \text{rés}_L( (\varphi^{-n} \partial^{\j} z \wedge (G(\eta)^{-1} t^{-k} \otimes f_2)) \otimes e_{\omega_\Delta}^\vee).
\end{equation} \normalsize
En utilisant la formule de l'équation \eqref{accdR}, on peut écrire
%Le choix de la base $f_1, f_2$ de $\DdR(D)$ nous permet de définir un accouplement $\langle \;,\; \rangle_{\rm dR}: \Delta_{\rm dif} \times M_{\rm dR} \to \Delta_{\rm dif}$ par la formule $x \wedge f_i = \langle x, f_{3-i} \rangle_{\rm dR} (f_1 \wedge f_2) = \langle x, f_{3-i} \rangle_{\rm dR} \Omega^{-1} e_\Delta$, $i = 1, 2$. Cela nous donne
\begin{eqnarray*}
\text{rés}_L( (\varphi^{-n} \partial^{\j} z \wedge (G(\eta)^{-1} t^{-k} \otimes f_2) \otimes e_{\omega_\Delta}^\vee) &=& - (-1)^{i + 1} \text{rés}_L( G(\eta)^{-1} \Omega^{-1} t^{-k+1} \langle \varphi^{-n} \partial^{\j} z, 1 \otimes f_1 \rangle_{\rm dR} t^{-1} dt) \\
&=& - (-1)^{i + 1} T_L( G(\eta)^{-1} \Omega^{-1} \langle [ \varphi^{-n} \partial^{\j} z]_{k - 1}, f_1 \rangle_{\rm dR}).
\end{eqnarray*}
Enfin, la formule \eqref{eaea11}, l'identité $[\varphi^{-n} z]_{k - 1} = \frac{p^{-n(k - 1)}}{(k - 1)!} [\varphi^{-n} \partial^{k - 1} z]_0$ et l'égalité $T_L = \frac{p}{p-1} p^{-n} \mathrm{Tr}_{L_n / L}$ en donnent
\small{ \begin{equation*} 
[ \partial^{\j} \tilde{z}, f^2_{\eta, k, m}]_{\PP^1} = (-1)^{i + 1} \frac{p}{p-1} \omega_\Delta(p)^n p^{-j(m+n)} \omega_\Delta(-1) \eta(-1)(-1)^k p^{-n k} \mathrm{Tr}_{L_n / L} \big(  G(\eta)^{-1} \Omega^{-1} \langle [  \varphi^{-n} \partial^{{\j}+k-1} z]_0, f_1 \rangle_{\rm dR} \big),
\end{equation*}
} \normalsize ce qui permet de conclure.
\end{proof}

\begin{prop} \label{formexpd} Soient $z \in \Delta^{\psi = \omega_\Delta^{-1}(p)}$, $\eta \colon \zpe \to L^\times$ un caractère d'ordre fini, $k \geq 1$ un entier, $j \in \Z$ tel que $j \geq -k$, $i \in \{ 1, 2 \}$ et $m \in \Z$. Notons $\check{z} = z \otimes e_{\omega_\Delta}^\vee \in \check{\Delta}^{\psi = 1}$. Alors $$ \langle \exp^*(\int_\Gamma \eta \chi^{-j-k} \cdot \mu_{\check{z}}) \otimes \mathbf{e}^{\rm dR, \vee}_{\eta, -j-k, \omega_\Delta^\vee}, f_i \rangle_{\rm dR} = (-1)^{i + 1} \frac{p-1}{p} p^{mj} \eta(-1) \frac{\omega_\Delta(-1) (-1)^k}{(j+k-1)!} [ \partial^{\j} \tilde{z}, f^{3-i}_{\eta, k, m}]_{\PP^1}. $$
\end{prop}

\begin{proof}
En projetant l'égalité du lemme \ref{expd1} sur la droite engendrée par $f_i$ on obtient la formule
\small{\begin{equation} \label{eaea12}
\langle \exp^*(\int_\Gamma \eta \chi^{-{\j}} \cdot \mu_{\check{z}}) \otimes \mathbf{e}^{\rm dR, \vee}_{\eta, -j, \omega_\Delta^\vee} , f_i \rangle_{\rm dR} = \frac{\omega_\Delta(p)^n}{({\j} + k -1)!} p^{-n({\j} + k)} \mathrm{Tr}_{L_n / L} ( G(\eta)^{-1} \Omega^{-1} \, \langle [ \varphi^{-n} \partial^{{\j} + k - 1} z ]_0 , f_i \rangle_{\rm dR} ).
\end{equation}} \normalsize

L'égalité du lemme \ref{expd2} et la formule \eqref{eaea12} ci-dessus donnent
\small{ \begin{eqnarray*} 
[ \partial^{\j} \tilde{z}, f^{3 - i}_{\eta, k, m}]_{\PP^1} &=& (-1)^{i + 1} \frac{p}{p-1} \omega_\Delta(p)^n p^{-j(m+n)} \omega_\Delta(-1) \eta(-1)(-1)^k p^{-n k} \mathrm{Tr}_{L_n / L} \big(  G(\eta)^{-1} \Omega^{-1} \langle [  \varphi^{-n} \partial^{{\j}+k-1} z]_0, f_i \rangle_{\rm dR} \big) \\
&=& (-1)^{i + 1} \frac{p}{p-1} p^{-m{\j}} \omega_\Delta(-1) \eta(-1) (-1)^k (j + k - 1)! \langle \exp^*(\int_\Gamma \eta \chi^{-{\j}} \cdot \mu_{\check{z}}) \otimes \mathbf{e}^{\rm dR, \vee}_{\eta, -j, \omega_\Delta^\vee} , f_i \rangle_{\rm dR},
\end{eqnarray*}} \normalsize ce qui permet de conclure.
\end{proof}

\subsubsection{Exponentielle et accouplement $[\;,\;]_{\PP^1}$} \label{accexp}

On définit, pour $\eta$ un caractère d'ordre fini, $i = 1, 2$ et $m \in \Z$, la fonction $g^i_{\eta, m} \in \mathrm{LA}_{\mathrm{c}}(\qp, \Delta_{\rm dif}^-)^\Gamma$ par la formule $$ g^i_{\eta, m}(p^n a) = \left\{
  \begin{array}{c c}
    (-1)^{i + 1} \sigma_a (G(\eta)^{-1} \Omega t^{-1}) \otimes f_i & \quad \text{ si $n = m$} \\
    0 & \quad \text{si $n \neq m$}  \\
  \end{array}
\right. $$
Pour $a \in \zpe$, on a donc \[ g^i_{\eta, m}(p^m a) = (-1)^{i + 1} \eta(a) \det_\Delta(a) \cdot (G(\eta)^{-1} \Omega) \cdot (at)^{-1} \otimes f_i. \] 

\begin{lemme} \label{exp2}
Soient $z \in \Delta^{\psi = 1}$, $\eta \colon \zpe \to L^\times$ un caractère d'ordre fini, $i \in \{ 1, 2\}$ et ${\j} \geq 1$. Alors \footnote{Rappelons une dernière fois que $\tilde{z} \in (\Pi(\Delta)^* \otimes \omega_\Delta)^{{\matrice p 0 0 1} = 1}$ l'élément correspondant à $z$ par le lemme \ref{psidist}.}
\[ [ \partial^{-{\j}} \tilde{z}, g^{3 - i}_{\eta, m}]_{\PP^1} = (-1)^{i + 1} \frac{p}{p-1} \omega_\Delta(p)^{-m} p^{j(m + n)} \eta(-1) p^{-n} \, \mathrm{Tr}_{L_n / L}( G(\eta)^{-1} \langle [\varphi^{n} \partial^{-j} z]_0, f_i \rangle_{\rm dR} ). \]
\end{lemme}

\begin{proof}
Il suffit de traiter le cas $i = 1$. On a $\tilde{z} \in (\Pi(\Delta)^* \otimes \omega_\Delta)^{{\matrice p 0 0 1} = 1}$, $\partial^{-{\j}} \tilde{z} \in (\Pi(\Delta)^* \otimes \omega_\Delta)^{{\matrice p 0 0 1} = p^j}$ et on a, pour $n \gg 0$,
\begin{eqnarray*} [ \partial^{-{\j}} \tilde{z}, g^2_{\eta, m}]_{\PP^1} &=& \sum_{i \in \Z} \omega_\Delta(p^{-i}) p^{j(i + n)} [\varphi^{-n} \partial^{-j} z,  g^2_{\eta, m}(p^i) )]_{\rm dif} \\
&=& - \omega_\Delta(p)^{-m} p^{j(m + n)} [ \varphi^{-n} \partial^{-j} z, G(\eta)^{-1} \Omega t^{-1} \otimes f_2 )]_{\rm dif},
%&=& \omega_\Delta(p)^{-m} \alpha^{-m-n} p^{j(m + n)} \omega_\Delta(-1) \text{rés}_L( (  \varphi^{-n} \partial^{-{\j}} z \wedge  \sigma_{-1} (G(\eta)^{-1} \Omega t^{-1} \otimes f_2) ) \otimes e_{\omega_\Delta}^\vee),
\end{eqnarray*}
où la première égalité est la formule pour l'accouplement $[\;,\;]_{\PP^1}$ en termes du modèle de Kirillov (cf. formule \eqref{formulemagique2}), et la deuxième suit de ce que $g^2_{\eta, m}(p^i) = 0$ dès que $i \neq m$. Par définition de l'accouplement $[\;,\;]_{\rm dif}$ (cf. formule \eqref{accdif}), et en notant que $\sigma_{-1}(G(\eta)^{-1}) = \eta(-1) G(\eta)^{-1}$, $\sigma_{a}(\Omega t^{-1}) = \det_\Delta(a) a^{-1} \Omega t^{-1} = \omega_\Delta(a) \, \Omega t^{-1}$, on obtient 
\[ [ \varphi^{-n} \partial^{-j} z, G(\eta)^{-1} \Omega t^{-1} \otimes f_2 )]_{\rm dif}  = \eta(-1) \, \text{rés}_L \big( (  \varphi^{-n} \partial^{-j} z \wedge (G(\eta)^{-1} \Omega t^{-1} \otimes f_2 )) \otimes e_{\omega_\Delta}^\vee \big). \]
%$$ \text{rés}_L( (  \varphi^{-n} \partial^{-j} z \wedge  \sigma_{-1} (G(\eta)^{-1} \Omega t^{-1} \otimes f_2) ) \otimes e_{\omega_\Delta}^\vee) = \omega_\Delta(-1) \eta(-1) \text{rés}_L( (  \varphi^{-n} \partial^{-j} z \wedge (G(\eta)^{-1} \Omega t^{-1} \otimes f_2 )) \otimes e_{\omega_\Delta}^\vee). $$
Or, la formule \[ x \wedge (h \otimes f_i) = h \langle x, 1 \otimes f_{3-i} \rangle_{\rm dR} ((1 \otimes f_{3-i}) \wedge (1 \otimes f_i)) = h \langle x, 1 \otimes f_{3-i} \rangle_{\rm dR} (-1)^{i} \Omega^{-1} dt \otimes e_{\omega_\Delta}, \] $x \in \Delta_{\rm dif}$, $h \in L_\infty((t))$ et $i = 1, 2$, et la définition de $\text{rés}_L$ comme la composée du résidu avec la trace de Tate normalisée donnent
$$ \text{rés}_L \big( (  \varphi^{-n} \partial^{-j} z \wedge (G(\eta)^{-1} \Omega t^{-1} \otimes f_2 )) \otimes e_{\omega_\Delta}^\vee \big) = - \frac{p}{p-1} p^{-n} \, \mathrm{Tr}_{L_n / L} ( G(\eta)^{-1} \langle  [ \varphi^{-n} \partial^{-j} z ]_0, f_1 \rangle_{\rm dR}). $$ On conclut alors que
$$ [ \partial^{-{\j}} \tilde{z}, g^2_{\eta, m}]_{\PP^1} = \frac{p}{p-1} \omega_\Delta(p)^{-m} p^{j(m + n)} \eta(-1) p^{-n} \, \mathrm{Tr}_{L_n / L}( G(\eta)^{-1} \langle [\varphi^{n} \partial^{-j} z]_0, f_1 \rangle_{\rm dR} ), $$ ce qui finit la preuve du lemme.
\end{proof}

\begin{prop} \label{formexp} Soient $z \in \Delta^{\psi = 1}$, $\eta \colon \zpe \to L^\times$ un caractère d'ordre fini, $i \in \{ 1, 2\}$ et ${\j} \geq 1$. Alors $$ \langle \exp^{-1}(\int_\Gamma \eta \chi^{\j} \cdot \mu_z) \otimes \mathbf{e}^{\rm dR, \vee}_{\eta, j} , f_i \rangle_{\rm dR} = (-1)^{i + 1} \frac{p - 1}{p} (-1)^{j} (j - 1)! \eta(-1) \omega_\Delta(p)^m p^{-jm} [ \partial^{-j} \tilde{z}, g^{3 - i}_{\eta, m}]_{\PP^1}. $$
\end{prop}

\begin{proof}
%On a, pour $n \gg 0$,
%\begin{equation} \label{asdasd1}
%\exp^{-1}(\int_\Gamma \eta \chi^{\j} \cdot \mu_z) = (-1)^{\j} ({\j}-1)! \alpha^{-n} p^{-n} \, \mathrm{Tr}_{L_n / L} ( [\varphi^{-n} (\partial^{-j} z \otimes e_\eta \otimes t^{-j} e_j)]_0).
%\end{equation}
%et, comme $\varphi^{-n}(e_\eta \otimes t^{-j} e_j) = p^{nj} \, e_\eta \otimes t^{-j} e_j$, on a $$ \mathrm{Tr}_{L_n / L} ( [\varphi^{-n} (\partial^{-{\j}} z \otimes e_\eta \otimes t^{-j} e_j)]_0) = p^{nj} \, \mathrm{Tr}_{L_n / L} ( [\varphi^{-n} \partial^{-{\j}} z \otimes e_\eta \otimes t^{-j} e_j]_0). $$ 
%Comme précédemment, l'élément $\varphi^{-n} \partial^{-{\j}} z \otimes e_\eta \otimes t^{-j} e_j$ appartient à $L_n[[t]] \otimes \DdR(\Delta(\eta\chi^{\j})) = L_n[[t]] \otimes M_{\rm dR} \otimes L \cdot \mathbf{e}^{\rm dR}_{\eta, j} $ et peux donc être exprimé sous la forme $ x \otimes \mathbf{e}^{\rm dR}_{\eta, j}$, avec $x = G(\eta)^{-1} \varphi^{-n} \partial^{-j} z \in L_n[[t]] \otimes M_{\rm dR}$. On a alors $$ [\varphi^{-n} \partial^{-{\j}} z \otimes e_\eta \otimes t^{-{\j}} e_{\j}]_0 = G(\eta)^{-1} [\varphi^{-n} \partial^{-j} z]_0 \otimes \mathbf{e}^{\rm dR}_{\eta, j}. $$ Notons finalement que l'élément $\mathbf{e}^{\rm dR}_{\eta, j}$ est invariant par $\Gamma$ et commute donc à la trace, ce qui donne
%$$ \mathrm{Tr}_{L_n / L} ( [\varphi^{-n} (\partial^{-{\j}} z \otimes e_\eta \otimes t^{-{\j}} e_{\j})]_0) = p^{n{\j}} \, \mathrm{Tr}_{L_n / L}( G(\eta)^{-1} [\varphi^{-n} \partial^{-{\j}} z]_0) \otimes \mathbf{e}^{\rm dR}_{\eta, j}, $$ et,
En accouplant la formule du lemme \ref{exp1} avec $f_i$, on en déduit la formule 
\begin{equation} \label{eaea1} \langle \exp^{-1}(\int_\Gamma \eta \chi^{\j} \cdot \mu_z) \otimes \mathbf{e}^{\rm dR, \vee}_{\eta, j}, f_i \rangle_{\rm dR}  = (-1)^{\j} ({\j}-1)! p^{n(j-1)}  \mathrm{Tr}_{L_n / L} ( G(\eta)^{-1} \langle [\varphi^{-n} \partial^{-{\j}} z]_0, f_i \rangle_{\rm dR}).
\end{equation}
%d'où $$ \langle \exp^{-1}(\int_\Gamma \eta \chi^{\j} \cdot z) \otimes G(\eta)^{-1} e_{\eta}^\vee \otimes t^{\j} e_{-{\j}}, f_1 \rangle_{\rm dR} = (-1)^{\j} ({\j}-1)! p^{n({\j} - 1)} G(\eta)^{-1} \Omega ( ([\varphi^{-n} \partial^{\j} z]_0) \wedge f_2 ).

Par ailleurs, $\tilde{z} \in (\Pi(\Delta)^* \otimes \omega_\Delta)^{{\matrice p 0 0 1} = 1}$, $\partial^{-{\j}} \tilde{z} \in (\Pi(\Delta)^* \otimes \omega_\Delta)^{{\matrice p 0 0 1} = p^j}$ et, par le lemme \ref{exp2} on a, pour $n \gg 0$,
$$ [ \partial^{-{\j}} \tilde{z}, g^{3 - i}_{\eta, m}]_{\PP^1} = (-1)^{i + 1} \frac{p}{p-1} \omega_\Delta(p)^{-m} p^{j(m + n)} \eta(-1) p^{-n} \, \mathrm{Tr}_{L_n / L}( G(\eta)^{-1} \langle [\varphi^{n} \partial^{-j} z]_0, f_i \rangle_{\rm dR} ). $$
Enfin, en comparant avec la formule (\ref{eaea1}) on conclut que
$$ [ \partial^{-{\j}} \tilde{z}, g^{3 - i}_{\eta, m}]_{\PP^1} = (-1)^{i + 1} \frac{p}{(p-1)} \frac{(-1)^{\j}}{ ({\j} - 1)! } \eta(-1) \omega_\Delta(p)^{-m} p^{jm} \langle \exp^{-1}(\int_\Gamma \eta \chi^{\j} \cdot \mu_z) \otimes \mathbf{e}^{\rm dR, \vee}_{\eta, j}, f_i \rangle_{\rm dR}, $$
et on finit la preuve.
\end{proof}

%Soit $i \in \{1, 2 \}$. On dispose d'un isomorphisme $$ \mathrm{LA}_{\rm c}(\qpe, L_\infty)^\Gamma \otimes \mathrm{Sym}^{k - 1} \xrightarrow{\sim} \mathrm{LC}_{\rm c}(\qpe, (L_\infty[[t]] / t^k) \otimes \mathscr{L}_i)^\Gamma $$ donné par $$ \phi \otimes e_1^i e_2^{k - 1 - i} \mapsto [x \mapsto (k - 1 - i)! \phi(x) (xt)^i \otimes f_i]. $$ Sous cet isomorphisme, la fonction $f^i_{\eta, m}$ correspond à $$ \frac{1}{(k - 1)!} G(\eta)^{-1} \xi_{\eta, m} \otimes e_1^0 e_2^{k - 1} $$ et la fonction $h^i_{\eta^{-1}, m}$ correspond à $$ (\omega_\pi(p) p^{k - 1})^m \cdot (G(\eta^{-1})^{-1} \Omega) \cdot \xi_{\eta^{-1} \omega_\pi, m} \otimes e_1^{k - 1} e_2^0. $$

\subsubsection{Vecteurs localement algébriques} \label{vectlalg}

La représentation $\Pi(\Delta)$ ne possède pas de vecteurs localement algébriques et on va tordre l'action de $G$ de sorte que l'on fasse en apparaître, ce qui nous permettra de regarder les fonctions $f^i_{\eta, k, m}$ et $g^i_{\eta, m}$ comme des éléments du modèle de Kirillov d'une certaine représentation localement algébrique.

Reprenons les notations de la section \S \ref{Changement} et notons $\pi = \mathrm{LL}_p(\Delta)$ dans la suite. Rappelons qu'on a la relation $\omega_\pi = x \omega_\Delta$ entre les caractères centraux de $\Pi(\Delta)$ et $\pi$. Soit $k \geq 1$ et soit $i \in \{1, 2\}$. On peut regarder les fonctions $f^i_{\eta, k, m}$ et $g^i_{\eta, m}$ comme des éléments de $\mathrm{LA}_{\rm c}(\qpe, L_\infty^{-k} \otimes M_{\rm dR})^{\Gamma}$ qui, d'après le diagramme (\ref{Kirillovdiag}), s'identifie aux vecteurs localement algébriques $\Pi(\Delta, k)^{\rm alg}$ de la représentation $\Pi(\Delta, k)$.  Rappelons que l'on dispose des isomorphisme de $G$-modules (cf. remarque \ref{kiralg}) \[ \iota_i \colon \mathrm{LC}_{\rm c}(\qpe, L_\infty)^\Gamma \otimes \mathrm{Sym}^{k - 1} \otimes \det^{-k} \xrightarrow{\sim} \mathrm{LA}_{\rm c}(\qpe, L^{-k}_\infty \otimes M_{\rm dR})^\Gamma  , \] \[ \phi \otimes e_1^j e_2^{k - 1 - j} \mapsto [x \mapsto (k - 1 - j)! \phi(x) (xt)^{j-k} \otimes f_i], \;\;\; 0 \leq j \leq k - 1. \] Alors
\begin{itemize}
\item On a \footnote{Observons que $\eta$ est vu comme un caractère sur $\qpe$ en posant $\eta(p) = 1$.} $f^i_{\eta, k, m}(p^n a) = 0$ si $n \neq m$ et \begin{eqnarray*} (-1)^{i + 1} \, f^i_{\eta, k, m}(p^m a) &=& (k-1)! G(\eta)^{-1} \eta(a) (a t)^{-k} \otimes f_i \\
&=& (k-1)! p^{mk} G(\eta)^{-1} \eta(p^m a) (p^mat)^{-k} \otimes f_i \end{eqnarray*} et donc $$ \iota_i^{-1}(f^i_{\eta, k, m}) = (-1)^{i + 1} p^{mk} G(\eta)^{-1} \xi_{\eta, m} \otimes e_1^0 e_2^{k - 1}, $$ où $\xi_{\eta, m}$ est la fonction définie dans \S \ref{Kirillovaa}. 
\item On a $g^i_{\eta^{-1}, m}(p^n a) = 0$ si $n \neq m$ et
\begin{eqnarray*}
(-1)^{i + 1} \, g^i_{\eta^{-1}, m}(p^m a) &=& (G(\eta^{-1})^{-1} \Omega) \, \eta^{-1}(a) \mathrm{det}_\Delta(a) \, (a t)^{-1} \otimes f_i \\
&=& (\mathrm{det}_\Delta(p) p^{-1})^{-m} \, (G(\eta^{-1})^{-1} \Omega) \, \eta^{-1}(p^m a) \mathrm{det}_\Delta(p^m a) \, (p^m a t)^{- 1} \otimes f_i \\
&=& \omega_\Delta(p)^{-m} \, (G(\eta^{-1})^{-1} \Omega) \, \eta^{-1}(p^m a) \, \omega_\pi(p^m a) \, (p^m a t)^{- 1} \otimes f_i,
\end{eqnarray*}
où pour la dernière égalité on a utilisé $\det_\Delta(a) = \omega_\pi(a)$. D'où $$ \iota_i^{-1}(g^i_{\eta^{-1}, m}) = (-1)^{i + 1} \omega_\Delta(p)^{-m} \cdot (G(\eta^{-1})^{-1} \Omega) \cdot \xi_{\eta^{-1} \omega_\pi, m} \otimes e_1^{k - 1} e_2^0. $$
\end{itemize}

\begin{lemme} \label{formkir}
Soient $k \geq 1$ et $i \in \{ 1, 2 \}$. On a $$ w \ast_k f^i_{\eta, k, m} = (-1)^k p^{mk} \eta(-1) \omega_\Delta(p)^{-c(\pi \otimes \eta^{-1}) - m} \Omega^{-1} \frac{G(\eta^{-1})}{G(\eta)} \epsilon(\pi \otimes \eta^{-1}) g^i_{\eta^{-1}, -c(\pi \otimes \eta^{-1}) - m} $$
\end{lemme}

\begin{proof}
Par la discussion précédente et la proposition \ref{BHeqfonct}, on a \footnote{Le facteur $(-1)^k$ vient de l'action de $w$ sur le facteur $\det^{-k}$ du terme de gauche de l'isomorphisme $\iota_i$.}
\begin{eqnarray*}
w \ast_k f^i_{\eta, k, m} &=& \iota \left( {\matrice 0 1 1 0} \cdot \big( (-1)^{i + 1} p^{mk} G(\eta)^{-1} \, (\xi_{\eta, m} \otimes e_1^0 e_2^{k - 1}) \big) \right) \\
&=& (-1)^k (-1)^{i + 1} p^{mk} \eta(-1) G(\eta)^{-1} \epsilon(\pi \otimes \eta^{-1}) \iota \big( \xi_{\eta^{-1}\omega_\pi, -c(\pi \otimes \eta^{-1}) - m} \otimes e_1^{k - 1} e_2^0 \big) \\
&=& (-1)^k p^{mk} \eta(-1) \omega_\Delta(p)^{-c(\pi \otimes \eta^{-1}) - m} \Omega^{-1} \frac{G(\eta^{-1})}{G(\eta)} \epsilon(\pi \otimes \eta^{-1}) g^i_{\eta^{-1}, -c(\pi \otimes \eta^{-1}) - m},
\end{eqnarray*}
où, dans la deuxième égalité on a utilisé le fait que $\det {\matrice 0  1 1 0} = -1.$ Ceci permet de conclure.
\end{proof}

\subsection{L'équation fonctionnelle: le cas de poids de Hodge-Tate nuls}

%\subsubsection{Côté Iwasawa}

Soit $\Delta \in \Phi \Gamma(\Robba)$ De Rham, non triangulin et à poids de Hodge-Tate nuls comme dans \S \ref{phiGammadR}. Si $z \in \Delta^{\psi = 1}$, on note $w_\Delta(z) = \mathrm{Res}_\zp(w \cdot \tilde{z}) \in \Delta^{\psi = \omega_\Delta^{-1}(p)}$. L'élément $w_\Delta(z)$ est aussi (cf. lemme \ref{dualiteinv}) l'unique élément $x \in \Delta^{\psi = \omega_\Delta^{-1}(p)}$ tel que $(1 - \omega_\Delta^{-1}(p) \varphi) x = w_\Delta((1 - \varphi) z)$, où $w_\Delta \colon \Delta^{\psi = 0} \to \Delta^{\psi = 0}$ est l'involution définie par restriction à $\zpe$ de l'action de ${\matrice 0 1 1 0}$ sur $\Delta \boxtimes \PP^1$.
%dans \cite[\S V.2.2]{ColmezPhiGamma} pour le cas étale, et tordant convenablement en générale \footnote{Plus précisément, si $\delta \colon \qp \to L^\times$ est tel que $\Delta(\delta)$ est étale, alors $w_\Delta (z) := \delta(-1) (w_{\Delta(\delta)}(z \otimes e_\delta)) \otimes e_{\delta^{-1}}$. }.

On est maintenant en condition de montrer le résultat principal de cette section:

% Vieille version
%\begin{theorem} \label{eqfonct1}
%Soient $\alpha \in L^\times$ tel que $v_p(\alpha) \in \Z$, $z \in \Delta^{\psi = \alpha}$ et notons $\check{z} = w_\Delta(z) \otimes e_{\omega_\Delta}^\vee \in \check{\Delta}^{\psi = \alpha}$. Soient $\eta \colon \zpe \to L^\times$ un caractère d'ordre fini, $i \in \{ 1, 2 \}$ et $m \in \Z$. Alors \[ \exp^*(\int_\Gamma \eta \chi^{-j} \cdot \mu_{\check{z}}) \otimes \mathbf{e}^{\rm dR, \vee}_{\eta, -j, \omega_\Delta^\vee} = C(\Delta, \eta, j) \cdot \exp^{-1}(\int_\Gamma \eta^{-1} \chi^{j} \cdot \mu_z) \otimes \mathbf{e}^{\rm dR, \vee}_{\eta^{-1}, j}, \] où \[ C(\Delta, \eta, j) = \eta(-1) \frac{(-1)^{j}}{(j - 1)!^2} \Omega^{-1} \alpha^{c(\pi \otimes \eta^{-1})} \frac{G(\eta^{-1})}{G(\eta)} \epsilon(\pi \otimes \eta^{-1} \otimes | \cdot |^j) , \] pour tout $j \geq 1$.
%\end{theorem}

\begin{theorem} \label{eqfonct1}
Soient $z \in \Delta^{\psi = 1}$ et notons $\check{z} = w_\Delta(z) \otimes e_{\omega_\Delta}^\vee \in \check{\Delta}^{\psi = 1}$. Soient $\eta \colon \zpe \to L^\times$ un caractère d'ordre fini, $i \in \{ 1, 2 \}$ et $m \in \Z$. Alors \[ \exp^*(\int_\Gamma \eta \chi^{-j} \cdot \mu_{\check{z}}) \otimes \mathbf{e}^{\rm dR, \vee}_{\eta, -j, \omega_\Delta^\vee} = C(\Delta, \eta, j) \cdot \exp^{-1}(\int_\Gamma \eta^{-1} \chi^{j} \cdot \mu_z) \otimes \mathbf{e}^{\rm dR, \vee}_{\eta^{-1}, j}, \] où \[ C(\Delta, \eta, j) = \Omega^{-1} \frac{(-1)^{j}}{(j - 1)!^2} \, \epsilon(\eta^{-1})^{-2} \epsilon(\pi \otimes \eta^{-1} \otimes | \cdot |^j) \] pour tout $j \geq 1$.
\end{theorem}

\begin{proof}
Il suffit de montrer le résultat en projetant sur $f_i$, pour $i \in \{ 1, 2 \}$. Soient $j \geq 0$ et $k \geq 1$ et notons
\[ A_i = (-1)^{i + 1} \langle \exp^*(\int_\Gamma \eta \chi^{-j-k} \cdot \mu_{\check{z}}) \otimes \mathbf{e}^{\rm dR, \vee}_{\eta, -j-k, \omega_\Delta^\vee}, f_i \rangle_{\rm dR}, \]
\[ B_i = (-1)^{i + 1} \langle \exp^{-1}(\int_\Gamma \eta^{-1} \chi^{j + k} \cdot \mu_z) \otimes \mathbf{e}^{\rm dR, \vee}_{\eta^{-1}, j+k} , f_i \rangle_{\rm dR}. \]

En posant $m = 0$ dans la formule de la proposition \ref{formexpd} on obtient
\[  A_i = \frac{p-1}{p} \eta(-1) \frac{\omega_\Delta(-1) (-1)^k}{(j+k-1)!} [ \partial^{\j} w \cdot \tilde{z}, f^{3-i}_{\eta, k, 0}]_{\PP^1}. \]
Or, on a
\[ \partial^j w \cdot \tilde{z} = w \cdot \partial^{-j} \tilde{z} = w \cdot \partial^k \partial^{-j-k} \tilde{z} = w \ast_{-k} \partial^{-j-k} \tilde{z} \]
et, en utilisant la formule $[g \ast_{-k} x, y]_{\PP^1} = \omega_\Delta(\det \; g)[x, g^{-1} \ast_k y]_{\PP^1}$ du lemme \ref{propinv}, on en déduit
\[ [ \partial^{\j} w \cdot \tilde{z}, f^{3-i}_{\eta, k, 0}]_{\PP^1} = \omega_\Delta(-1) [ \partial^{-j-k} \tilde{z}, w \ast_k f^{3-i}_{\eta, k, 0}]_{\PP^1}. \]
Grâce au lemme \ref{formkir}, on sait que
\[ [ \partial^{-j-k} \tilde{z}, w \ast_k f^{3-i}_{\eta, k, 0}]_{\PP^1} = (-1)^k \eta(-1) \omega_\Delta(p)^{-c(\pi \otimes \eta^{-1})} \Omega^{-1} \frac{G(\eta^{-1})}{G(\eta)} \epsilon(\pi \otimes \eta^{-1}) [ \partial^{-j-k} \tilde{z}, g^{3 - i}_{\eta^{-1}, -c(\pi \otimes \eta^{-1})}]_{\PP^1}. \]
Enfin, en appliquant la proposition \ref{formexp} avec $j = j + k$, on a
\[ [ \partial^{-j-k} \tilde{z}, g^{3 - i}_{\eta^{-1}, -c(\pi \otimes \eta^{-1})}]_{\PP^1} = \frac{p}{p-1} \frac{(-1)^{j + k}}{(j + k - 1)!} \eta(-1) (\omega_\Delta(p) p^{- j- k})^{c(\pi \otimes \eta^{-1})} B_i . \] Observons que $\epsilon(\pi \otimes \eta^{-1}) p^{-c(\pi \otimes \eta^{-1})(j + k)} = \epsilon(\pi \otimes \eta^{-1} \otimes | \cdot |^{j + k})$. En simplifiant toutes ces jolies formules, et en utilisant la formule $\eta(-1) \frac{G(\eta^{-1})}{G(\eta)} = p^n G(\eta)^{-2} = \epsilon(\eta^{-1})^{-2}$, on obtient \[ A_i = \Omega^{-1} \frac{(-1)^{j+k}}{(j + k - 1)!^2} \, \epsilon(\eta^{-1})^{-2} \epsilon(\pi \otimes \eta^{-1} \otimes | \cdot |^{j + k}) \cdot B_i. \] Comme $j \geq 0$ et $k \geq 1$ ont été choisis arbitrairement, ceci achève la démonstration.

\end{proof}

\begin{remarque} \leavevmode
\begin{itemize}
%\item Le même genre de calculs ont été faits par Nakamura (cf. \cite{Nakamura2}, proposition $3.15$) dans la bande critique, ce qui lui permet de démontrer la conjecture $\epsilon$ locale de Kato (cf. \cite{Nakamura2} thm. 1.1) pour les représentations de de Rham de dimension $2$ à poids de Hodge-Tate $k_1 \leq 0$ et $k_2 \geq 1$. Les calculs ci-dessus permettent, de la même manière, prouver cette conjecture pour toute représentation de de Rham de dimension $2$, cf. \cite{epsilonKato}.
\item Si $\eta \colon \zpe \to L^\times$ est le caractère trivial, la constante $C(\Delta, \eta, j)$ dévient \[ C(\Delta, \eta, j) = \frac{(-1)^{j}}{(j - 1)!^2} \Omega^{-1} \epsilon(\pi \otimes | \cdot |^j). \]
\item La valeur $\Omega^{-1}$ n'est que superflue car elle apparaît aussi dans le terme de gauche comme facteur dans $\mathbf{e}^{\rm dR, \vee}_{\eta, -j, \omega_\Delta^\vee}$. L'équation fonctionnelle du théorème ne dépend donc pas du choix de la base $f_1$ et $f_2$ de $\DdR(\Delta)$.
\end{itemize}
\end{remarque}

\subsection{L'équation fonctionnelle: cas de poids de Hodge-Tate $0$ et $k \geq 1$} \label{eqeq}

L'équation fonctionnelle du théorème \ref{eqfonct1} peut être tordue pour obtenir une équation fonctionnelle analogue quand le $(\varphi, \Gamma)$-module est de Rham à poids de Hodge-Tate $0, k \geq 1$. Commençons par fixer quelques notations.

%\subsubsection{Notations} \label{eqfonctknot}

Soient $k \geq 1$, $\mathscr{L} \subseteq M_{\rm dR}$ une droite et notons $D = \Delta_{k, \mathscr{L}}$. On a alors $\omega_D = \omega_\Delta x^k$ et l'involution $w_D$ sur $D^{\psi = 0}$ est donnée par (la restriction à $D^{\psi = 0}$ de) l'action $w_\Delta \ast_{-k} z =  \partial^{-k} w_\Delta \cdot z =  w_\Delta \partial^k \cdot z$, $z \in D^{\psi = 0}$. Notons \[ \check{\Delta}[k] = \Delta \otimes \omega_\Delta^{-1} x^{-k}. \] Ce module contient tous les duaux de Tate des $(\varphi, \Gamma)$-modules de la forme $\Delta_{k, \mathscr{L}}$ (en effet, comme $\Delta_{k, \mathscr{L}}$ est de dimension $2$, son dual de Tate s'identifie à $\Delta_{k, \mathscr{L}} \otimes \omega_{\Delta_{k, \mathscr{L}}}^{-1} = \Delta_{k, \mathscr{L}} \otimes \omega_\Delta^{-1} x^{-k} \subseteq \Delta \otimes \omega_\Delta^{-1} x^{-k}$). Le module $D$ est à poids de Hodge-Tate $0$ et $k$ et les poids de Hodge-Tate de $\check{D}$ sont dont $1 - k$ et $1$, d'où $\check{D}(k-1)$ est à poids $0, k$.

Soit $z \in D^{\psi = 1} \subseteq \Delta^{\psi = 1}$ et soit $y = (1 - \varphi) z \in D^{\psi = 0}$. Posons
\[ \check{y} = w_\Delta(y) \otimes e_{\omega_\Delta}^\vee \in \check{\Delta}^{\psi = 0}; \;\;\; \check{z} = \mathrm{Res}_\zp (w_\Delta(\tilde{z})) \otimes e_{\omega_\Delta}^\vee \in \check{\Delta}^{\psi = 1}, \]
\[ \check{y}[k] = w_D(y) \otimes e_{\omega_D}^\vee \in \check{\Delta}[k]^{\psi = 0} ; \;\;\; \check{z}[k] = \mathrm{Res}_\zp (w_D(\tilde{z})) \otimes e_{\omega_D}^\vee \in \check{\Delta}[k]^{\psi = 1}, \] On a bien \[ \check{y} = (1 - \varphi) \check{z}, \;\;\; \check{y}[k] = (1 - \varphi) \check{z}[k]. \] En effet,
\begin{eqnarray*}
(1 - \varphi) \check{z} &=& (1 - \varphi) (\mathrm{Res}_\zp (w_\Delta(\tilde{z})) \otimes e_{\omega_\Delta}^\vee ) = (1 - \omega_\Delta(p)^{-1} \varphi) \mathrm{Res}_\zp (w_\Delta(\tilde{z})) \otimes e_{\omega_\Delta}^\vee  \\
&=& \mathrm{Res}_\zpe \mathrm{Res}_\zp (w_\Delta(\tilde{z})) \otimes e_{\omega_\Delta}^\vee = w_\Delta(\mathrm{Res}_\zpe(z)) \otimes e_{\omega_\Delta}^\vee = \check{y}.
\end{eqnarray*}

%\subsubsection{Côté Iwasawa}

L'équation fonctionnelle du théorème \ref{eqfonct1} en poids de Hodge-Tate positifs prend la forme suivante:

\begin{theorem} \label{eqfonct1b}
Les notations comme ci-dessus. On a, pour tout $j \geq 0$, \[ \exp^*(\int_\Gamma \eta \chi^{-j} \cdot \mu_{\check{z}[k]}) \otimes \mathbf{e}^{\rm dR, \vee}_{\eta, -j, \omega_D^\vee} = C(D, \eta,  j) \cdot \exp^{-1}(\int_\Gamma \eta^{-1} \chi^{j} \cdot \mu_z) \otimes \mathbf{e}^{\rm dR, \vee}_{\eta^{-1}, j}, \] où \[ C(D, \eta, j) = - \Omega^{-1} \, \frac{\Gamma^*(-j + 1)}{\Gamma^*(j +k)}  \, \epsilon(\eta^{-1})^{-2} \epsilon(\pi \otimes \eta^{-1} \otimes | \cdot |^j). \]
\end{theorem}

\begin{proof}
Rappelons que l'opérateur $\partial$ commute à la restriction et l'on a donc $\check{z}[k] = \partial^{-k} \mathrm{Res}_\zp(w_\Delta(\tilde{z})) \otimes e_{\omega_\Delta}^\vee \otimes e_{x^{-k}} = \partial^{-k} \check{z} \otimes e_{x^{-k}}$. On a alors, pour $n \gg 0$,
\begin{eqnarray*}
\exp^*(\int_\Gamma \eta \chi^{-j} \cdot \mu_{\check{z}[k]}) &=&  p^{-n} \, \mathrm{Tr}_{L_n / L}([\varphi^{-n}(\partial^{-k} \check{z} \otimes e_\eta \otimes e_{-j} \otimes e_{x^{-k}})]_0) \\
&=&  p^{-n} \, \mathrm{Tr}_{L_n / L}([ t^{-k} \partial^{-k} \varphi^{-n} (\check{z} \otimes e_\eta \otimes e_{-j}) ]_0) \otimes t^k e_{x^{-k}}.
\end{eqnarray*}
Observons que, comme dans les démonstrations des propositions \ref{formexpd} et \ref{formexp}, le terme $[ t^{-k} \partial^{-k} \varphi^{-n} (\check{z} \otimes e_\eta \otimes e_{-j}) ]_0$ appartient à $\DdR(\Delta \otimes \omega_{\Delta}^\vee \otimes \eta \chi^{-j}) = \DdR(D) \otimes \mathbf{e}_{\eta, -j, \omega_\Delta^\vee}^{\rm dR}$. Si $\varphi^{-n} (\mathrm{Res}_\zp(w_\Delta(\tilde{z}))) = \sum_n a_l t^l$, en utilisant le fait que le terme de degré $l$, $l \geq 0$, en $t$ de $t^{-k} \partial^{-k} (a_l t^l)$ est $\frac{1}{(l+1) (l + 2) \hdots (l + k)} a_l t^l$ on en déduit
%\begin{eqnarray*}
\[
[ t^{-k} \partial^{-k} \varphi^{-n} (\check{z} \otimes e_\eta \otimes e_{-j}) ]_0 = \frac{1}{j (j+1) \cdots (j +k -1)} [ \varphi^{-n} (\check{z} \otimes \eta \otimes e_{-j}) ]_0 \]
%&=& \omega_\Delta(p)^n G(\eta)^{-1} \Omega [t^{-k} \partial^{-k} \sum_l a_l t^l ]_{j - 1} \otimes \mathbf{e}^{\rm dR}_{\eta, -j, \omega_\Delta^\vee} \\
%&=& \omega_\Delta(p)^n G(\eta)^{-1} \Omega \frac{1}{j (j+1) \cdots (j +k -1)} a_{j - 1} \otimes \mathbf{e}^{\rm dR}_{\eta, -j, \omega_\Delta^\vee} \\
%&=& \frac{1}{j (j+1) \cdots (j +k -1)} [ \varphi^{-n} (\check{z} \otimes \eta \otimes e_{-j}) ]_0.
%\end{eqnarray*}
Comme $\mathbf{e}^{\rm dR, \vee}_{\eta, -j, \omega_D^\vee} \otimes t^k e_{x^{-k}} = \mathbf{e}^{\rm dR, \vee}_{\eta, -j, \omega_\Delta^\vee}$, on en déduit
\[ \exp^*(\int_\Gamma \eta \chi^{-j} \cdot \mu_{\check{z}[k]}) \otimes \mathbf{e}^{\rm dR, \vee}_{\eta, -j, \omega_D^\vee} = \frac{1}{ j (j+1) \cdots (j + k - 1)} \exp^*(\int_\Gamma \eta \chi^{-j} \cdot \mu_{\check{z}}) \otimes \mathbf{e}^{\rm dR, \vee}_{\eta, -j, \omega_\Delta^\vee}, \]

Par le théorème \ref{eqfonct1}, on a 
\[ \exp^*(\int_\Gamma \eta \chi^{-j} \cdot \mu_{\check{z}}) \otimes \mathbf{e}^{\rm dR, \vee}_{\eta, -j, \omega_\Delta^\vee} = C(\Delta, \eta, j) \cdot \exp^{-1}(\int_\Gamma \eta^{-1} \chi^{j} \cdot \mu_z) \otimes \mathbf{e}^{\rm dR, \vee}_{\eta^{-1}, j}, \] où \[ C(\Delta, \eta, j) = \Omega^{-1} \, \frac{(-1)^{j}}{(j - 1)!^2} \, \epsilon(\eta^{-1})^{-2} \epsilon(\pi \otimes \eta^{-1} \otimes | \cdot |^j), \]
d'où le résultat.
%\small{ \[ \exp^*(\int_\Gamma \eta \chi^{-j} \cdot \mu_{\check{z}[k]}) \otimes \mathbf{e}^{\rm dR, \vee}_{\eta, -j, \omega_D^\vee} = \eta(-1) \frac{(-1)^{j}}{(j + k - 1)! (j - 1)!} \Omega^{-1} \frac{G(\eta^{-1})}{G(\eta)} \epsilon(\pi \otimes \eta^{-1} \otimes | \cdot |^j) \cdot  \exp^{-1}(\int_\Gamma \eta^{-1} \chi^{j} \cdot \mu_z) \otimes \mathbf{e}^{\rm dR, \vee}_{\eta^{-1}, j}, \]} \normalsize ce qui montre le lemme.
\end{proof}

\begin{remarque}
Si $\eta \colon \zpe \to L^\times$ est le caractère trivial, la constante $C(D, \eta, j)$ du théorème \ref{eqfonct1b} dévient \[ C(D, \eta, j) = - \Omega^{-1} \, \frac{\Gamma^*(-j + 1)}{\Gamma^*(j +k)} \, \epsilon(\pi \otimes | \cdot |^j). \]
\end{remarque}

\section{La conjecture $\epsilon$ locale de Kato en dimension $2$}

\subsection{Notations}

Fixons les notations dont on aura besoin dans la suite, pour lesquelles on renvoie à \cite{Nakamura2}, \cite{Nakamura3}.

\begin{itemize}
\item Soit $A$ une $\zp$-algèbre commutative satisfaisant l'une des deux conditions suivantes
\begin{itemize}
\item $A$ est un anneau semi-local noethérien complet pour la topologie définie pour son idéal de Jacobson $\mathrm{Jac}(A)$ et tel que $A / \ \mathrm{Jac}(A)$ est un anneau fini d'ordre une puissance de $p$.
\item $A = L$ est une extension finie de $\qp$.
\end{itemize}
Si $A$ satisfait une des ces conditions, on dit que $A$ est une $\zp$-algèbre de type (*). Pour une telle $A$, on note $\Robba_A$ l'anneau de Robba relatif sur $A$ (cf. \S \ref{phiGammarel}) et notons $\Phi\Gamma^{\text{ét}}(\Robba_A)$ la catégorie des $(\varphi, \Gamma)$-modules étales sur $\Robba_A$.

\item Soit $\D_{\rm perf}(A)$ la catégorie des complexes de $A$-modules parfaits (i.e quasi-isomorphes à un complexe borné de $A$-modules projectifs de type fini). On note \[ \mathrm{Det}_A: \D_{\rm perf}(A) \to \mathfrak{Inv}_A \] le foncteur déterminant (cf. \cite[\S 3A]{Nakamura3}) vers la catégorie des $A$-modules inversibles (i.e localement libres de rang $1$). Si $P$ est un $A$-module projectif, on a $\mathrm{Det}_A(P) = \wedge_A^{r_P} P$, où $r_P: \mathrm{Spec}\,A \to \Z$ dénote le rang de $P$. On note $\mathbf{1}_A = \mathrm{Det}_A(0) = A$ est l'objet unité (relatif au produit induit par le produit tensoriel).

\item Soit $D \in \Phi\Gamma^{\text{ét}}(\Robba_A)$ de rang $2$ et de Rham. On pose \[ \Delta_{A, 1}(D) = \mathrm{Det}_A(\mathscr{C}^\bullet_{\varphi, \gamma}(D)) \] le déterminant de la $\varphi, \gamma$-cohomologie de $D$ (cf. \S \ref{cohomphigamma}). Soit $\E_A$ définit comme précédemment si $A = L$ est une extension finie de $\qp$ et $\E_A = \varprojlim A / \mathrm{Jac}(A)^n [[T]][T^{-1}]$ autrement. Notons $D_0 \in \Phi\Gamma^{\text{ét}}(\E_A)$ le $(\varphi, \Gamma)$-module correspondant à $D$. On définit
\[ \mathscr{L}_A(D) := \{ x \in \det_{\E_A} D_0 \mid \varphi(x) = \det_D(p) x, \; \sigma_a(x) = \det_D(a) \cdot x, \;\;\; a \in \zpe \} = A \cdot e_D, \] et on pose \[ \Delta_{A, 2}(D) = \mathscr{L}_A(D). \]
Enfin, on définit la droite fondamentale locale de $D$ comme \[ \Delta_{A}(D) = \Delta_{A, 1}(D) \otimes \Delta_{A, 2}(D). \] Cette droite est compatible au changement de base, multiplicative pour les suites exactes courtes et se comporte bien sous la dualité.
\end{itemize}

Remarquons que, si $T \in \mathrm{Rep}_A \mathscr{G}_\qp$ et $i \in \{ 1, 2 \}$, on peut définir des modules $\Delta_{A, i}(T) \in \mathfrak{Inv}_A$ (cf. \cite[\S 1.7]{Kato2b}, \cite[\S 2.1.2]{Nakamura2}, \cite[\S 3B]{Nakamura3}) et on a (\cite[Corollary 3.2]{Nakamura3}, \cite[\S 2.2.2]{Nakamura2}) des isomorphismes canoniques
\[ \Delta_{A, i}(T) \cong \Delta_{A, i}(\D(T)). \]

\subsection{L'isomorphisme $\epsilon^{\rm dR}_{L}$}

Soient $A = L$ une extension finie de $\qp$ et $D \in \Phi\Gamma^{\text{ét}}(\Robba)$ de Rham. On a besoin (cf. \cite[\S 2.1.3]{Nakamura2}) de trois ingrédients pour définir l'isomorphisme $\epsilon^{\rm dR}_{L}$:

\begin{itemize}
\item La suite exacte fondamentale et les propriétés d'adjonction entre les applications exponentielle et exponentielle duale induisent (\cite[\S 3C eq. (25)]{Nakamura3}) la suite exacte 
\begin{eqnarray*} 0 &\to& H^0_{\varphi, \gamma}(D) \to \Dcris(D) \to \Dcris(D) \oplus t_D \to H^1_{\varphi, \gamma}(D) \\
&\to& \Dcris(\check{D})^* \oplus t_{\check{D}}^* \to \Dcris(\check{D})^* \to H^0_{\varphi, \gamma}(\check{D})^* \to 0,
\end{eqnarray*}
où $t_D = \DdR(D) / \mathrm{Fil}^0 \, \DdR(D)$. Par dualité, on sait que $H^0_{\varphi, \gamma}(\check{D})^* = H^2_{\varphi, \gamma}(D)$ et $(t_{\check{D}})^* = \mathrm{Fil}^0 \, \DdR(D)$. Le déterminant du complexe ci-dessus nous fournit donc un isomorphisme \[ \theta_L(D): \mathbf{1}_L \xrightarrow{\sim} \Delta_{L, 1}(D) \otimes \mathrm{Det}_L(\DdR(D)). \]

\item On note, pour $r \in \Z$, $n_r = \dim_L \mathrm{gr}^{-r} \DdR(D)$ et on définit \[ \Gamma_L(D) = \prod_{r \in \Z} \Gamma^*(r)^{-n_r}, \] où $ \Gamma^*(r) = \left\{
  \begin{array}{c c}
    (r - 1)! & \quad \text{ si $r > 0$} \\
    \frac{(-1)^r}{(-r)!} & \quad \text{si $r \leq 0$}  \\
  \end{array}.
\right. $

\item Soit $h_D = \sum_{r \in \Z} r \, n_r$. On a un isomorphisme (où on voit les deux espaces dans $\mathrm{Det}_{L_\infty((t))}(D_{\rm dif}))$
\[ \theta_{\rm dR, L}(D): \mathrm{Det}_L (\DdR(D)) \xrightarrow{\sim} \mathscr{L}_L(D), \;\;\; x \mapsto \epsilon(\D_{\rm pst}(D)) t^{h_D} x, \]
où $\D_{\rm pst}(D) = \D_{\rm pst}(\mathbf{V}(D))$ est le $(\varphi, N, \mathscr{G}_\qp)$-module filtré associé à $D$ (cf. \cite[\S 3.2]{Berger08} et remarque \ref{dpst}) et $ \epsilon(\D_{\rm pst}(D))$ dénote le facteur local de la représentation de Weil-Deligne (\cite{Deligne}) \footnote{La définition des facteurs epsilon dépend du choix d'un caractère additif ainsi que d'une mesure de Haar sur $\qp$ que l'on fixe comme dans \S \ref{replisses}. On sait, grâce à la compatibilité locale-globale dans la correspondance de Langlands $p$-adique \cite{Emerton2}, que les facteurs locaux des représentations de Weil-Deligne coïncident avec les facteurs locaux des représentations lisses fournies par la théorie du modèle de Kirillov. On pourrait donc définir les isomorphismes $\epsilon$ de de Rham en utilisant ces derniers et, dans ce cas-là, la preuve de la conjecture présentée dans ce chapitre ne dépendrait pas de telle compatibilité, et elle serait en conséquence purement locale.}. 
\end{itemize}

Avec les applications définies ci-dessus, on définit
\[ \epsilon^{\rm dR}_{L}: \mathbf{1}_L \xrightarrow{\Gamma_L(D) \theta_L(D)} \Delta_{L, 1}(D) \otimes \mathrm{Det}(\DdR(D)) \xrightarrow{1 \otimes \theta_{\rm dR, L}(D)} \Delta_L(D). \]

\subsection{Énoncé de la conjecture}

Rappelons (cf. \cite[\S 2.1.4]{Nakamura2}) la conjecture en question.

\begin{conjecture} \label{eaeaconj} Il existe une unique famille d'isomorphismes \[ \epsilon_{A}(D): \mathbf{1}_A \xrightarrow{\sim} \Delta_A(D), \] pour tout triplet $(A, D)$, avec $A$ une $\zp$ algèbre de type (*) et $D \in \Phi\Gamma^{\text{ét}}(\Robba_A)$, satisfaisant les propriétés suivantes
\begin{enumerate}
\item Pour tout morphisme $A \to A'$ de $\zp$-algèbres, on a
\[ \epsilon_{A}(D) \otimes 1_{A'} = \epsilon_{A'}(D \otimes_A A'). \]
\item Pour toute suite exacte $0 \to D' \to D \to D'' \to 0$ on a \[ \epsilon_{A}(D) = \epsilon_{A}(D') \otimes \epsilon_{A}(D''). \]
\item Pour tout $a \in \zpe$, si l'on remplace le système $(\zeta_{p^n})_{n \in \N}$ par $(\sigma_a(\zeta_{p^n}))_{n \in \N} = (\zeta_{p^n}^a)_{n \in \N}$ et l'on note $\epsilon_{A, \sigma_a(\zeta)}$ la nouvelle famille d'isomorphismes ainsi obtenues, on a
\[ \epsilon_{A, \sigma_a(\zeta)}(D) = \det_D(a) \, \epsilon_{A, \zeta}(D). \]
\item La composition \[ \mathbf{1}_A \xrightarrow{\epsilon_{A, \zeta}(D)} \Delta_A(D) \xrightarrow{\sim} \Delta_A(\check{D})^* \xrightarrow{\epsilon_{A, \sigma_{-1}(\zeta)}(\check{D})^*} \mathbf{1}_A \] donne l'identité, où l'isomorphisme du milieu est celui induit par la dualité locale.
\item Pour toute paire $(L, D)$, où $L$ est une extension finie de $\qp$ et $D \in \Phi \Gamma^{\text{ét}}(\Robba_L)$ est de Rham, on a \[ \epsilon_{L}(D) = \epsilon_{L}^{\rm dR}(D). \]
\end{enumerate}
\end{conjecture}

\subsection{Construction de l'isomorphisme} \label{constructionisom}

Dans \cite{Nakamura2}, Nakamura construit, pour toute paire $(A, D)$ comme ci-dessus, un candidat pour l'isomorphisme $\epsilon_{A}: \mathbf{1}_A \xrightarrow{\sim} \Delta_A(D)$ et il démontre le résultat suivant:

\begin{theorem} [{\cite[Th. 3.1]{Nakamura2}}]
Les isomorphismes $\epsilon_{A}(D)$, pour $(A, D)$ tel que $D$ est de rang $\leq 2$, satisfont les propriétés $(1), (2)$ et $(3)$ et $(4)$ de la conjecture. De plus, si $(A, D) = (L, D)$ est tel que $D$ est de Rham et soit triangulin soit non-triangulin à poids de Hodge-Tate $k_1 \leq 0$ et $k_2 \geq 1$, alors $\epsilon_{L}(D) = \epsilon_{L}^{\rm dR}(D)$.
\end{theorem}

Rappelons brièvement la construction de l'isomorphisme (cf. \cite[\S 3.2.2]{Nakamura2} pour les détails). Soit $D \in \Phi\Gamma^{\text{ét}}(\Robba_A)$ absolument irréductible et notons $\mathbf{Dfm}(D) \in \Phi\Gamma^{\text{ét}}(\Robba^+_A(\Gamma))$ sa déformation cyclotomique (\S \ref{cohomIw}) et
\[ \Delta_{A, i}^{\rm Iw}(D) = \Delta_{\Robba^+_A(\Gamma), i}(\mathbf{Dfm}(D)), \;\;\; i \in \{ 1, 2 \}. \]
L'opérateur $1 - \varphi$ induit (\cite[Corollaire V.1.13(iii)]{ColmezPhiGamma}) un isomorphisme de $\Robba_A(\Gamma)$-modules $D^{\psi = 1} \otimes_{\Robba_A^+(\Gamma)} \Robba_A(\Gamma) \xrightarrow{\sim} D^{\psi = 0}$ et donc un isomorphisme
\[ \Delta_{A, 1}^{\rm Iw}(D) \otimes_{\Robba^+_A(\Gamma)} \Robba_A(\Gamma) \xrightarrow{\sim} \mathrm{Det}_{\Robba_A(\Gamma)}(D^{\psi = 0})^{-1}. \]
Comme $\mathscr{L}_{\Robba^+_A(\Gamma)}(\mathbf{Dfm}(D)) = \mathscr{L}_A(D) \otimes_A \Robba^+_A(\Gamma)$, l'isomorphisme ci-dessus induit
\[ \Delta_{A}^{\rm Iw}(D) \otimes_{\Robba^+(\Gamma)} \Robba_A(\Gamma) \xrightarrow{\sim} (\mathrm{Det}_{\Robba_A(\Gamma)}(D^{\psi = 0}) \otimes_A \mathscr{L}_A(D)^\vee )^{-1}. \]
Enfin, l'accouplement d'Iwasawa (cf. \S \ref{accIw}) nous permet de trivialiser ce dernier module en définissant un isomorphisme
\[ \mathrm{Det}_{\Robba_A(\Gamma)}(D^{\psi = 0}) \otimes_A \mathscr{L}_A(D)^\vee \xrightarrow{\sim} \Robba_A(\Gamma) \]
par $(x \wedge y) \otimes z^\vee \mapsto [\sigma_{-1}] \langle w_D(x) \otimes z^\vee \otimes e_1, y \rangle_{\rm Iw}.$ Ceci induit donc
\[ \eta_{A}(D): \mathbf{1}_{\Robba_A(\Gamma)} \xrightarrow{\sim} \Delta_A^{\rm Iw}(D) \otimes_{\Robba^+_A(\Gamma)} \Robba_A(\Gamma). \]

On montre (cf. \cite[Prop. 3.4]{Nakamura2}) que l'isomorphisme $\eta_{A}(D)$ descend en un isomorphisme
\[ \epsilon^{\rm Iw}_{A}(D): \mathbf{1}_{\Robba_A^+(\Gamma)} \xrightarrow{\sim} \Delta_A^{\rm Iw}(D) \]
et, pour $\delta: \zpe \to A^\times$ un caractère localement analytique, on définit
\[ \epsilon_{A}(D(\delta)) = \epsilon^{\rm Iw}_{A}(D(\delta)) \otimes_{\Robba^+(\Gamma), f_\delta} 1_A: \mathbf{1}_A \xrightarrow{\sim} \Delta_A^{\rm Iw}(D) \otimes_{\Robba^+_A(\Gamma), f_\delta} A \xrightarrow{\sim} \Delta_A(D(\delta)), \]
où on note $f_\delta: \Robba^+_A(\Gamma) \to A$ le morphisme induit par $[\gamma] \mapsto \delta(\gamma)$, $\gamma \in \Gamma$.

Si on déroule les définitions ci-dessus (cf. aussi \cite[Rem. 3.8]{Nakamura2}), on voit que la trivialisation du module $\Delta_A^{\rm Iw}(D) = \mathrm{Det}(D^{\psi = 1}) \otimes_A \mathscr{L}_A(D)^\vee \xrightarrow{\sim} \mathrm{Det}_{\Robba_A(\Gamma)}(D^{\psi = 0}) \otimes_A \mathscr{L}_A(D)^\vee$ construite pour définir l'isomorphisme $\epsilon_{A}(D(\delta))$ est induite par
\[ (x \wedge y) \otimes z^\vee \mapsto \int_\Gamma \delta \cdot \mu_{[\sigma_{-1}] \langle w_D((1 - \varphi)x \otimes z^\vee), (1 - \varphi) y \rangle_{\rm Iw}}. \]

\subsection{Interpolation}

%Reprenons les notations de \S \ref{eqeq} et soient $D \in \Phi\Gamma^{\text{ét}}(\Robba)$ de dimension $2$, de Rham non-triangulin à poids de Hodge-Tate $0$ et $k \geq 0$, $\mathscr{L} = \mathrm{Fil}^0 \, \DdR(D) \subseteq \DdR(D)$ sa filtration de Hodge et $\Delta = \Nrig(D) \in \Phi\Gamma(\Robba)$ qui est de Rham à poids de Hodge-Tate nuls de sorte que $D = \Delta_{k, \mathscr{L}}$.

Soit $D \in \Phi\Gamma^{\text{ét}}(\Robba)$ de dimension $2$, de Rham non-triangulin à poids de Hodge-Tate $0$ et $k \geq 0$. Pour $j \in \Z$, notons \footnote{Par un petit abus de notation, on note par $\langle \;,\; \rangle_{\rm dif}$ l'accouplement pour tous le différents tordus de $D$.}
\[ \langle \;,\; \rangle_{\rm dif} \colon \DdR(\check{D}(-j)) \times \DdR(D(j)) \to L \]
l'accouplement naturel induit par restriction de l'accouplement $\langle \;,\; \rangle_{\rm dif}$ entre $\Ddif(\check{D}(-j))$ et $\Ddif(D(j))$ défini par $\langle x, y \rangle_{\rm dif} := T_L(\mathrm{r\acute{\mathrm{e}}s}_0(\langle x, y \rangle))$ pour $n \gg 0$, $x \in \Ddif(\check{D}(-j)) = \mathrm{Hom}(\Ddif(D(j)), L_\infty((t)) dt), y \in \Ddif(D(j))$ (cf. \cite[\S VI.3.4]{ColmezPhiGamma}). Notons que l'accouplement ci-haut n'est rient d'autre que l'accouplement $\{ \; , \; \}_{\rm dR}$ défini dans \S \ref{appsexp}. Rappelons que l'on a une identification entre $D(j) \otimes \omega_{D(j)}^\vee$ et $\check{D}(-j)$, envoyant $x \otimes e_{\omega_{D(j)}^\vee}$ sur la forme linéaire $y \mapsto (x \wedge y) \otimes e_{\omega_{D(j)}^\vee}$. On en déduit des isomorphismes entre $\Ddif(D(j)) \otimes \omega_{D(j)}^\vee$ et $\Ddif(\check{D}(-j))$ envoyant $x \otimes e_{\omega_{D(j)}^\vee}$ sur la forme linéaire $y \mapsto (x \wedge y) \otimes e_{D(j)}^\vee dt$.

Rappelons que l'on a fixé une base $\{ f_1, f_2 \}$ de $\DdR(D)$, ce qui permet de fixer des bases $\{ f_1 \otimes \mathbf{e}^{\rm dR}_{j}, f_2 \otimes \mathbf{e}^{\rm dR}_{j} \}$ des modules $\DdR(D(j)) = \DdR(D) \otimes L \cdot \mathbf{e}^{\rm dR}_j$. Ce choix induit un produit scalaire
\[ \langle \;,\; \rangle_{\rm dR} \colon \DdR(D(j)) \times \DdR(D(j)) \to L \]
défini par la formule $x = \sum_{i = 1, 2} \langle x, f_i \otimes \mathbf{e}^{\rm dR}_j \rangle \, (f_i \otimes \mathbf{e}^{\rm dR}_j)$. Rappelons aussi que $\Omega$ est défini par la formule $ f_1 \wedge f_2 =  \Omega^{-1} t^{-k} e_{\omega_{D}} \cdot dt$.

\begin{lemme} \label{lemmeacc} Sous les identifications faites ci-dessus, on a, pour $x \in \DdR(\check{D}(-j))$, \[ \langle x, f_{i} \otimes \mathbf{e}^{\rm dR}_j \rangle_{\rm dif} = (-1)^{i} \langle x \otimes \mathbf{e}^{\rm dR, \vee}_{-j, \omega_D^\vee}, f_i \rangle_{\rm dR}. \]
\end{lemme}

\begin{proof}
En écrivant $x = x' \otimes e_{\omega_{D(j)}^\vee}$, $x' \in \Ddif(D)$, on a
\[ \langle x, f_{i} \otimes \mathbf{e}^{\rm dR}_j \rangle_{\rm dif} = T_L(\mathrm{r\acute{\mathrm{e}}s}_0( (x' \wedge (f_i \otimes \mathrm{e}_j^{\rm dR})) \otimes e_{D(j)}^\vee dt). \]
Or, \[ (x' \wedge (f_i \otimes \mathrm{e}_j^{\rm dR})) = \langle x', f_{3 - 1} \otimes \mathbf{e}^{\rm dR}_j \rangle_{\rm dR}  ( (f_{3-i} \otimes \mathbf{e}^{\rm dR}_j) \wedge (f_i \otimes \mathbf{e}^{\rm dR}_i)), \] et, comme $(f_{3 - i} \otimes \mathbf{e}^{\rm dR}_j \wedge f_i \otimes \mathbf{e}^{\rm dR}_j) = (f_{3 - i} \wedge f_i) \otimes \mathbf{e}^{\rm dR, \otimes 2}_j = (-1)^{i} \Omega^{-1} t^{-k -2j} e_{D(j)}$, on en déduit
\[ \mathrm{r\acute{\mathrm{e}}s}_0( (x' \wedge (f_i \otimes \mathrm{e}_j^{\rm dR})) \otimes e_{D(j)}^\vee dt = (-1)^i \langle x \otimes \mathbf{e}_{\omega_{D(j)}^\vee}^{\rm dR, \vee}, f_i \otimes \mathbf{e}_j^{\rm dR} \rangle_{\rm dR} = (-1)^i \langle x \otimes \mathbf{e}_{-j, \omega_{D}^\vee}^{\rm dR, \vee}, f_i \rangle_{\rm dR} , \] ce qui permet de conclure.
\end{proof}

%\begin{proof}
%Observons d'abord que $\langle x, f_{i} \otimes \mathbf{e}^{\rm dR}_j \rangle_{\rm dif} = \langle x \otimes \mathbf{e}^{\rm dR}_j, f_{i} \rangle_{\rm dif}$. Notons \footnote{On voit tous ces éléments dans $D_{\rm dif} = \Delta_{\rm dif}$ et $\check{D}_{\rm dif} = \Delta_{\rm dif} \otimes \omega_{D}^\vee$. } $y = x \otimes \mathbf{e}^{\rm dR}_j \in \DdR(\check{\Delta})$ et $y' = t^{1 - k} (y \otimes e_{\omega_{D}}) \in \Ddif(D)$ de sorte que $y = t^{k - 1} y' \otimes e_{\omega_D}^\vee$. D'après \cite[Lem. VI.4.16]{ColmezPhiGamma}, l'application $x \mapsto t^{k - 1} x \otimes e_{\omega_{D}}^\vee$ induit un isomorphisme de $\Delta_{\rm dif}$ sur son $L_\infty[[t]] \cdot dt$-réseau dual dans $\check{D}_{\rm dif}$. On a alors
%\begin{eqnarray*} 
%\langle y, f_i \rangle_{\rm dif} &=& \langle t^{k - 1} y' \otimes e_{\omega_D}^\vee, f_i \rangle_{\rm dif} \\
%&=& \mathrm{r\acute{\mathrm{e}}s}_0(t^{k - 1} (y' \wedge f_i) \otimes e_{\omega_D}^\vee).
%\end{eqnarray*}
%Or, \[ y' \wedge f_i = \langle y', f_{3- i} \rangle_{\rm dR} \cdot (f_{3 - i} \wedge f_{i}) = (-1)^{i-1} \Omega^{-1} t^{-k} \cdot \langle y', f_{3- i} \rangle_{\rm dR} \cdot e_{\omega_D} \cdot dt. \] On en déduit
%\[ \langle x, f_i \otimes \mathrm{e}^{\rm dR}_j \rangle_{\rm dif} = (-1)^{i-1} \langle \Omega^{-1} t^{1 - k} x \otimes \mathbf{e}^{\rm dR, \vee}_{-j} \otimes e_{\omega_D^\vee}^\vee, f_i \rangle_{\rm dR} = (-1)^{i-1} \cdot \langle x \otimes \mathbf{e}_{-j, \omega_D^\vee}^{\rm dR, \vee}, f_i \rangle_{\rm dR}, \] ce qui permet de conclure.
%\end{proof}

\begin{theorem} \label{conjepsilon}
Soit $D \in \Phi \Gamma^{\text{ét}}(\Robba)$ de rang $2$, de Rham non-triangulin. Alors \[ \epsilon_{L}(D) = \epsilon_{L}^{\rm dR}(D). \]
\end{theorem}

\begin{proof}
Soit $D \in \Phi\Gamma^{\text{ét}}(\Robba)$ de rang $2$, de Rham non-triangulin à poids de Hodge-Tate $0$ et $k \geq 0$. Pour montrer le théorème, il suffit de montrer que les isomorphismes epsilon coïncident sur $D(j)$ pour tout $j \in \Z$. La preuve du théorème consiste à expliciter les deux isomorphismes en question pour se ramener à l'équation fonctionnelle des théorèmes \ref{eqfonct1} et \ref{eqfonct1b}.

La suite exacte définissant le morphisme $\theta_L(D)$ acquiert, dans le cas non-triangulin, la forme suivante: 
\begin{equation} \label{suiteexatedR}
0 \to \DdR(D(j)) / \mathrm{Fil}^0 \, \DdR(D(j)) \xrightarrow{\exp} H^1_{\varphi, \gamma}(D(j)) \xrightarrow{\exp^*} \mathrm{Fil}^0 \, \DdR(D(j)) \to 0.
\end{equation}
De plus, on sait que $H^0_{\varphi, \gamma}(D(j)) = H^2_{\varphi, \gamma}(D(j)) = 0$ et donc $\mathrm{Det}(\mathscr{C}^\bullet_{\varphi, \gamma}(D(j)) = \mathrm{Det}(H^1_{\varphi, \gamma}(D(j))^\vee$. On a trois situations possibles:
\begin{enumerate}
\item Soit $\mathrm{Fil}^0 \, \DdR(D(j)) = 0$ (le cas de poids de Hodge-Tate strictement positifs).
\item Soit $\dim_L \mathrm{Fil}^0 \, \DdR(D(j)) = 1$ (le cas de poids de Hodge-Tate $k_1 \leq 0$ et $k_2 > 0$).
\item Soit $\mathrm{Fil}^0 \, \DdR(D(j)) = \DdR(D(j))$ (le cas de poids de Hodge-Tate négatifs).
\end{enumerate}
Comme on a remarqué, le deuxième cas a été traité par Nakamura, et les deux autres cas sont en dualité (condition $(4)$ de la conjecture). Il suffit donc de montrer le résultat dans le premier cas. Les deux lemmes suivants explicitent les deux isomorphismes en jeu.

\begin{lemme}
Le morphisme $\beta_{L}^{\rm dR}(D(j)) \colon \mathrm{Det}( H^1_{\varphi, \gamma}(D(j))) \to \mathscr{L}_L(D(j))$ déduit de $\epsilon^{\rm dR}_{L}(D(j))$ satisfait \[ \beta^{\rm dR}_{L}(D(j)) \left( (\int_\Gamma \chi^j \cdot \mu_z) \wedge \exp (f_i \otimes \mathbf{e}^{\rm dR}_j) \right) = (-1)^{i} C_1 \, e_{D(j)}, \] où \[ C_1 = \Gamma_L(D(j)) \, \Omega^{-1} \, \epsilon(\D_{\rm pst}(D(j))) \, \langle \exp^{-1}(\int_\Gamma \chi^j \cdot \mu_z) \otimes \mathbf{e}^{\rm dR, \vee}_j, f_{3 - i} \rangle. \]
\end{lemme}

\begin{proof}
Notons $\theta_1'(D(j)) \colon \mathrm{Det}(H^1_{\varphi, \Gamma}(D(j))) \to \mathrm{Det}(\DdR(D(j)))$ le morphisme qui s'en déduit en prenant le déterminant de la suite exacte \eqref{suiteexatedR}. On a alors
\begin{eqnarray*}
\theta_1'(D(j))\left( (\int_\Gamma \chi^j \cdot \mu_z) \wedge \exp(f_i \otimes \mathbf{e}^{\rm dR}_j) \right) &=& \exp^{-1}(\int_\Gamma \chi^j \cdot \mu_z) \wedge (f_i \otimes \mathbf{e}^{\rm dR}_j) \\
&=& \langle \exp^{-1}(\int_\Gamma \chi^j \cdot \mu_z), f_{3 - i} \otimes \mathbf{e}^{\rm dR}_j \rangle \cdot (f_{3 - i} \otimes \mathbf{e}^{\rm dR}_j \wedge f_i \otimes \mathbf{e}^{\rm dR}_j). \\
\end{eqnarray*}
Comme $(f_{3 - i} \otimes \mathbf{e}^{\rm dR}_j \wedge f_i \otimes \mathbf{e}^{\rm dR}_j) = (-1)^{i} \Omega^{-1} t^{-k -2j} e_{D(j)}$, la dernière expression est égale à
\[ (-1)^{i} \Omega^{-1}  \, t^{-k - 2j} \, \langle \exp^{-1}(\int_\Gamma \chi^j \cdot \mu_z) \otimes \mathbf{e}^{\rm dR, \vee}_j, f_{3 - i} \rangle \, e_{D(j)}. \]

On conclut à partir de la formule $\beta^{\rm dR}_{L} = \theta_2 \circ \theta_1 = \Gamma_L \cdot \theta_2 \circ \theta_1'$ que $\beta^{\rm dR}_{L}(D(j) \left( (\int_\Gamma \chi^j \cdot \mu_z) \wedge (f_i \otimes \mathbf{e}^{\rm dR}_j) \right)$ est donné par
\[ (-1)^{i} \Gamma_L(D(j)) \, \Omega^{-1} \, \epsilon(\D_{\rm pst}(D(j))) \, \langle \exp^{-1}(\int_\Gamma \chi^j \cdot \mu_z) \otimes \mathbf{e}^{\rm dR, \vee}_j, f_{3 - i} \rangle \, e_{D(j)}, \] ce qui finit la preuve du lemme.
\end{proof}

\begin{lemme}
Le morphisme $\beta_{L}(D(j)) \colon \mathrm{Det}( H^1_{\varphi, \gamma}(D(j))) \to \mathscr{L}_L(D(j))$ déduit de $\epsilon_{L}(D(j))$ satisfait \[ \beta_{L}(D(j)) \left( (\int_\Gamma \chi^j \cdot \mu_z) \wedge \exp (f_i \otimes \mathbf{e}^{\rm dR}_j) \right) = (-1)^{i}  C_2 \, e_{D(j)}.\] où
\[ C_2 = (-1)^j \langle \exp^*(\int_\Gamma \chi^{-j} \cdot \mu_{\check{z}}) \otimes \mathbf{e}^{\rm dR, \vee}_{-j, \omega_D^\vee}, f_{3 - i} \rangle_{\rm dR}. \]
\end{lemme}

\begin{proof}
Soit $z' \in D^{\psi = 1}$ tel que $\int_\Gamma \chi^j \cdot \mu_{z'} = \exp( f_i \otimes \mathbf{e}^{\rm dR}_j)$. On a, par définition de $\beta_{L}(D(j))$ (cf. la dernière formule de \S \ref{constructionisom}),
\begin{eqnarray*}
\beta_{L}(D(j)) \left( (\int_\Gamma \chi^j \cdot \mu_z) \wedge \exp (f_i \otimes \mathbf{e}^{\rm dR}_j) \right) &=& \int_\Gamma \chi^j \cdot \langle \sigma_{-1}((1 - \varphi) \check{z}), (1 - \varphi)  z' \rangle_{\rm Iw} \, e_{D(j)} \\
&=& \langle \exp^*(\int_\Gamma \chi^{-j} \cdot \mu_{\sigma_{-1}(\check{z})}), \exp^{-1}(\int_\Gamma \chi^j \cdot \mu_{z'}) \rangle_{\rm dif} \, e_{D(j)} \\ 
&=& (-1)^j \langle \exp^*(\int_\Gamma \chi^{-j} \cdot \mu_{\check{z}}), f_i \otimes \mathbf{e}^{\rm dR}_j \rangle_{\rm dif} \, e_{D(j)}, \\ 
\end{eqnarray*}
où la deuxième égalité suit du lemme \ref{lemme26}, et la dernière suit par définition de l'action de $\Lambda$ sur la cohomologie d'Iwasawa. Par le lemme \ref{lemmeacc} ci-dessus, cette dernière expression coïncide avec
\[ (-1)^{i} (-1)^j \langle \exp^*(\int_\Gamma \chi^{-j} \cdot \mu_{\check{z}}) \otimes \mathbf{e}^{\rm dR, \vee}_{-j, \omega_D^\vee}, f_{3 - i} \rangle_{\rm dR} \, e_{D(j)}, \]
ce qui permet de conclure.
\end{proof}

Revenons à la preuve du théorème. Les poids de Hodge-Tate de la représentation $D(j)$ sont $j \geq 1$ et $k+j \geq 1$ et on a \[ \Gamma(D(j)) = \Gamma^*(j)^{-1} \Gamma^*(k+j)^{-1} = \frac{1}{(j+k-1)! (j-1)!} = (-1)^{j-1} \frac{\Gamma^*(-j + 1)}{\Gamma^*(j + k)}. \] Enfin, la compatibilité locale-globale dans la correspondance de Langlands $p$-adique pour $\mathrm{GL}_2(\qp)$ (\cite{Emerton2}) montre que \[ \epsilon(\D_{\rm pst}(D(j))) = p^{-c(\pi) j} \epsilon(\pi) = \epsilon(\pi \otimes | \cdot |^j), \] où $\pi = \mathrm{LL}_p(\Nrig(D))$ est comme dans \S \ref{Changement}. Le théorème est donc équivalent à l'identité
\[ \exp^*(\int_\Gamma \chi^{-j} \cdot \mu_{\check{z}}) \otimes \mathbf{e}^{\rm dR, \vee}_{-j, \omega_D^\vee} = - \frac{\Gamma^*(-j + 1)}{\Gamma^*(j + k)}  \, \Omega^{-1} \, \epsilon(\pi \otimes | \cdot |^j) \, \exp^{-1}(\int_\Gamma \chi^j \cdot \mu_z) \otimes \mathbf{e}^{\rm dR, \vee}_j. \] \normalsize
On conclut en utilisant le théorème \ref{eqfonct1b} (voir aussi la remarque après sa démonstration) si $k \neq 0$, ou le théorème \ref{eqfonct1} si $k = 0$.
\end{proof}

%\newpage
%\nocite{*}
\bibliographystyle{acm}
\bibliography{bibliography}

\end{document}